\documentclass[11pt]{amsart}
\usepackage[utf8]{inputenc}
\usepackage[margin=1in]{geometry}
\usepackage[pdftex]{color}
\usepackage{amsthm}
\usepackage{amssymb}
\usepackage{bbm}
\usepackage{amsmath}
\usepackage{amsfonts}
\usepackage{mathtools}
\usepackage{breqn}
\usepackage{indentfirst}
\usepackage{enumitem}
\usepackage{hyperref}
\usepackage[normalem]{ulem}

\DeclarePairedDelimiter\ceil{\lceil}{\rceil}

\newcommand{\ti}[1]{\tilde{#1}}
\newcommand{\und}[1]{\underline{#1}}
\newcommand{\ov}[1]{\overline{#1}}
\newcommand{\R}{\mathbbm{R}}
\newcommand{\C}{\mathbbm{C}}
\newcommand{\T}{\mathbbm{T}}
\newcommand{\Z}{\mathbbm{Z}}
\newcommand{\D}{\mathbbm{D}}

\newcommand{\bp}{\mathbb{P}}
\newcommand{\ve}{\varepsilon}
\newcommand{\om}{\omega}

\newcommand{\Om}{\Omega}
\newcommand{\CC}{\mathcal{C}}
\newcommand{\LL}{\mathcal{L}}

\newcommand{\A}{\mathbb{A}}
\newcommand{\ro}{\rho}
\newcommand{\te}{\theta}
\newcommand{\tb}{\underline{\theta}}
\newcommand{\ttb}{\tilde{\underline{\theta}}}
\newcommand{\si}{\mathcal{S}}
\newcommand{\uom}{\underline{\omega}}
\newcommand{\ute}{\underline{\theta}}
\newcommand{\diam}{\text{diam}}

\newcommand{\rV}{\right\rVert}
\newcommand{\lV}{\left\lVert}
\newcommand{\rv}{\right\rvert}
\newcommand{\lv}{\left\lvert}

\theoremstyle{plain}
\newtheorem{theorem}{Theorem}[section]

\theoremstyle{plain}
\newtheorem{lemma}[theorem]{Lemma}

\theoremstyle{plain}
\newtheorem{proposition}[theorem]{Proposition}

\theoremstyle{plain}
\newtheorem*{claim}{Claim}

\theoremstyle{plain}
\newtheorem*{theorem*}{Theorem}

\theoremstyle{remark}
\newtheorem{remark}[theorem]{Remark}

\theoremstyle{definition}
\newtheorem{definition}[theorem]{Definition}

\title[Stable intersections of conformal Cantor sets with large Hausdorff dimensions]{Stable intersections of conformal regular Cantor sets with large Hausdorff dimensions}

\author{Hugo Araújo}
\thanks{The paper was written while the first author was at IMPA and PUC, the second author at IMPA and the third author at UFRJ}
\thanks{This work was supported by CNPq and CAPES}

\author{Carlos Gustavo Moreira}

\author{Alex Zamudio Espinosa}

\date{}

\begin{document}

\maketitle

\begin{abstract}
In this paper we prove that among pairs $K,\,K' \subset \C$ of conformal dynamically defined Cantor sets with sum of Hausdorff dimensions $HD(K)+HD(K')>2$, there is an open and dense subset of such pairs verifying $\text{int}(K-K')\neq \emptyset$. This is motivated by the work \cite{MY}, where Moreira and Yoccoz proved a similar statement for dynamically defined Cantor sets in the real line. Here we adapt their argument to the context of conformal Cantor sets in the complex plane, this requires the introduction of several new concepts and a more detailed analysis in some parts of the argument.
\end{abstract}

\section{Introduction}

The aim of this paper is to prove a complex version of the main theorem in \cite{MY}. We will prove that among pairs $K,\,K' \subset \C$ of conformal dynamically defined Cantor sets with sum of Hausdorff dimensions $HD(K)+HD(K')>2$, there is an open and dense subset of such pairs verifying $\text{int}(K-K')\neq \emptyset$. The analogous statement for dynamically defined Cantor sets in the real line was proved by Moreira and Yoccoz in \cite{MY}. Here we adapt their argument to the context of conformal Cantor sets, this requires the introduction of several new concepts and a more detailed analysis in some parts of the argument. The main new difficulties come from the fact that smooth real maps are naturally conformal, in the sense that their derivative preserves angles, and this is not true for maps in dimension two. %In fact, in dimension two conformality at all points impose a strong restriction on the maps, they have to be holomorphic. Supposing that our maps 
A rough version of our main theorem is the following.

\begin{theorem}\label{thm:main0}
There exists an open and dense subset $\mathcal{V}$ of the space of pairs $(K_1,K_2)$ of conformal Cantor sets in the complex plane such that $HD(K_1)+HD(K_2) > 2$ satisfying the following. If $(K,K') \in \mathcal{V}$, then\
\[
\mathcal{I}_s(K,K') = \{\lambda\in \C: (K+\lambda,K')\text{ has stable intersection}\}
\]
is dense in
\[
\mathcal{I}=\{\lambda\in \C: (K+\lambda)\cap K'\neq \emptyset\}.
\]
In particular, $\text{int}(K-K')\neq \emptyset$.

%There is an open and dense set $U \subset \Omega_{\Sigma}\times \Omega_{\Sigma'}$ of pairs of conformal Cantor sets $(K,K')$, such that if $(K,K')\in U$ then\[\mathcal{I}_s(K,K') = \{\lambda\in \C: (K+\lambda,K')\text{ has stable intersection}\}\]is dense in\[\mathcal{I}=\{\lambda\in \C: (K+\lambda)\cap K'\neq \emptyset\}.\]In particular, $\text{int}(K-K')\neq \emptyset$.
\end{theorem}

The work of Moreira and Yoccoz was motivated by a conjecture of Palis, according to which generic pairs of dynamically defined Cantor sets in the real line $K,\, K'$ either verify that their arithmetic difference $K-K'=\{x-y:x\in K,\, y\in K'\}$ has zero Lebesgue measure or non-empty interior. Palis conjecture emerged from his work with Takens (\cite{PT}, \cite{PT1}), where, in their study of homoclinic bifurcations for surface diffeomorphisms, they used the crucial fact that if $HD(K)+HD(K')<1$, then $K-K'$ has zero Lebesgue measure. Looking for a converse, Palis proposed its conjecture.

%In \cite{M} Moreira introduced the concept of stable intersection of Cantor sets in the real line. A pair of dynamically defined Cantor sets $K,\, K'\subset \R$ are said to have stable intersection if $\ti{K}\cap \ti{K}'\neq \emptyset$ for all $\ti{K},\, \ti{K}'$ close enough to $K,\, K'$, respectively.

The study of homoclinic bifurcations has proved to be fruitful in the understanding of dynamics for surface diffeomorphisms. Complicated dynamical phenomena arise from them. For example, arbitrarily close to any diffeomorphism exhibiting a generic homoclinic tangency, there are open regions in which any diffeomorphism belonging to a residual set has an infinite number of sinks- this is the so called Newhouse phenomenon. Looking for analogous results and using similar ideas to those of Newhouse, Buzzard \cite{B} proved the existence of an open set of automorphisms of $\C^2$ with stable homoclinic tangencies. Buzzard's strategy was very similar to the work of Newhouse [9], constructing a “very thick” horseshoe, such that the Cantor sets, this time living in the complex plane, associated to it would also be “very thick”. However, the concept of thickness does not have a simple extension to this complex setting and so the argument to guarantee intersections between the Cantor sets after a small pertubation is different.

Furthermore, Moreira and Yoccoz were able to use their solution to Palis conjecture in the study of homoclinic bifurcations for surface diffeomorphisms (see \cite{MY2}). They proved that given a surface diffeomorphism $F$ with a homoclinic quadratic tangency associated to a horseshoe with dimension larger than one, the set of diffeomorphisms close to $F$ presenting a stable tangency has positive density at $F$. 

The authors of this paper have been trying to push all this theory to the context of Cantor sets in the complex plane and apply it to the study of homoclinic bifurcations of automorphisms of $\C^2$. In \cite{AM}, Araújo and Moreira introduced the concept of conformal Cantor sets and also extended the concept of recurrent compact sets to this context, which serves as a tool to obtain stable intersection. More importantly, Araújo and Moreira showed that given a complex horseshoe of an automorphism of $\C^2$, one can associate to it a conformal Cantor set. Using all this, they were able to recast Buzzard's example in terms of the theory of conformal Cantor sets.

In the paper \cite{MZ}, Moreira and Zamudio proved a multidimensional version of the scale recurrence lemma for conformal Cantor sets. The real version of this lemma was a key step in the solution, by Moreira and Yoccoz, of Palis conjecture. Moreover, Moreira and Zamudio used the multidimensional conformal version of the scale recurrence lemma to prove a dimension formula relating the Hausdorff dimension of the image of a product of conformal Cantor sets $h(K_1\times ... \times K_n)$, where $h$ is a $C^1$ function, and the sum of the Hausdorff dimensions $HD(K_1)+...+HD(K_n)$.

In this paper we intend to use the results in \cite{AM}, \cite{MZ} and obtain a version of Palis conjecture for conformal Cantor sets. In a future paper we plan to use the concepts and ideas laid down in this paper to study homoclinic bifurcations for automorphisms of $\C^2$. We plan to obtain results analogous to the ones in \cite{MY2}. We expect that conformal regular Cantor sets in $\C$ will play a role in the study of homoclinic bifurcations of automorphisms of $\C^2$, similar to regular Cantor sets in $\R$ in the study of homoclinic bifurcations of surface diffeomorphisms.

The paper is organized as follows. In section \ref{sec:def} we fix the notation, present basic concepts and results which will be used later. Most of the proofs are omitted and references are given to the works \cite{AM}, \cite{MZ}, \cite{Z} and \cite{PT95}. We also reduce our work to the proof of theorem \ref{thm:main2}; its proof is long and contained in the remaining sections. In section \ref{sec:randpert} we define a random family of Cantor sets and analyze some of its geometrical properties; it is from this family that we will find the pair of Cantor sets in the conclusion of \ref{thm:main2}. In section \ref{sec:proof} we prove theorem \ref{thm:main2}, assuming propositions \ref{prop:probes} and \ref{prop:bigL}. The proof of these propositions is postponed to the last two sections (\ref{sec:bigL} and \ref{sec:proofprobes}) of this work.

\section{Notation and preliminaries}\label{sec:def}

In this section we give the definitions of the objects appearing in this paper and recall some results regarding them that have appeared on previous works \cite{AM}, \cite{MZ}, \cite{Z}. We also present new results and reduce the proof of the main theorem to the proof of theorem \ref{thm:main2}.

\begin{subsection}{Notations}
Here we introduce some of the notations we will use along the text:

\begin{itemize}
\item We will work with the space of complex numbers $\C$. We identify it with $\R^2$ in the usual way. In this space we will use the Euclidean metric, given by the norm $|(x,y)|=\sqrt{x^2+y^2}$.
\item We will also work with the space of non-zero complex numbers $\C^*$. We will sometimes identify $\C^*$ with $J=\R\times \T$, where $\T=\R/(2\pi \Z)$, through the map $(t,v)\to e^{t+iv}$. Note that $J$ has the structure of a commutative group. We will endow $J$ with the metric coming from the inclusion $J=\C^*\subset \C$.
\item Given a linear map $A:\R^2\to \R^2$ we denote its norm by $|A|$, its minimum norm by $m(A)$. They are given by
\[|A|=\sup_{v\neq 0} \frac{|Av|}{|v|},\,\, m(A)=\inf_{v\neq 0} \frac{|Av|}{|v|}.\]
We will say that $A$ is conformal if $|A|=m(A)$.
\item We will use $Id$ to denote the identity matrix, sometimes we will use the same symbol for the identity function. Each case will be clear from the context.
\item Let $(X,d)$ be a metric space and $A\subset X$. We will use the notation $V_{\delta}(A)$ for the open $\delta$ neighborhood around $A$, i.e. $V_{\delta}(A)=\{x\in X:\,d(x,A)< \delta\}$. For a point $x\in X$ and a positive real number $r$ we use the notation $B(x,r)=V_r(\{x\})$.
\item We will have to deal with many inequalities and several parameters. In order to reduce the number of constants introduced along the text we will use the following notations: Given expressions $f(x)$ and $g(x)$ depending in the parameter $x$, when we write $f(x)\lesssim g(x)$ this means that there is a constant $C>0$ such that $f(x)\leq Cg(x)$ for all $x$. The constant $C$ can only depend on other constants when those have already been fixed. This ensures that we will not get any contradiction between the different constants that will appear. $f\gtrsim g$ will mean $g\lesssim f$, $f\approx g$ will mean $f\lesssim g$ and $g\gtrsim f$ both hold.
\item Derivative of a $C^l$ function, we use two notations $D^jf(x)$ or $D^j_x f$, both mean derivative of order $j$ of $f$ at the point $x$. They are $j$-linear functions.
\item The uniform norm of functions $f: X \to Y$, where $X$ and $Y$ are subsets of normed vector spaces, will be denoted by $\lV f \rV \coloneqq \sup_{x \in X} \lv f(x) \rv$. In many occasions, $f$ will be the derivative of order $j$ of a $C^l $ function.
\item Given a real number $m$ we will denote by $[m]$ its integer part, this is $[m]=\sup\{r\in \Z:r\leq m\}$.
\end{itemize}
\end{subsection}

\begin{subsection}{The space of conformal regular Cantor Sets}

We begin by the very concept of \emph{conformal regular Cantor set}. First we remember that a $C^m$ \emph{regular Cantor set}, also called \emph{dynamically defined Cantor set}, is given by the following data.

\begin{itemize}

\item A finite set $\mathbb{A}$ of letters and a set $B \subset \mathbb{A} \times \mathbb{A} $ of admissible pairs.

\item For each $a\in \mathbb{A}$ a compact connected set $G(a)\subset \mathbb{C}$.
 
\item A $C^{m}$ map $g: V \to \mathbb{C}$, for $m>1$, defined on an open neighbourhood $V$ of $\bigsqcup_{a\in \mathbb{A}}G(a)$.
\end{itemize}

These data must verify the following assumptions:

\begin{itemize}
    
\item The sets $G(a)$, $a \in \mathbb{A}$, are pairwise disjoint.

\item $(a,b)\in B $ implies $G(b) \subset g(G(a))$, otherwise $G(b) \cap g(G(a)) = \emptyset$.

\item For each $a \in \mathbb{A}$, the restriction $g|_{G(a)}$ can be extended to a $C^{m}$ diffeomorphism from an open neighborhood of $G(a)$ onto its image such that $m(Dg) > \mu$ for some constant $\mu > 1$, where $m(A) := \displaystyle{\inf_{v \neq 0}\frac{\lvert Av \rvert}{\lvert v \rvert}}$ is the minimum norm of the linear operator $A$ on $\mathbb{R}^2$.

\item The subshift $(\Sigma, \sigma)$ induced by $B$, called the type of the Cantor set, 
\[ \Sigma=\{ \und{a}= (a_0, a_1, a_2, \dots  ) \in \mathbb{A}^{\mathbb{N}}:(a_i,a_{i+1}) \in B, \forall i \geq 0\},\]

$\sigma (a_0,a_1,a_2, \dots) = (a_1,a_2,a_3, \dots)$, is topologically mixing.

\end{itemize}

Once we have all these data we can define a Cantor set (i.e. a totally disconnected, perfect compact set) on the complex plane: 
\[ K=\bigcap_{n \geq 0} g^{-n}\left(  \bigsqcup_{a \in \mathbb{A}} G(a) \right). \] 

We say that such a set is \emph{conformal} if, for all $x \in K$, the derivative of $g$ at $x$, denoted by $Dg(x) : \R^2 \to \R^2$, is a conformal linear operator. 

Notice that we always consider the degree of differentiability of the map $g$, $m$, to be a real number larger than one. This means that $g$ has derivatives up to order $[m]$ and $D^{[m]}g$ is Hölder with exponent $m-[m]$. This hypothesis, as we will precise later in this section, allows us to control the geometry of small parts of the Cantor set $K$. All Cantor sets in this paper will be conformal regular Cantor sets. 

Besides, as we mentioned on the introduction, an important family of dynamically defined sets are contained in the class of $C^{m}$ conformal regular Cantor sets. Let $G$ be an automorphism of $\C^2$ exhibiting a horseshoe $\Lambda$ and $p$ be a hyperbolic periodic point in it. Then, there is a subset $U\subset W^s(p)$, open in the topology of $W^s(p)$ as an immersed manifold,  and some $\ve > 0$ sufficiently small such that $\Lambda \cap U$ is, after some parametrization, a $C^{1+\ve}$ conformal regular Cantor set. See \textbf{Theorem A} of \cite{AM}. 

We will usually write only $K$ to represent all the data that defines a particular dynamically defined Cantor set. Of course, the compact set $K$ can be described in multiple ways as a Cantor set constructed with the objects above, but whenever we say that $K$ is a Cantor set we assume that one particular set of data as above is fixed. In this spirit, we may represent the Cantor set $K$ by the map $g$ that defines it as described above, since all the data can be inferred if we know $g$.

Notice that in our definition we did not require the pieces $G(a)$ to have non-empty interior. To circumvent this, we introduce the following sets. 

\begin{lemma}\label{lemma:star}

There is $\delta> 0$ sufficiently small such that the sets $G^*(a) \coloneqq V_\delta(G(a))$ satisfy:
\begin{enumerate}[label=(\roman*)]

 \item  $G^*(a)$ is open and connected.
 
 \item $G(a) \subset G^*(a)$ and $g|_{G^*(a)}$ can be extended to an open neighbourhood of $\ov{G^*(a)}$, such that it is a $C^{m}$ embedding (with  $C^{m}$ inverse ) from this neighbourhood to its image and $m(Dg) > \mu$.
 
 \item The sets $\ov{G^*(a)}$, $a \in \mathbb{A}$, are pairwise disjoint.
 
 \item $(a,b) \in B$ implies $\ov{G^*(b)}\subset g(G^*(a))$, and $(a,b) \notin B$ implies  $\ov{G^*(b)}\cap \ov{g(G^*(a))} = \emptyset$.

\end{enumerate}
\end{lemma}

The sets $\ov{G^*(a)}$ could substitute the pieces $G(a)$ in our definition as to make the hypothesis of open interiors be true. These changes do not enlarge the Cantor set. To see this, we introduce more notation and a previous result.

Associated to a Cantor set $K$ we define the sets 
\begin{align*}
    \Sigma^{fin} & = \{(a_0, \dots ,a_n): (a_i,a_{i+1}) \in B \ \forall i ,\, 0 \leq i < n \}, \\
\Sigma^- & = \{(\dots, a_{-n}, a_{-n+1},\dots,a_{-1},a_0): (a_{i-1},a_i) \in B \ \forall i \leq 0\}.
\end{align*} 
Given $\und{a}=(a_0, \dots, a_n)$, $\und{b}=(b_0, \dots , b_m)$, $\und{\theta}^1=(\dots,\theta^1_{-2},\theta^1_{-1},\theta^1_{0})$ and $\und{\theta}^2=(\dots,\theta^2_{-2},\theta^2_{-1},\theta^2_{0})$, we write:

\begin{itemize}
\item if $a_n=b_0$, $\und{ab}=(a_0, \dots,a_{n}, b_1, \dots, b_m)$;
\item if $\theta^1_0=a_0$, $\und{\theta^1a} = (\dots, \theta^1_{-2},\theta^1_{-1}, a_0, \dots, a_n )$
\item if $\und{\theta}^1 \neq \und{\theta}^2$ and $\theta ^1 _0 = \theta ^2_ 0$,   $\und{\theta}^1 \wedge \und{\theta}^2=(\theta_{-j}, \theta_{-j+1}, \dots , \theta_0)$, in which $\theta _{-i} =\theta^1 _{-i}= \theta^2 _{-i} $ for all $i=0, \dots, j $ and $\theta^1_{-j-1} \neq \theta^2_{-j-1}$.
\item Define the distance between $\tb^1$ and $\tb^2$ by $d(\tb^1,\tb^2)=diam(G(\tb^1\wedge \tb^2))$.
\item  if $\und{a}$ starts with $\und{b}$, we define $\und{a}/\und{b}$ as the unique finite word such that $\und{a}= \und{b}(\und{a}/\und{b})$.
%\item if $\theta^1_0=a_n$, $\und{\theta}^1 \wedge \und{a}=(a_{n-j}, \dots, a_n) $, in which $a_{n-i}=\theta^1_{-i}$ for all $i=0, \dots , j$ and $j\geq n$ or  $a_{n-j-1} \neq \theta^1_{-j-1}$. 
\end{itemize}

For $\und{a}=(a_0, a_1, \dots , a_n) \in \Sigma^{fin}$ we say that it has length $n$ and define:
$$G(\und{a})= \{x \in \bigsqcup_{a \in \mathbb{A}} G(a) , \; g^j(x) \in G(a_j), \;j=0,1,\dots, n \}$$
and the function $f_{\und{a}}: G(a_n) \to G(\und{a})$ by:

$$ f_{\und{a}} =  g|^{-1}_{G(a_0)} \circ g|^{-1}_{G(a_1)} \circ \dots \circ (g|^{-1}_{G(a_{n-1})})|_{G(a_n)} . $$

Notice that $f_{(a_i,a_{i+1})} = g|^{-1}_{G(a_{i})}$. Furthermore, we can consider the sets $G^*(\und{a})$ defined in the same way 
$$G^*(\und{a})= \{x \in \bigsqcup_{a \in \mathbb{A}} G^*(a) , \; g^j(x) \in G^*(a_j), \;j=0,1,\dots, n \}$$but using the $*$ version of the pieces and consider the function $f_{\und{a}}$ to be defined in the larger set $G^*(a_n)$ having image equal to $G^{*}(\und{a})$.

Now we have the following lemma.

\begin{lemma} \label{lemma: decay}
Let $K$ be a dynamically defined Cantor set and $G^*(\und{a})$ the sets defined above. There exists a constant $C>0$ such that:

$$diam(G^*(\und{a})) < C\mu^{-n}.$$

\end{lemma}

As a consequence of this lemma we can see that 
\[K=\bigcap_{n \geq 0}g^{-n}\left(\bigsqcup_{a \in \mathbb{A}}G^*(a)\right)\]
since $G(\und{a})\subset G^*(\und{a})$ and $\text{diam}(G^*(\und{a})) \rightarrow 0$, and so the Cantor set has not been enlarged. Another consequence is that $K$ is contained in the interior of the union of the pieces $G^*(a)$. From now on we will work with the assumption that the sets $G(a)$ have non-empty interior and that they contain $K$ in the interior of their union. We keep the definition $G^*(a) \coloneqq V_{\delta}(G(a))$ as before because it will be useful in the definition of limit geometries.

The following definition will be useful in the future.

\begin{definition}\label{def:homeo}
For every Cantor set $K$ we define the homeomorphism 
\[
    H : K  \to \Sigma 
\]
that carries each point $x \in K $ to its itinerary along the pieces $G(a)$, that is 
\[
H(x) = (a_0, a_1, \dots, a_n, \dots)
\] 
if and only if $g^i(x) \in G(a_i) $ for all $i \ge 0$.
\end{definition}

Our main result concerns generic Cantor sets. So now we fix the topology on the space of Cantor sets. We remind that any Cantor set we are considering is given by a map $g$ that is, at least, $C^{1+\ve}$ for some $\ve > 0$.  %In order to carry out perturbations in the $C^m$ topology for $m \ge 2$ we will assume that our Cantor sets are {\em non essentially real}. This hypothesis was introduced in \cite{MZ}, it requires that the Cantor set is not contained in a $C^1$ curve. We will explain it in more detail in subsection \ref{sub:perC}. 

\begin{definition}{(The space $\Om^{m}_{\Sigma}$)} \label{topo} For a fixed symbolic space $\Sigma$ and real number $m>1$ (we also allow m=$\infty$). The set of all $C^m$ conformal regular Cantor sets $K$ with the type $\Sigma$ is defined as the set of all $C^m$ conformal Cantor sets described as above whose set of data includes the alphabet $\mathbb{A}$ and the set  $B$ of admissible pairs used in the construction of $\Sigma$. We denote it by $\Om^{m}_{\Sigma}$.

\end{definition} 

The topology on $\Om^m_{\Sigma}$ is generated by a basis of neighbourhoods $U_{K,\delta} \subset \Om^{m}_{\Sigma}$ where $ K $ is any $C^{m}$ Cantor set in $ \Om^m_{\Sigma}$ and $ \delta > 0  $.  The neighborhood $U_{K,\delta}$ is the set of all $C^m$ conformal regular Cantor sets $K'$ given by $g': V' \to \C, \, V' \supset \bigsqcup_{a \in \mathbb{A}} G'(a)$ such that $G(a) \subset V_{\delta}(G'(a))$, $G'(a) \subset V_{\delta}(G(a))$ (that is, the pieces are close in the Hausdorff topology) and the restrictions of $g'$ and $g$ to $V \cap V'$  are $\delta$ close in the $C^{m}$ metric. The topology on $\Omega^{\infty}_{\Sigma}$ is the one such that a sequence of $C^{\infty}$ Cantor sets $K_n$ converges to $K$ if and only if the sequence converges to $K$ in the topology of $\Omega^{m}_{\Sigma}$ for every $m \in (1, \infty)$.% Thus, the topology in $\Omega^{\infty}_{\Sigma}$ is generated by sets of the form $U\cap \Omega^{\infty}_{\Sigma}$ such that $U\subset \Omega^{m}_{\Sigma}$ is open in $\Omega^{m}_{\Sigma}$ 

We also consider the union $\Omega_{\Sigma}:=\bigcup_{m>1} \Omega^m_{\Sigma}$, the topology in $\Omega_{\Sigma}$ is the finest topology such that the inclusions $\Omega^m_{\Sigma}\subset \Omega_{\Sigma}$ are continuous maps, the so called inductive limit topology. Thus, a set $U\subset \Omega_{\Sigma}$ is open if and only if $U\cap \Omega^m_{\Sigma}$ is open in $\Omega^m_{\Sigma}$ for all $m>1$. It is not difficult to prove that an open set $U\subset \Omega_{\Sigma}$ can be written as a union $U=\bigcup_{m>1} U_m$, where each $U_m$ is open in $\Omega^m_{\Sigma}$ and $U_m\subset U_{m'}$ if $m>m'$.

%If we are interested only on Cantor sets with regularity at least $r$, for some $r>1$, we consider only $\Om_{\Sigma}^r$, defined in the same way as above only requiring $g$ to be $C^r$. The topology is defined in the analogous way, this time considering the distance between $g$ and $g'$ in the $C^r$ metric. We also allow the case $r=\infty$. The topology on $\Omega^{\infty}_{\Sigma}$ is the one such that a sequence of $C^{\infty}$ maps $g_n$ converges to $g$ if and only if the sequence converges to $g$ in the topology of $\Omega^{r}_{\Sigma}$ for every $r \in (1, \infty)$.

%\begin{remark}
%Notice that the space $\Om_{\Sigma}$ we are considering is the same one in {\bf Definition 2.3} of \cite{AM}, with the exception of the additional hypothesis that the Cantor sets are non essentially real. As pointed previously, this change is necessary to control perturbations of $C^m$ Cantor sets for $m \ge 2$. {\color{red} In the Appendix we will discuss the validity of our results when we also consider essentially real Cantor sets in the definition of  $\Om_{\Sigma}$.}
%\end{remark}

\end{subsection}

\begin{subsection}{Limit geometries}

To study the geometry of small parts of our Cantor sets, we introduce more objects. For each $\und{a} = (a_0, \, \dots,\,a_n) \in \Sigma^{fin}$, denote by $K(\und{a})$ the set $K \cap G(\und{a})$. For each $a \in \mathbb{A}$, fix a point $c_a\in K(a)$. We will refer to these points as base points. Define $c_{\und{a}} \in K(\und{a})$ by 
\[
c_{\und{a}}  \coloneqq f_{\und{a}}(c_{a_n}).
\]
Additionally, given $\und{\theta} = ( \dots, \theta_{-n}, \dots, \theta_0 ) \in \Sigma^- $ we write $\und{\theta}_n \coloneqq (\theta_{-n}, \dots, \theta_0  )$ and $r_{{\und{\theta}}_n}:= \text{diam}(G(\und{{\theta}}_n))$.

Given $\ute \in \Sigma^-$ and $n \ge 1$, define $\Phi_{\ute_n}$ as the unique map in 
\[
Aff(\C) \coloneqq \{\alpha z + \beta,\ \alpha \in \C^*,\,\beta \in \C\}
\]
such that 
\[
\Phi_{\ute_n}(c_{\ute_n}) = \left(\Phi_{\ute_n} \circ f_{\ute_n}\right) (c_{\te_0}) = 0 \qquad \text{and} \qquad D\left(\Phi_{\ute_n} \circ f_{\ute_n}\right) (c_{\te_0}) = Id.
\]
The maps $\Phi_{\ute_n}$ act as a \emph{normalization} of small parts of the Cantor set $K$. For that purpose, we define the maps $k^{\ute}_n$ by
\[
k^{\ute}_n \coloneqq \Phi_{\ute_n} \circ f_{\ute_n}.
\]
Through them we have the first result that allows control over the sets $G(\ute_n)$. 

In what follows we consider some $m > 1$ fixed and $K$ being a Cantor set in the space $\Omega_{\Sigma}^{m}$.

\begin{proposition}\label{lemma:limgeo} (Limit Geometries) For each $\und{\theta} \in \Sigma^-$ the sequence of $C^{m}$ embeddings $k^{\und{\theta}}_n: G^*(\theta_0) \to \C $ converges in the $C^{m}$ topology to a $C^m$ embedding $ k^{\und{\theta}}: G^*(\theta_0) \to \C $. The convergence is uniform over all $\und{\theta} \in \Sigma^-$ and in a small neighbourhood of $g$ in $\Om^{m}_{\Sigma}$. 
\end{proposition}

The $ k^{\und{\theta}}: G(\theta_0) \to \C $ defined for any $\und{\theta} \in \Sigma^-$ are called the \emph{limit geometries} of $K$.

\begin{remark}\label{rmk:dgl}
Define $\Sigma^-_a = \{\und{\theta} \in \Sigma^-, \,\und{\theta}_0 = a \}$ and consider in this set the topology given by the metric $d(\und{\theta}^1,\und{\theta} ^2) = \text{diam}(G(\und{\theta}^1 \wedge \und{\theta}^2)) $. Likewise, for $m > 1 $, let $\text{Emb}_{m}(G^*(a), \C)$ be the space  of $C^{m}$ embeddings from $G^*(a)$ to $\C$ with $C^{m}$ inverse equipped with the topology given by the $C^1$ metric
\[
d(g_1,g_2)= \max\{||g_1-g_2||, ||D g_1 - D g_2|| \}.
\]
For fixed $0 < \ve < 1$ and a $C^{1+\ve}$ Cantor set $K$, the map $k: \Sigma^-_a \to \text{Emb}_{1+\ve}(G^*(a), \C), \; \und{\theta} \mapsto k^{\und{\theta}}$ is $\ve$-Hölder, if we consider the metrics described above for both spaces. In case the Cantor set $K$ is $C^m$, for $m\geq 2$, then there is a constant $C>0$ such that $d(k^{\tb^1},k^{\tb^2})\leq C d(\tb^1, \tb^2)$. The constant $C$ can be chosen uniformly in a neighborhood of the Cantor set.

Since the convergence is uniform with respect to $\tb$ and in a neighborhood of $\Om^{m}_{\Sigma}$, the limit geometries $k^{\tb}$ depend continuously in $\tb$ and the Cantor set $K$.

We also remark that the derivative $Dk^{\tb}(x)$ is conformal for all $x\in K(\theta_0)$.
\end{remark}

\begin{remark}
It is important to observe that limit geometries depend on the choice of base points. Nonetheless, different choice of base points do not change the resultant limit geometries by much, only by an affine transformation that is bounded by some constant $C$ depending on $K$. Here we mean that such transformations are given by maps $A(z) = \alpha z + \beta$, where $\lv \alpha\rv, \lv \beta \rv < C$. This bound is, as before, uniform for Cantor sets $\ti{K}$ sufficiently close to $K$. See the paragraph after Corollary 3.2 of \cite{AM}.

For reasons that will become more clear in the future, from now on we assume that for each $a \in \mathbb{A}$ the corresponding base point $c_a$ is a pre-periodic point.

\end{remark}

%\begin{remark}
%The limit geometries defined in \cite{MZ} are a little bit different, in a way that suited better the purposes of that paper. There, the value of $\lv D\Phi_{\ute_n} \rv $ is chosen to make the diameter of $k^{\ute}_n(G(\te_0))$ to be equal to $1$. Again, the end result is not that different, only by multiplication by a bounded complex number (see subsection \ref{se:scrl}). This is the case because the diameter of the sets $G(\ute_n)$ have size comparable with $\| Df_{\ute_n}(c_{\te_0})\|$ and $m\left( Df_{\ute_n}(c_{\te_0})\right)$; see corollaries of \cite{}.
%\end{remark}

%\begin{remark}
%Proposition \ref{lemma:limgeo} remains true when we consider the limit geometries defined on the slightly larger sets $G^*(\te_0)$. In both cases, the rate of convergence of $k^{\ute}_n$ to $k^\ute$ is exponential, with error controlled by $\mu^{-n}$. This control is also the case for $\ti{K}$ in a sufficiently close to $K$.
%\end{remark}

Before proceeding, we fix some more notation. For $\ute \in \Sigma^-$ and $\und{a} \in \Sigma^{fin}$ we write 
\begin{align*}
    G^{\ute}(\und{a}) \coloneqq k^{\ute}(G(\und{a})), \qquad  K^{\ute}(\und{a}) \coloneqq k^{\ute}(K(\und{a})),  \qquad c^{\ute}_{\und{a}} \coloneqq k^{\ute}(c_{\und{a}}).
\end{align*}

Furthermore, to establish stable intersections, we are going to analyse very small parts of the Cantor sets, whose size will be controlled by a real number $\rho$. This number should be regarded as a variable that we are going to assume in various instances to be very small, as to make the various estimates we are going to find in the future to fit all together. This being said, let $c_0$ be a sufficiently large constant.
%as required in the statement of the \emph{scale recurrence lemma} (lemma \ref{lem:scl}) in subsection \ref{se:scrl}

\begin{definition}
For $0 < \rho < 1$, the set $\Sigma(\rho)$ is defined as the set of words $\und{a} \in \Sigma^{fin}$ such that
\[
c_0^{-1}\rho \le \diam(G(\und{a})) \le c_0\rho.
\]
We say that the set $G(\und{a})$ has an \emph{approximate size} $\rho$.
\end{definition}

Using standard techniques (see \cite{PT} and \cite{Z}), one can prove that there is a constant $C$ depending only in the Cantor set and the parameter $c_0$ such that
\begin{equation}\label{eq:sigrho}
C^{-1} \rho^{-d}\leq \# \Sigma(\rho)\leq C \rho^{-d},\end{equation}
where $d=HD(K)$ is the Hausdorff dimension of $K$.

\begin{remark}\label{remark:onsize}
 Notice that, because the set of limit geometries represent a compact subset of \linebreak $\cup_{a \in \mathbb{A}} {Emb_{m}(G(a))}$, every piece of \emph{approximate size} $\rho$ also contains the ball $B(c_{\und{a}}, (c_0^{\prime})^{-1}\rho)$ for some $c_0^\prime > 0$ depending only on $K$. The result remains valid for perturbations $\ti{K}$ sufficiently close to $K$. Even more, by maybe enlarging $c_0^\prime$ a little bit, because of Corollaries 3.3 and 3.4 of \cite{AM}, it follows that for any $\und{a} = ( a_0,\,\dots,\,a_n) \in \Sigma^{fin}$
\begin{equation}\label{eq:c0prima}
    {(c'_0)}^{-1} \le \frac{\lv Df_{\und{a}}(c_{a_n}) \rv}{\diam(G(\und{a}))} \le c'_0.
\end{equation}
This allows us to control the approximate size of the sets $G(\und{a})$ through the derivative of the map $f_{\und{a}}$ at $c_{a_n}$.
\end{remark}

\end{subsection}

\begin{subsection}{Recurrent compact criterion for stable intersections}
Our next objects are called \emph{configurations}. They are a way of moving a Cantor set in the plane without changing its internal structure.

 \begin{definition}
 A $C^m$-\emph{configuration} of a piece $G(a)$ of a Cantor set is a $C^m$, $m > 1$, diffeomorphism
\[
h: G(a) \to U \subset \C.\]
The space of all $C^m$ configurations of a piece $G(a)$ is denoted by $\mathcal{P}^m(a)$ and we equip it with the $C^{m}$ topology. The space of all configurations is denoted by $\mathcal{P}(a) = \cup_{m > 1} \mathcal{P}^m(a)$ and we equip it with the inductive limit topology. This is, $U\subset \mathcal{P}(a)$ is open if and only $U\cap \mathcal{P}^m(a)$ is open in the topology of $\mathcal{P}^m(a)$, for all $m>1$.
 \end{definition}
 
If $h$ is an affine map, we call it an \emph{affine configuration}. Observe that a limit geometry is a configuration of a piece. Configurations of the type $A \circ k^{\ute}$, where $A \in Aff(\C)$ and $\ute \in \Sigma^-$, are of great importance to our work and so are called \emph{affine configurations of limit geometries}.

The renormalization operators represent a way of looking into smaller parts of the Cantor set.

\begin{definition}
Let $K$ and $K'$ be two Cantor sets. Choose any pair of words $\und{a} = (a_0,\,a_1,\,\dots,\,a_n) \in \Sigma^{fin}$ and $\und{a}'= (a'_0,\,a'_1,\,\dots,\,a'_m) \in {\Sigma'}^{fin}$. Then, the renormalization operator $T_{\und{a}}T'_{\und{a}'}$ acts on any pair of configurations $h: G(a_0) \to \C$ and $h': G(a'_0) \to \C$ by 
\[
T_{\und{a}}T'_{\und{a}'}(h,h') \coloneqq (h \circ f_{\und{a}}, h' \circ f_{\und{a}'}).
\]
\end{definition}

The notation above clearly indicates that we can consider the operators $T_{\und{a}}$ and $T'_{\und{a'}}$ as separate, each acting on configurations of $K$ and $K'$ respectively. In a very similar way to the proposition \ref{lemma:limgeo}, the one defining limit geometries, one can show (see Lemma 3.11 of \cite{AM}) that the set of affine configurations of limit geometries form an attractor in the space of configurations under the action of renormalizations. Even more, see lemma 3.8 of \cite{AM}, the renormalization operators act in a very simple manner over limit geometries:

\begin{lemma}\label{simpleformula}
For any $\ute \in \Sigma^-$ and $\und{a} \in \Sigma^{fin}$, $\und{a}=(a_0,...,a_m)$, there is an affine transformation $F^{\ute}_{\und{a}} \in Aff(\C)$ such that 
\[
k^{\ute}\circ f_{\und{a}} = F^{\ute}_{\und{a}} \circ k^{\ute\und{a}}. 
\]
Moreover, this transformation can be calculated by
\begin{align*}
   DF^{\ute}_{\und{a}}  & = \lim_{n \rightarrow \infty} \left(Df_{\ute_n}(c_{\te_0})\right)^{-1} \cdot Df_{(\ute\und{a})_{n+m}}(c_{a_m})  \quad \text{and} \\
   F^{\ute}_{\und{a}}(0) & =  c^{\ute}_{\und{a}} = k^{\ute}(c_{\und{a}}).
\end{align*}
\end{lemma}

Now we properly establish the notion of \emph{stable intersection} between Cantor sets. Given two Cantor sets $K,\,K'$ and any pair of configurations $(h_a, h'_{a'}) \in \mathcal{P}(a) \times \mathcal{P}'(a') $  we say that it is:

\begin{itemize}
    \item \emph{linked} whenever $h_a(G(a)) \cap h'_{a'}(G(a')) \neq \emptyset$.
    \item \emph{intersecting} whenever $h_a(K(a)) \cap h'_{a'}(K'(a')) \neq \emptyset$.
    \item \emph{has stable intersections} whenever $\ti{h}_{a}(\ti{K}(a)) \cap \ti{h'}_{a'}(\ti{K'}(a')) \neq \emptyset$ for any pairs of Cantor sets $(\ti{K},\ti{K'}) \in \Om_\Sigma \times \Om_{\Sigma'}$ in a small neighbourhood of $(K,K')$ and any configuration pair $(\ti{h}_{a},\ti{h'}_{a'})$ that is sufficiently close to $(h_a,h'_{a'})$ in the topology of $\mathcal{P}(a) \times \mathcal{P}'(a')$.
\end{itemize} 

The set $\mathcal{I}_s(K, K')$ in the statement of theorem \ref{thm:main0} represents the set of all $\lambda \in \C$ such that $(\tau_{\lambda}, Id)$ is a pair of configurations having stable intersections in the sense just described, where $\tau_z$ is the translation by $z$ on $\C$. %To prove this theorem we will make use of the following two results: Theorem \ref{thm:main2} and Theorem \ref{thm:main}. 

%The main theorem in the introduction is a particular case of theorem \ref{thm:main}, it guarantees stable intersections for affine configurations of limit geometries of Cantor sets. To state it, we need to fix some more notation.
The main theorem in the introduction is a consequence of theorem \ref{thm:main}, because it guarantees stable intersections for affine configurations of Cantor sets. In turn, theorem \ref{thm:main} is a consequence of theorem \ref{thm:main2}. The statement of this theorem requires that we recall some more concepts.
First, notice that the space of affine configurations of limit geometries of a Cantor set can be seen as the image of the continuous association
\begin{align*}
    I : \mathcal{A} \coloneqq Aff(\C) \times \Sigma^- &\to \mathcal{P} \\
    (A,\und{\theta}) &\mapsto A \circ k^{\und{\theta}}.
\end{align*}

\begin{definition}\label{c}The space of relative affine configurations of limit geometries will be denoted by $\mathcal{C}$. It is the quotient of $\mathcal{A} \times \mathcal{A'}$ by the action of the affine group by composition on the left, that is, $\left((A, \ute), (A', \ute')\right) \mapsto \left((B \circ A, \ute), (B \circ A', \ute')\right)$, where $B$ ranges in $Aff(\C)$.
\end{definition} 

The concepts of \emph{linking}, \emph{intersection} and \emph{stable intersection} were well defined for pairs of affine configurations of limit geometries, and since they are invariant by the action of $Aff(\C)$, they are also defined for relative configurations in $\mathcal{C}$.

Also, since the renormalization operator acts by composition on the right on $(A,\und{\theta})$, its action commutes with the multiplication on the left by affine transformations and so it can be naturally defined on $\mathcal{C}$. This space can be identified with $\Sigma^- \times {\Sigma'}^{-} \times Aff(\C)$ by the identification $[(A, \ute), (A', \ute')] \equiv (\ute, \ute', {A'}^{-1} \circ A)$ and, in this manner, the topology on $\mathcal{C}$ is the product topology on $\Sigma^- \times {\Sigma'}^{-} \times Aff(\C)$. The action of the renormalization operator on a relative configuration can be described by 
\[
T_{\und{a}}T'_{\und{a}'}(\ute, \ute', A) = (\ute\und{a}, \ute'\und{a}', \left(F^{\ute'}_{\und{a}'}\right)^{-1} \circ A \circ F^{\ute}_{\und{a}}),
\]
and if $A = s z + t$, then
\begin{equation}\label{eq:formularenorma}
    \left(F^{\ute'}_{\und{a}'}\right)^{-1} \circ A \circ F^{\ute}_{\und{a}} (z) = \frac{DF^{\ute}_{\und{a}}}{DF^{\ute'}_{\und{a}'}}\, s z + \frac{1}{DF^{\ute'}_{\und{a}'}} \left( s c^{\ute}_{\und{a}} + t - c^{{\ute}'}_{\und{a}'} \right). 
\end{equation}

It is more convenient to see the space $\mathcal{C}$ through one more identification: 
\begin{align*}
    \Sigma^- \times {\Sigma'}^{-} \times Aff(\C) & \equiv  \Sigma^- \times {\Sigma'}^{-} \times \C^* \times \C \\
    (\ute, \ute', A) & \equiv (\ute, \ute', s, t),
\end{align*}
where $A(z) = s z + t $. We will call $s$ the \emph{scale} part of the relative configuration and $t$ the \emph{translation} part.  The equation \eqref{eq:formularenorma} provides to us a formula for the renormalization under this identification if we analyse the scale and translation parts separately:
\[
s \mapsto \frac{DF^{\ute}_{\und{a}}}{DF^{\ute'}_{\und{a}'}}\, s \qquad \text{and} \qquad
t \mapsto \frac{1}{DF^{\ute'}_{\und{a}'}} \left( s c^{\ute}_{\und{a}} + t - c^{{\ute}'}_{\und{a}'}\right).
\]

The space of relative scales is given by $\si=\Sigma^{-} \times {\Sigma'}^- \times J$, where $J=\C^*$. We identify $J$ with $\R\times \T$ through the map $(t,v)\to e^{t+iv}$. It acts on $\C$ by complex multiplication. The space $\CC$ of relative configurations projects to $\si$ by the map
\begin{align*}
\CC &\to \si\\
[(\tb,A),(\tb',A')] &\to (\tb,\tb',DA/DA'),
\end{align*}
where $DA$ means derivative of the affine map $A$ (which is an element in $J$). We trivialize $\CC\to \si$ in the following way: we map $[(\tb,A),(\tb',A')]\in \CC$ to $(\tb,\tb',s,\lambda)$ such that $s=DA/DA'$ and $\lambda=(A')^{-1}\circ A(0)$. In this sense we can think of $\CC$ as $\si\times \C$. The renormalization operators act on the space of relative scales by
\[T_{\und{a}}T'_{\und{a}'}(\ute, \ute', s)=\left(\tb \und{a}, \tb'\und{a}', s \cdot (DF^{\ute}_{\und{a}}/DF^{\ute'}_{\und{a}'})\right).\]
Most of the time we will work with scales which are bounded away from zero and infinity, for this purpose we introduce the notation
\[J_R=\{s\in J:\, e^{-R}\leq |s|\leq e^R\},\]
where $R$ is a positive real number.

The object which we present in the following definition will play a central role in the proof of our main theorems. It is a useful tool to get stable intersection between pairs of Cantor sets.

\begin{definition}[Recurrent compact]\label{reccompact} Let $K$ and $K'$ be a pair of Cantor sets. Let $\mathcal{L}$ be a compact set in $\mathcal{C}$. We say that $\mathcal{L}$ is \emph{recurrent} (for the pair $(K,K')$) if for any relative affine configuration of limit geometries $v \in \mathcal{L}$, there are finite words $\und{a}$, $\und{a}'$ such that $u=T_{\und{a}}T'_{\und{a'}} (v) $ satisfies $u \in \text{int }\mathcal{L}$, where the $T_{\und{a}}T'_{\und{a'}}$ are renormalization operators associated to the pair of Cantor sets $K$ and $K'$.

If such a renormalization can be done using words $\und{a}$ and $\und{a}'$ such that their total size combined is equal to one, we say that such a set is \textit{immediately recurrent}. 

\end{definition}

\textbf{Theorem B} of \cite{AM} states that if $u$ belongs to a recurrent compact set associated to a pair of Cantor sets $K$ and $K'$, then it represents pairs of affine configurations of limit geometries of these Cantor sets that have stable intersections. For the convenience of the reader, we copy its statement below.

\begin{theorem*} The following properties are true:
\begin{enumerate}
    \item Every recurrent compact set is contained in an immediately recurrent compact set.
    \item Given a recurrent compact set $\mathcal{L}$ (resp. immediately recurrent) for $g$, $g'$, for any $\ti{g}$, $\ti{g}'$ in a small neighbourhood of  $(g,g') \in \Om_{\Sigma} \times \Om_{{\Sigma}'}$ we can choose base points $\ti{c}_a \in \ti{G}(a) \cap \ti{K}$ and  $\ti{c}_{a'} \in \ti{G}(a') \cap \ti{K}'$ respectively close to the pre-fixed $c_a$ and $c_{a'}$, for all $a \in \mathbb{A}$ and $a' \in \mathbb{A'}$, in a manner that $\mathcal{L}$ is also a recurrent compact set for $\ti{g}$ and $\ti{g}'$.
    \item Any relative configuration contained in a recurrent compact set has stable intersections.
\end{enumerate}

\begin{remark}
For each pair of maps $(\ti{g},\ti{g}')$ in the small neighbourhood of $(g,g')$ in the theorem above, let $\ti{H}$ and $\ti{H}'$ be the corresponding homeomorphisms defined in \ref{def:homeo}. The base points $\ti{c}_a \in \ti{G}(a) \cap \ti{K}$ and  $\ti{c}_{a'} \in \ti{G}(a') \cap \ti{K}'$ in the theorem above are chosen so that $\ti{H}(\ti{c}_a) =  H(c_a)$ and $\ti{H}'(\ti{c}_{a'}) =  H'(c_{a'})$ for all $a \in \mathbb{A}$ and all $a' \in \mathbb{A}'$, meaning that their itineraries under the action of $\ti{g}$ and $\ti{g}'$ are the same for all pairs of maps in this neighbourhood. In subsequent contexts, the base points will be chosen in the same way.

\end{remark}    

\end{theorem*} 

%The main objective of this paper is, given a pair of Cantor sets determined by the maps $g$ and $g'$ whose sum of Hausdorff dimensions is larger than $1$, to construct a family of Cantor sets determined by maps $g^{\uom}$ close to $g$, for $\uom$ belonging to some parameter space $\Omega$, such that for some $\uom$ the pair of Cantor sets given by the maps $g^{\uom}$ and $g'$ have stable intersections.

\end{subsection}

\begin{subsection}{Perturbation of Conformal Cantor Sets}\label{sub:perC}
Let $K$ be a conformal Cantor set defined by a $C^m$ map $g$. We show that if $K$ is non-essentially real then arbitrarily close to $K$, in the $C^{[m]}$ topology, there is a $C^{\infty}$ conformal Cantor set $\ti{K}$ defined by a map $\ti{g}$ that is holomorphic on a small open neighbourhood of $\ti{K}$. This is an important property that will allow us to perturb more freely the conformal Cantor sets and adapt the random perturbation argument from \cite{MY} to our context.

We begin with the following lemma:

\begin{lemma}\label{lem:conder}
Let $K$ be a $C^m$ conformal Cantor set given by $g$. For $x \in K$ consider the set
\[K^{dir}_x :=\bigcap_{\delta>0} \ov{\left\{\frac{y-x}{|y-x|}:y\in B(x,\delta)\cap (K\setminus \{x\})\right\}}.\]
Assume that, for all $x\in K$, $K^{dir}_x$ has two linearly independent vectors (over $\R$). Then, for all $1\leq l\leq [m]$ and $x\in K$ the $l$-linear map $D^l_x g:\R^2\times ...\times \R^2 \to \R^2$ is conformal, i.e. there is a complex number $c^l_x$ such that
\[D^l_x g(z_1,...,z_l)=c^l_x\cdot z_1 \cdot z_2 \cdot \dots \cdot z_l.\]
The operation $\cdot$ in the right hand side of the last equality corresponds to complex multiplication.
\end{lemma}

\begin{proof}
Notice that the case $l=1$ is just the definition of conformality for the Cantor set. Now we proceed by induction, assume the result for $l-1$. Let $w \in K^{dir}_x$, then there are sequences $t_n \to 0$ and $w_n \to w$ such that $x+t_n w_n \in K$. Hence
\begin{align*}
D^l_x g(w,z_1...,z_{l-1})&= \lim_{n\to \infty} \frac{D^{l-1}_{x+t_n w_n}g(z_1,...,z_{l-1})-D^{l-1}_x g(z_1,...,z_{l-1})}{t_n}\\
&= \lim_{n\to \infty} \frac{c^{l-1}_{x+t_n w_n}\cdot z_1\cdots z_{l-1}-c^{l-1}_x \cdot z_1 \cdots z_{l-1}}{t_n}\\
&= \left(\lim_{n\to \infty} \frac{c^{l-1}_{x+t_n w_n}-c^{l-1}_x}{t_n}\right) \cdot z_1 \cdots z_{l-1}.
\end{align*}
This shows that the limit $\lim_{n\to \infty} \frac{c^{l-1}_{x+t_n w_n}-c^{l-1}_x}{t_n}$ exists, denote it by $c^{l}_x(w)$. Moreover
\[D^l_x g(w,z_1...,z_{l-1})=c^{l}_x(w)\cdot z_1\cdots z_{l-1}.\]
If we take another vector $\ti{w}\in K^{dir}_x$, and using the symmetry of the operator $D^l_x g$, we would have
\[c^{l}_x(w)\ti{w}=D^l_x g(w,\ti{w},1,...,1)=D^l_x g(\ti{w},w,1,...,1)= c^l_x(\ti{w}) w.\]
This shows that $\frac{c^{l}_x(w)}{w}$ does not depend on $w$, denote it by $c^l_x$. Since we can choose $w,\tilde{w}$ generating $\R^2$, we conclude that
\[D^l_x g(z_1...,z_{l})=c^{l}_x\cdot z_1 \cdots z_{l}.\]
\end{proof}

To use this lemma we need to consider Cantor sets that are indeed two dimensional. This concept is precised by the following definition. 

\begin{definition}
We will say that a Cantor set $K$ is {\em essentially real} if there exists $\ute \in \Sigma^-$ such that the limit Cantor set $K^{\ute}(\te_0)$ is contained in a straight line. Otherwise, we say it is {\em non-essentially real}.
\end{definition}

It is not difficult to prove that $K$ is essentially real if and only if for every $\ute \in \Sigma^-$ the limit Cantor set $K^{\ute}(\te_0)$ is contained in a straight line. Moreover, one can prove that $K$ being essentially real is equivalent to $K$ being contained in a $C^1$ one dimensional manifold embedded on the plane.% Indeed, let $\alpha$ be a $C^1$ local parametrization by arc length and $x=\alpha(t)\in K$. Using Taylor formula at $t$ one gets that $\alpha(t+s)=\alpha(t)+s\cdot \alpha'(t)+o(s)$. This shows that $\alpha|_{(t-s,t+s)}$ is contained inside a cone based at $x$ and with opening angle $\beta$ such that $\tan \beta = o(s)/s$. The compacity of $K$ allows us to chose the constants in the $o$ notation uniformly, that is, independent of $x$. Given $\tb \in \Sigma^-$, for all $s>0$, one can choose $n_0$ such that for all $n>n_0$ one has that $K\cap G(\tb_n) \subset \text{Im}(\alpha|_{(t-s,t+s)})$ for some parametrization by arc length $\alpha$. Therefore, for all $n>n_0$ we have that $\text{Im}(k_n^{\tb})$ is contained in a cone with opening angle $\beta$ such that $\tan \beta= o(s)/s$. Making $n$ go to infinity one gets that $K^{\tb}(\theta_0)$ is contained in a cone with opening angle $\beta$ such that $\tan \beta= o(s)/s$. Making $s$ go to zero we obtain that $K^{\tb}(\theta_0)$ is contained in a straight line.

Lemma 1.4.1 from \cite{Z} can be adapted to our context and it can be used to prove that every point $x$ belonging to a non-essentially real Cantor set $K$ verifies that $K^{dir}_x$ has two linearly independent vectors (over $\R$). We now show that being non-essentially real is an open property.% Indeed, to see that lemma 1.4.1 implies $K^{dir}_x$ has two linearly independent vectors we can proceed by contradiction. If this is not the case then there would be $\delta>0$ such that $K\cap B_{\delta}(x)$ is contained in a cone with opening angle less than $a$ (the constant from lemma 1.4.1). Now let $\tb \in \Sigma^-$ and an integer $n$ such that $x\in G(\tb^n)$ and $G(\tb^n)\subset B_{\delta(x)}$. Then $K\cap G(\tb^n)$ is contained in a cone with opening angle less than $a$, and from this we get that $k^{\tb}_n(K)$ is also contained in such a cone. Taking a a sequence of such $\tb$ and with the corresponding $n$ going to infinity, one can take a a converging subsequence and obtain that there is some $\ti{\tb}\in \Sigma^-$ such that $K^{\ti{\tb}}(\ti{\theta}_0)$ is contained in a cone with opening angle less than $a$. This implies that $G^{\ti{\tb}}(\und{a})$ intersects such cone, for all $\und{a}$, and this contradicts lemma 1.4.1. 

\begin{lemma}\label{lem:pesr}
Let $K$ be a $C^m$ non-essentially real conformal Cantor set. Every conformal Cantor set, close enough to $K$ in the $C^m$ topology, is also non-essentially real.
\end{lemma}
\begin{proof}
If the lemma does not hold, we would have a sequence $K_n$ of conformal Cantor sets converging to $K$ and such that every $K_n$ is essentially real. Let $\tb \in \Sigma^-$, denote by $k^{\tb, n}$ the limit geometry associated to $\tb$ and the Cantor set $K_n$. Since all $K_n$ are essentially real then, for all $n$, $K^{\tb,n}(\theta_0)=k^{\tb, n}(K_n)$ is contained in a line passing through the origin. Taking a subsequence we can assume that, as $n$ goes to infinity, $K^{\tb, n}(\theta_0)$ converges to a a set contained in a line passing through the origin. Using the fact that limit geometries depend continuously on the Cantor set, we conclude that $K^{\tb}(\theta_0)$ is contained in a line and therefore $K$ is essentially real, contradicting the hypothesis in the lemma.
\end{proof}

\begin{lemma}\label{lem:phol} Let $(K,g)$ be a $C^m$ non-essentially real conformal Cantor set. Arbitrarily close to $K$, in the $C^{[m]}$ topology, we can construct a $C^{\infty}$ conformal Cantor set $(\ti{K},\ti{g})$ such that $\ti{g}$ is holomorphic on a neighbourhood of $\ti{K}$.
\end{lemma}
\begin{proof}
Since $K$ is non-essentially real then the $l$-derivative $D^l_x g$, at a point $x$ in the Cantor set, is determined by a complex number, which we denote by $g^{(l)}(x)$, i.e.
\[D^l_x g(z_1...,z_{l})=g^{(l)}(x)\cdot z_1 \cdots z_{l}.\]
In this situation, the Taylor approximation of $g$ at the point $x$ is
\begin{align*}
    g(x+z) & =  g(x)+\sum_{j=1}^{[m]} \dfrac{1}{j!}\left(D^j_x g\right) (z, \dots, z) \;+ \\
    & + \int_{[0,1]^{[m]}}t^{[m]-1}_1 t^{[m]-2}_2 \cdots t_{[m]-1}\left(\left[D^{[m]}_{x+t_1 t_2 \dots t_{[m]} z} - D^{[m]}_x \right]g\right)(z,\dots,z) dt_1\dots dt_{[m]}\\
    & =  g(x)+\sum_{j=1}^{[m]} \dfrac{1}{j!} g^{(j)}(x) z^j \;+ \\
    & + \int_{[0,1]^{[m]}}t^{[m]-1}_1 t^{[m]-2}_2 \cdots t_{[m]-1}\left(\left[D^{[m]}_{x+t_1 t_2 \dots t_{[m]} z} - D^{[m]}_x \right]g\right)(z,\dots,z) dt_1\dots dt_{[m]}.
\end{align*}
Hence, $g$ is approximated (close to $x$) by a complex polynomial, which is an holomorphic function. Now, to globally aproximate $g$ by a function $\tilde{g}$ which is holomorphic in a neighborhood of its Cantor set $\tilde{K}$, we are going to take many of the previous polynomial approximations and glue them together.

Choose any real number $\rho$ larger than zero. Consider $\Lambda \subset \Sigma(\rho) $ such that $\{G(\und{a})\cap K\}_{\und{a}\in \Lambda}$ is a partition of $K$. For each $\und{a}\in \Lambda$ we choose a point $x_{\und{a}}\in G(\und{a})\cap K$ and define the polynomial
\[p_{\und{a}}(z)= \sum_{j=0}^{[m]} \dfrac{1}{j!} g^{(j)}(x_{\und{a}}) (z-x_{\und{a}})^j.\]
We can also consider $C^\infty$ functions $\phi_{\und{a}}:\C\to \R$, $\und{a}\in \Lambda$, with the following properties:
\begin{itemize}
    \item $\phi_{\und{a}}(z)=1$ for all $z\in G(\und{a})$.
    \item $supp(\phi_{\und{a}})\subset V_{\ti{c}\rho}(G(\und{a}))$, for a constant $\ti{c}$ independent of $\rho$.
    \item $supp(\phi_{\und{a}})\cap supp(\phi_{\und{b}})= \emptyset$, for all $\und{a}\neq \und{b}$.
    \item $\|D^j\phi_{\und{a}}\| \leq \ti{C} \rho^{-j}$, for a constant $\ti{C}$ independent of $\rho$.% To construct such function one can take convolution between the characteristic function of a $c\rho$-neighborhood of $G(\und{a})$ and a $C^{\infty}$ function, supported in a ball of radius $c\rho$ and with integral one. For more details see Hormander, the analysis of parial differential operators I, theorems 1.4.1 and 1.4.2.
\end{itemize}

Indeed, to be able to construct these bump functions, we only need to show that given $\und{a}^0 \neq \und{a}^1$ both in  $\Lambda \subset \Sigma(\rho)$, the distance between the pieces $G(\und{a}^0)$ and $G(\und{a}^1)$ is at least $3\ti{c} \rho$ for some constant $\ti{c} > 0$ independent of $\rho$. For that, we can suppose that $\und{a}^0 = \und{a}a_0$ and $\und{a}^1 = \und{a}a_1$ for some $a_0 \neq a_1 \in \mathbb{A}$, since this would be the worst scenario. If $\ute \in \Sigma^-$ ends with $\und{a}$, the distance between these sets is comparable to \[
\diam\left(G(\und{a})\right) \cdot d\left(k^{\ute} (G(a_0)), k^{\ute} (G(a_1))\right).
\]
Hence the existence of $\ti{c}$ follows from the compactness of the space of limit geometries.

Now, let $\hat{g}$ be $C^{\infty}$ and very close, in the $C^{[m]}$ topology, to $g$. Define $\ti{g}$, with the same domain as $g$, by
\[\ti{g}(z)=\sum_{\und{a}\in \Lambda} \phi_{\und{a}}(z) p_{\und{a}}(z)+ \left(1-\sum_{\und{a}\in \Lambda} \phi_{\und{a}}(z)\right) \hat{g}(z).\]
Notice that
\[\hat{g}-\ti{g}=\sum_{\und{a}\in \Lambda}\phi_{\und{a}} \cdot (\hat{g}-p_{\und{a}}).\]
Therefore, the $C^{[m]}$ norm of $\hat{g}-\ti{g}$ will be small provided 
\[\|D^j\phi_{\und{a}}\|\cdot \|D^{k-j}(\hat{g}-p_{\und{a}})|_{V_{\ti{c}\rho}(G(\und{a}))}\|\leq \|D^j\phi_{\und{a}}\|\cdot \left[\|D^{k-j}(g-p_{\und{a}})|_{V_{\ti{c}\rho}(G(\und{a}))}\|+\|D^{k-j}(g-\hat{g})|_{V_{\ti{c}\rho}(G(\und{a}))}\|\right]\] 
is small, for all $0\leq k\leq [m]$, $0\leq j\leq k$ (remember that support of $\phi_{\und{a}}$ is contained in $V_{\ti{c}\rho}(G(\und{a}))$). We already know that $\|D^j\phi_{\und{a}}\| \leq \ti{C} \rho^{-j}$. On the other hand, Taylor approximation implies that
\[\lim_{\rho \to 0} \sup_{\und{a}\in \Lambda} \frac{\|D^{k-j}(g-p_{\und{a}})|_{V_{\ti{c}\rho}(G(\und{a}))}\|}{\rho^{j}}=0.\]
We conclude that taking $\rho$ small enough and $\hat{g}$ close enough to $g$, we get $\ti{g}$ as $C^{[m]}$ close to $g$ as we want. Notice that thanks to the way in which we defined $\ti{g}$, it is $C^{\infty}$ and by lemma \ref{lem:pesr} we can suppose it is non-essentially real. Moreover, in the set $\sqcup_{\und{a}\in \Lambda}G(\und{a})$ the function $\ti{g}$ is holomorpic.
We can also guarantee that $\ti{g}$ verifies the hypothesis necessary to define a dynamically defined Cantor set (with the same sets $G(a)$, $a\in \mathbb{A}$), we just need to take $\ti{g}$ $C^1$-close enough to $g$. Even more, the Cantor set $\ti{K}$, generated by $\ti{g}$, is contained in $\sqcup_{\und{a}\in \Lambda}G(\und{a})$. 
\end{proof}

To prove our main theorems we will use the scale recurrence lemma (see subsection \ref{se:scrl}). To use this lemma we need that our Cantor sets are non-essentially affine. A $C^m$ Cantor set $K$, with $m\geq 2$, is said to be \emph{non-essentially affine} when there is a pair of limit geometries $\ute^0$ and $\ute^1$ in $\Sigma^-$ such that $\te^0_0=\te^1_0$ and a point $x_0 \in K^{\ute^0}(\te_0)$ such that
\[
D^2\left[k^{\ute^1}\circ(k^{\ute^0})^{-1}\right](x_0) \neq 0.
\]
The following lemma allow us to perturb and get a non-essentially affine Cantor set.
\begin{lemma}\label{lemma:c8approx}
Let $K$ be a $C^m$ non-essentially real conformal Cantor set. Arbitrarly close to $K$, in the $C^{[m]}$ topology, there is a $C^{\infty}$ Cantor set $\ti{K}$ which is non-essentially real, non-essentially affine and such that its expanding function $\ti{g}$ is holomorphic in a neighborhood of $\ti{K}$.
\end{lemma}
\begin{proof}

Let $(\hat{K},\hat{g})$ be the perturbed Cantor set from lemma \ref{lem:phol}. If $\hat{K}$ is non-essentially affine we are done. Otherwise, choose a piece $G(a) $, $a \in \mathbb{A}$, and let $c_a$ be the corresponding base point. As previously mentioned, it is pre-periodic.

\begin{claim} If $\rho > 0$ is sufficiently small, we can chose $\und{a} \in \Sigma(\rho)$ ending with $a$ so that no word in $\Sigma(\rho^{1/3})$ appears more than once in $\und{a}$. Given any word $\und{a}$ with this property, if $\rho > 0$ is sufficiently small, there is $\tb^0 \in \Sigma^-$ ending with $\und{a}$ and such that $\und{a} $ does not appear elsewhere in $\tb^0$. Furthermore, there is $\tb^{1} \in \Sigma^-$ ending with $a$ such that no subword of $\und{a}$ belonging to $\Sigma(\rho^{1/3})$ appears in it. 
\end{claim} 

\begin{proof}

Since the shift is mixing over $\Sigma$, there must be at least two sequences of distinct lengths (both larger than 1) $\und{a}^1, \und{a}^2  \in \Sigma^{fin}$ such that both end and begin with $a$ and no other letter in these words is $a$. Now construct $\und{a}$ as 
\[
\und{a} = \underbrace{\und{a}^1\, \und{a}^1 \, \dots\,\und{a}^1}_{N_1 \text{ times}} \underbrace{\und{a}^1\, \und{a}^2 \, \dots\,\und{a}^1 \und{a}^2}_{N_2 \text{ times}} \underbrace{\und{a}^1\, \und{a}^2 \und{a}^2 \, \dots\,\und{a}^1 \und{a}^2 \und{a}^2 }_{N_3 \text{ times}} \underbrace{\und{a}^2\, \dots \und{a}^2 \und{a}^2 }_{N_4 \text{ times}}.
\]
If $\rho$ is sufficiently small, a suitable choice of $N_1$, $N_2$, $N_3$ and $N_4$ can be done so that no subword in $\Sigma(\rho^{1/3})$ appears more than once. This can be seen by analysing the behaviour of the distance between two consecutive letters $a$ in any two subwords of $\und{a}$.

Now, choose a subword $\und{a}' \in \Sigma(\rho^{1/3})$ such that $\und{a}$ begins with $\und{a}'$. If $\und{a}'$ never appears in $\tb' \in \Sigma^-$, then we can make $\tb^0 = \tb'\und{a}$. Indeed, for $\und{a}$ to appear more than once, the word $\und{a} / \und{a}'$ must be contained in $\und{a}$ but in another position. This implies that a subword in $\Sigma(\rho^{1/3})$, corresponding to the beginning of $\und{a} / \und{a}'$ for example, appears more than once in $\und{a}$. But this does not happen.

Now, suppose we are given a word $\und{b} \in \Sigma(\rho)$. We want to prove there is some $\tb \in \Sigma^-$ such that $\und{b}$ never appears in $\tb$. Choose a beginning $\und{b}_1 \in \Sigma({\rho}^{1/3})$ of $\und{b}$ and an ending $\und{b}_2 \in \Sigma({\rho}^{1/3})$. Let $\und{b}' \in \Sigma({\rho}^2)$ be such that $\und{b}_1$ and $\und{b}_2$ never appear in it and also suppose that the first letter of $\und{b}'$ is the same as its last letter. Then, similarly to the analysis before, we can make $\tb = \dots \und{b}'\dots \und{b}' \in \Sigma^-$. 

The existence of $\und{b}'$ comes from a counting argument, in which we show that the words in which $\und{b}_1$ or $\und{b}_2$ appear do not account for all possible words $\und{b}' \in \Sigma(\rho^2)$. Remember that if $d$ is the Hausdorff dimension of $K$, then $\# \Sigma(\rho) \approx \rho^{-d} $. 

The number of words in $\Sigma(\rho^2)$ ending with $\und{b}_1$ is $\lesssim \rho^{-5d/3}$. More than that, if we fix a starting position for the appearance of $\und{b}_1$, such as it beginning in the $1000^{th}$ letter of $\und{b}'$ (remember $\rho$ is very small), then the same estimate remains true. Notice however that the number of letters of $\und{b}'$ is $\lesssim \log{\rho^{-1}}$, and so the number of words in $\Sigma(\rho)$ that fail our requirements is $\lesssim \rho^{-5d/3} \log{\rho^{-1}}$, thus, for $\rho$ small enough, there must be $\und{b'} \in  \Sigma(\rho^2)$ that satisfies our requirements.

To construct $\tb^1$, we can use the same argument, all we need to observe is that the number of subwords of $\und{a}$ in $\Sigma (\rho^{1/3})$ is also $\lesssim \log{\rho}$.

\end{proof}

By maybe shrinking $\rho$ even further, we can assume that the periodic part of the itinerary of $c_a$ is a word very small when compared to $\und{a}$. The combinatorial conditions above imply that $f_{\tb^0_n} (c_a) $ belongs to $G(\und{a})$ only when $\tb^0_n = \und{a}$. Besides, $f_{\tb^1_n} (c_a) $ never belongs to this set.

Let $\phi: G(\und{a}) \to \C$ be an holomorphic map $C^{[m]}$ close to the identity, and suppose it fixes the point $c_{\und{a}}$ and has derivative equal to the identity at this point. Similar to the construction on lemma \ref{lem:phol}, define a new Cantor set given by an expanding map $\ti{g}$ that is equal to $\hat{g}$ outside a small neighborhood of $G(\und{a})$ and equal to $\hat{g} \circ \phi$ in $G(\und{a})$. Note that the perturbed base point $\ti{c}_a$ is equal to $c_a$, thanks to the pre-periodicity of $c_a$ and the combinatorial properties of $\und{a}$. Moreover, the limit geometry corresponding to $\tb^1$ stays the same close to $c_a$, that is, we can choose a neighbourhood $V = G(\und{b})$ of the base point $\ti{c}_a=c_a$, with $\und{b} \in \Sigma^{fin}$ sufficiently large, such that $\ti{k}^{\tb^1}|_V = \hat{k}^{\tb^1}|_V$. On the other hand, for $\tb^0$, 
\[
\ti{k}^{\tb^0}|_V =\hat{k}^{\tb^0}|_V \circ f_{\und{a}}^{-1} \circ  \phi^{-1} \circ f_{\und{a}},
\]
since the affine reescalings $\Phi_{\tb_n^{0}}$ stay the same.
Notice that the map $\ti{g}$ is still holomorphic in a neighbourhood of $\ti{K}$ and, because of lemma \ref{lem:pesr}, it is non-essentially real if $\phi$ is $C^1$ sufficiently close to the identity. However, we can still choose $\phi$ so that 
\[
D^2\left(f_{\und{a}}^{-1} \circ  \phi^{-1} \circ f_{\und{a}}\right)(c_a) \neq 0.
\]
Hence $D^2 \ti{k}^{\tb^0} (c_a) \neq D^2 \hat{k}^{\tb^0} (c_a)$. This implies that 
\[
D^2\left[\ti{k}^{\ute^1}\circ(\ti{k}^{\ute^0})^{-1}\right](0) \neq 0
\]
and so the new Cantor set is also non-essentially affine.

%Consider $(\hat{K},\hat{g})$ be the perturbed Cantor set from lemma \ref{lem:phol}. If $\hat{K}$ is non-essentially affine we are done. Otherwise, choose $\tb^1,\, \tb^2\in \Sigma^-$ and $\und{a}\in \Sigma^{fin}$ such that $\tb^1$ ends with $\und{a}$ and $\tb^2$ does not contain the word $\und{a}$. Define a new Cantor set given by a expanding map $\ti{g}$ such that $\ti{g}$ coincides with $\hat{g}$ outside a small neighborhood of $\hat{G}(\und{a})$ and $\ti{g}(z)=\hat{g}\circ \psi(z)$ for all $z\in \hat{G}(\und{a})$, where $\psi$ is a function very close to the identity.
\end{proof}

\begin{remark}\label{remk:pholder}
We observe that the $C^{\infty}$ Cantor set $\ti{K}$ constructed in lemma \ref{lem:phol} can be also assumed to be close to $K$ in the $C^{m'}$ topology for all $m' \in (1,m)$. All one needs to do is to choose $\hat{g}$ close to $g$ in this topology. This can be done using a mollifier $\varphi$ supported in a very small neighbourhood of the origin and making $\hat{g} = g * \varphi$. 
Notice that the $C^m$ proximity between the maps $p_{\und{a}}$ and $g$ in $supp(\phi_{\und{a}})$ comes from the fact that $g$ is $C^m$ and $supp(\phi_{\und{a}})$ has diameter of order $\rho$. Moreover, using the fact that $D^{[m]}g$ is $\ve$-Holder, for $\ve=m-[m]$, one gets the improved estimate $\|D^{m-j}(g-p_{\und{a}})|_{V_{\ti{c}\rho}(G(\und{a}))}\|=O(\rho^{j+\ve})$ for all $0\leq j\leq [m]$. This can be seen analysing the Taylor approximation of the derivatives of $g$ in this domain.

It follows that, given any Cantor set $K \in \Omega_{\Sigma}$, essentially real or not, arbitrarily close to $K$ in the topology of $\Omega_{\Sigma}$, there is a Cantor set $\ti{K} \in \Omega^{\infty}_{\Sigma} \subset \Omega_{\Sigma}$ that satisfies the conclusion of lemma \ref{lemma:c8approx}. Indeed, suppose that $K$ is $C^{1+\ve}$ for some $\ve\in (0,1)$. Then the only part of the argument (polynomial approximation in the $C^{1+\ve}$ topology) in lemma \ref{lem:phol} which uses the non-essentially real hypothesis can be done using the conformality of the Cantor set. Moreover, one can also choose $\ti{K}$ to be non-essentially real. To do so, we first suppose that the expanding function $\ti{g}$ is already holomorphic in a neighborhood of $\ti{K}$. Then one chooses a periodic point $p\in \ti{K}$, of period $n$, and observes that if $D \ti{g}^n(p)\notin \R$ then $\ti{K}$ can not be essentially real. Thus, if $D \ti{g}^n(p)\notin \R$ we are done, otherwise we perturb $\ti{K}$ along the periodic orbit $\{p,...,\ti{g}^{n-1}(p)\}$ to get such property. One execution of this idea of perturbing along a periodic orbit can be found in the proof of theorem \ref{thm:main}.
\end{remark}
%As a consequence of this discussion we have the following lemma:

%\begin{lemma} Arbitrarily close to a conformal Cantor set $(K,g)$ we can construct another conformal Cantor set $(\ti{K},\ti{g})$ with the following properties:
%\begin{itemize}
 %   \item Every point $x \in K$ is accumulated by points in $K$ through at least two different directions, that is, $K^{dir}_x$ has two linear independent vectors.  
 %   \item $\ti{g}$ is holomorphic on a neighbourhood of $\ti{K}$.
%\end{itemize}
%\end{lemma}

%We will call a Cantor set with these properties \emph{tame}.
\end{subsection}

\subsection{Main theorems}

Here we state our main theorems. The proof of theorem \ref{thm:main2} will be given throughout the remaining sections. Using this theorem we will prove theorem \ref{thm:main}. In particular, we will get that there is an open and dense set, among pairs of conformal Cantor sets $K$, $K'$ with $HD(K)+HD(K')>2$, such that all elements in this set verify $\text{int}(K-K')\neq \emptyset$. Before stating the theorems, we remark that the Hausdorff dimension varies continuously with the Cantor set. This is proven in \cite{PT95} for Cantor sets in the real line and the argument there can be adapted to our context.   

\begin{theorem}\label{thm:main2}
Given a pair  of non-essentially real conformal Cantor sets $(K, K')$ in $\Omega^{\infty}_{\Sigma}\times \Omega^{\infty}_{\Sigma'}$, such that $HD(K)+HD(K')> 2$. Arbitrarily close to $K$, $K'$, in the $C^{\infty}$ topology, we can find conformal Cantor sets $\tilde{K}$, $\tilde{K}'$ in $\Omega^{\infty}_{\Sigma}$, $\Omega^{\infty}_{\Sigma'}$ respectively, such that $\tilde{K}$, $\tilde{K}'$ has a non empty recurrent compact set.
\end{theorem}

%\begin{theorem}\label{thm:main2}
%Given a pair  of non essentially real conformal Cantor sets $(K, K')$ in $\Omega^{r}_{\Sigma}\times \Omega^{r'}_{\Sigma'}$, such that $HD(K)+HD(K')> 2$. Arbitrarily close to $K$, $K'$, in the $C^{r},\, C^{r'}$ topologies, we can find conformal Cantor sets $\tilde{K}$, $\tilde{K}'$ in $\Omega^{\infty}_{\Sigma}$, $\Omega^{\infty}_{\Sigma'}$ respectively, such that $\tilde{K}$, $\tilde{K}'$ has a non empty recurrent compact set.
%\end{theorem}

Define the set $U$ as the pairs of conformal Cantor sets $(K,K')$ in $\Omega_{\Sigma}\times \Omega_{\Sigma'}$ such that for every relative configuration $(\tb,\tb',s,t)$, there is $\ti{t}\in \C$ such that the configuration $(\tb,\tb',s,\ti{t})$ has stable intersection. 

\begin{theorem}\label{thm:main}
The set $U$ is open in $\Omega_{\Sigma}\times \Omega_{\Sigma'}$ and $U\cap \Omega^{\infty}_{\Sigma}\times \Omega^{\infty}_{\Sigma'}$ is dense, in the $C^{\infty}$ topology, in $\{(K,K')\in \Omega^{\infty}_{\Sigma}\times \Omega^{\infty}_{\Sigma'}:\, HD(K)+HD(K')>2,\, K,\,K'\text{ are non-essentially real}\}$. Moreover, for any $(K,K')$ in $U$ and $(h,h')\in \mathcal{P}\times \mathcal{P'}$ such that $Dh(z)$ and $Dh'(z')$ are conformal for all $(z,z')\in K\times K'$, the set
\[\mathcal{I}_s=\{\lambda\in \C: (h+\lambda,h')\text{ has stable intersection}\}\]
is dense in
\[\mathcal{I}=\{\lambda\in \C: (h+\lambda,h')\text{ is intersecting}\}.\]
In particular, $\text{int}(K-K')\neq \emptyset$.
\end{theorem}

\begin{proof}
The proof is very similar to the one for the corresponding result in \cite{MY}, except for the use of lemma \ref{lem:genhyp}. For the openness of $U$, one observes that, if $R$ is big enough, then any relative configuration $(\tb,\tb',s,t)$ can be transported, using a renormalization operator, to the set $\Sigma^-\times \Sigma^{\prime -}\times \{e^{-R}\leq |s|\leq e^{R}\}\times \C$. Given $(K,K')$ in $U$, from compactness of the set $\Sigma^-\times \Sigma^{\prime -}\times \{e^{-R}\leq |s|\leq e^{R}\}$ one sees that there is a neighborhood of $(K,K')$ such that for any pair in this neighborhood, we have $(\tb,\tb',s,t)\in \Sigma^-\times \Sigma^{\prime -}\times \{e^{-R}\leq |s|\leq e^{R}\}\times \C$ implies there is $\ti{t}\in \C$ such that $(\tb,\tb',s,\ti{t})$ has stable intersection. Thus the same happens for the whole $\Sigma^-\times \Sigma^{\prime -}\times \C^*\times \C$, and the neighborhood is contained in $U$.   

Furthermore, in the same context of the previous paragraph, from compactness of the set $\Sigma^-\times \Sigma^{\prime -}\times \{e^{-R}\leq |s|\leq e^{R}\}$, for each $r>1$ there is some $\delta > 0$ such that if $\ti{h}$ and $\ti{h}'$ are maps $\delta$-close to the identity in the $C^r$ metric, then  $(\ti{h}\circ B \circ k^{\ute}, \ti{h}' \circ k^{\ute'}) $ has stable intersections, where $B(z) = sz + \ti{t}$ and $e^{-R}\leq |s|\leq e^{R}$. We will need this later.

For the denseness we use theorem \ref{thm:main2}. Given $(\tilde{K},\tilde{K}')\in \Omega^{\infty}_{\Sigma}\times \Omega^{\infty}_{\Sigma'}$, with $HD(K)+HD(K')>2$ and both of them non-essentially real, arbitrarily close to it there is $(K,K')\in \Omega^{\infty}_{\Sigma}\times \Omega^{\infty}_{\Sigma'}$ having a non empty recurrent compact set $\LL$. Perturbing, we may assume that $(K,K')$ also has periodic points $p$, $p'$, associated to finite words $\und{a}$, $\und{a}'$, that is $f_{\und{a}}(p)=p$, $f_{\und{a}'}(p')=p'$, such that if we write
\[a_{i,j}=Df_{\und{a}^i}(p)/Df_{\und{a}^{\prime j}}(p')=DF^{\tb_0}_{\und{a}^i}/DF^{\tb_0'}_{\und{a}^{\prime j}}\in \C^*,\]
where $\und{a}^i$ is concatenation of $\und{a}$ with itself $i$ times, similarly for $\und{a}'$, $\tb_0=(...,\und{a},\und{a})$ and $\tb'_0=(...,\und{a}',\und{a}')$, then the set $\{a_{i,j}\}_{i,j\in \Z^+}$ is dense in $\C^*$ (see lemma \ref{lem:genhyp}).

More precisely, in order to get the density of $\{a_{i,j}\}$, we need to perturb $(g,g')$ such that $Dg^{m}(p),\, Dg^{\prime m'}(p')\in \C^*\approx \R\times \T$ have the property in lemma \ref{lem:genhyp}, where $m$, $m'$ are the periods of $p$, $p'$, respectively. To do this, we define a family of conformal Cantor sets given by expanding functions $(g_{x},g'_{y})$ depending in complex parameters $x$, $y$, such that the pairs $(Dg_x^{m}(p_x), Dg_y^{\prime m'}(p_y))\in \C^*\times \C^*$ form an open set. Choose a word $\und{d}\in \Sigma(\alpha)$ such that $p\in G(\und{d})$, $\alpha$ is small enough such that there are not more points of the periodic orbit of $p$ contained in $G(\und{d})$ and $g$ is holomorphic in $V_{\alpha}(K)$ (by lemma \ref{lem:phol} we may assume this). For $0<c<1$ small enough one has that $K\cap V_{c\alpha}(G(\und{d}))=K\cap G(\und{d})$. We choose a $C^{\infty}$ function $\psi_x$ such that
\[\psi_x(z)= \begin{cases} x\cdot (z-p)+p & \text{ if } z\in V_{c\alpha/3}(G(\und{d})),\\
z & \text{ if } z\notin V_{2c\alpha/3}(G(\und{d})),
\end{cases}\]
where $x$ is in the ball of center $1$ and radius $\delta'$ in $\C$. Define $g_x=g\circ \psi_x$, notice that we can take $g_x$ as close as we want to $g$ in the $C^{\infty}$ topology by choosing $\delta'$ small enough. If we choose $\delta'$ small enough then we can guarantee that the Cantor set $K_x$, associated to $g_x$, and the set $\psi_x(K_x)$ are contained in $V_{c\alpha/3}(K)$ and therefore $K_x\cap V_{2c\alpha/3}(G(\und{d}))\setminus V_{c\alpha/3}(G(\und{d}))=\emptyset$, which implies that
\[Dg_x(z)=Dg(\psi_x(z))\cdot D\psi_x(z)\]
is conformal for all $z\in K_x$. This proves that $K_x$ is a conformal Cantor set. Moreover, notice that $p
_x=p$ is still a periodic point of $g_x$ with the same period $m$, and $Dg_x^{m}(p_x)=x\cdot Dg^{m}(p)$. Doing the same construction for $g'$ one sees that for some value of $x$, $y$ the pair $(Dg_x^{m}(p_x), Dg_y^{\prime m'}(p'_y))$ satisfy the hypothesis of lemma \ref{lem:genhyp}.

Using equation \eqref{eq:formularenorma}, the denseness of the set $\{a_{n,m}\}_{n,m\in \Z^+}$ and the fact that $\LL$ has non-empty interior, one concludes that for any $(\tb,\tb',s,t)$ there is $m,\,n\in \Z^+$, $\und{b}$, $\und{b}'$ in $\Sigma^{fin}$, $\Sigma^{\prime fin}$, respectively, and $\lambda\in \C$ such that $T_{\und{a}^n\und{b}}T'_{\und{a}^{\prime m}\und{b}'}(\tb,\tb',s,t+\lambda)\in \LL$. Therefore $(K,K')\in U\cap \Omega^{\infty}_{\Sigma}\times \Omega^{\infty}_{\Sigma}$.

For the final part, let $\lambda \in \C$ be such that $(h + \lambda, h')$ is intersecting and take any $\ve > 0$. Then there is at least one pair of words $\und{a} = (a_0,\,a_1,\, \dots,\,a_n) \in \Sigma^{fin}$ and $\und{a}'= (a'_0, \,,a'_1,\,\dots,\,a'_m) \in {\Sigma'}^{fin}$ sufficiently large such that $((h + \lambda) \circ f_{\und{a}}, h'\circ f_{\und{a}'})$ are still intersecting and the diameters of the sets $G(\und{a})$ and $G(\und{a}')$ are smaller than $\ve$. Indeed, if it was not the case, the sets $(h+\lambda)(G(\und{a}))$ and $h'(G(\und{a}))$ would be disjoint for all $\und{a}$ and $\und{a}'$ sufficiently large, a contradiction with the intersecting hypothesis. We will prove that if $\ve$ is small there exists $\ti{\lambda} \in \C$ such that $|\ti{\lambda} - \lambda| \lesssim \ve$ and $(h + \ti{\lambda}, h')$ has stable intersections.

Now, following definition 3.10 of \cite{AM}, we can scale these pairs of configurations by normalizing in the second coordinate, obtaining another pair with intersection. More precisely, we choose $A' \in Aff(\C)$ such that 
\[
(A' \circ h' \circ f_{\und{a}'}) (c_{a'_m}) = 0 \quad \text{ and } \quad D (A' \circ h' \circ f_{\und{a}'}) (c_{a'_m}) = Id  \]
and consider now the pair of configurations $(A'\circ (h+\lambda) \circ f_{\und{a}}, A' \circ h' \circ  f_{\und{a}'})$. This pair of configurations is intersecting, because this property is clearly preserved under composition on the left by affine transformations.

Let $r \in (1,2)$ be such that $h$ and $h'$ are both $C^r$. Reasoning as in the proof of lemma 3.11 of \cite{AM} (see claim 3.12), we observe that this pair of configurations can be written as 
\[
(A'\circ (h+\lambda) \circ f_{\und{a}}, A' \circ h' \circ  f_{\und{a}'}) = (\ti{h}\circ B \circ k^{\ute}, \ti{h}' \circ k^{\ute'} ),
\]
where: $\ute \in \Sigma^-$ ends with $\und{a}$; $\ute' \in {\Sigma'}^-$ ends with $\und{a}'$; the maps $\ti{h}$ and $\ti{h}'$ are close to the identity in the $C^{r}$ topology,
and $B$ is a bounded affine transformation in $Aff(\C)$. More precisely, if we set\footnote{Notice that if $Dh(x),\, Dh'(y)$ were not conformal for all $x\in K,\, y\in K'$ we could not guarantee that $B\in Aff(\C)$.} 
\[
DB = DA' \cdot Dh(c_{\und{a}}) \cdot Df_{\und{a}}(c_{a_n})  = s \in \C \quad \text{and} \quad B(0) = A' \circ (h+\lambda) \circ f_{\und{a}}(c_{a_n}) = t \in \C,
\] 
and choose $\und{a}$ and $\und{a}'$ with appropriate lengths, then $ e^{-R} \le \lv s \rv \le e^{R} $ and there is some constant $c>0$ (independent from $\ve$) such that $\lv B(0) \rv < c\,e^{R}$ and the distance of the maps $\ti{h}$ and $\ti{h}'$ to the identity is bounded by $c\, \diam(G(\und{a}))^{r-1}$ and $c \, \diam(G(\und{a}'))^{r-1}$

Consider now the relative configuration $(\ute, \ute', s, t)$. By the previous part, there is some $\ti{t} \in \C$ such that, writing $\ti{B}(z) = sz +  \ti{t}$, the pair of configurations $(\hat{h}\circ \ti{B} \circ k^{\ute}, \hat{h}' \circ k^{\ute'}) $  has stable intersections for every pair of maps $\hat{h}, \hat{h}'$ $\delta$-close to the identity in the $C^{r} $ metric. Notice that, since $(\ti{h}\circ B \circ k^{\ute}, \ti{h}' \circ k^{\ute'} )$ is intersecting, $\lv t-  \ti{t}\rv$ is bounded by $e^R\left(\diam(k^{\ute}) + \diam(k^{\ute'})\right) $. Therefore, by compactness of the set of all pairs of limit geometries $(k^{\ute}, k^{\ute'})$, we may enlarge $c$, still independently of $\ve$, so that $\lv t- \ti{t} \rv$ is bounded by $c\, e^R$.

Now we choose $\ti{\lambda}$ such that $A'(h(c_{\und{a}}) + \ti{\lambda}) =  \ti{t}$; it follows that \[
(A'\circ (h + \ti{\lambda}) \circ f_{\und{a}}, A' \circ h' \circ  f_{\und{a}'}) = (\hat{h} \circ \ti{B} \circ k^{\ute}, \ti{h}' \circ k^{\ute'}),
\]
where $D\ti{B} = s$, $\ti{B}(0) = \ti{t}$, and the distances of $\hat{h}$ and $\ti{h}'$ to the identity are bounded from above by $c\, \diam(G(\und{a}))^{r-1}$ and $c \, \diam(G(\und{a}'))^{r-1}$. Choosing $\ve$ sufficiently small, and so $\und{a}$ and $\und{a}'$ very big, these distances to the identity are less than $\delta$. This implies that $(h+\ti{\lambda}, h')$ have stable intersections. Notice, finally, that
\[
\lv \ti{\lambda} - \lambda \rv \le \lv {DA'}\rv^{-1} \cdot \lv \ti{t} - t \rv \lesssim \diam(G(\und{a}')) c\, e^R \lesssim \ve c\, e^R.
\]
Hence, making $\ve$ very small, we approximate $(h+\lambda, h')$ by $(h+\ti{\lambda}, h')$ that has stable intersections, completing the proof.
\end{proof}

\begin{remark}
Notice that thanks to remark \ref{remk:pholder}, the set $U\subset \Omega_{\Sigma}\times \Omega_{\Sigma'}$ is dense,  with respect to the topology of $\Omega_{\Sigma}\times \Omega_{\Sigma'}$, inside the set $\{(K,K'): HD(K)+HD(K')>2\}$. 
\end{remark}

\begin{lemma}\label{lem:genhyp}
Let $(t,v),(t',v')\in (\R\setminus\{0\})\times (\R/(2\pi \Z))$, and consider the subgroup $E=\{m(t,v)+n(t',v'):\,m,n\in \Z\}$. Let $w,\,w'\in \R$ be representatives of $v,\,v'$, respectively. Then $E\subset \R\times \T$ is dense if and only there is not $(\beta_1,\beta_2)\in \Z^2\setminus\{0\}$ such that
\[\beta_1\cdot \frac{t}{t'}+\beta_2\left(w-w'\frac{t}{t'}\right)\in \Z.\]
Moreover, if $E$ is dense and $t/t'>0$ then $\{m(t,v)-n(t',v'):\,m,n\in \Z^+\}$ is dense. The set of pairs $((t,v),(t',v'))\in (\R\times \T)^2$ such that $E$ is dense is a countable intersection of open and dense sets. 
\end{lemma}
\begin{proof}
The lemma is proved using Kronecker's theorem. It states that a vector $(w_1,...,w_k)\in \T^k$ generates a dense subgroup if and only if there is not $(a_1,...,a_k)\in \Z^k\setminus \{0\}$ such that $a_1w_1+...+a_kw_k=0$.

Let $p:\R\to \T$ be the canonical projection and choose $w,\,w'\in \R$ such that $p(w)=v$, $p(w')=v'$. Note that $E$ is dense if and only if the set 
\[\{(t,w),(t',w'),(0,1)\}\]
generates a dense subgroup in $\R^{2}$. 

Moreover, this last property is invariant under invertible linear transformations in $\R^{2}$. Let $A:\R^{2} \to \R^{2}$ be the linear map such that $A(t',w')=(1,0)$ and $A(0,1)=(0,1)$. Then, $E$ is dense if and only if the set
\[\{A(t,w),(1,0),(0,1)\}\]
generates a dense subgroup in $\R^{2}$. It is clear that this happens if and only if the projection of $A(t,w)$ to $\T^2$ generates a dense subgroup in $\T^{2}$. Thus, using Kronecker's theorem we see that $E$ is dense if and only if there is not $(\beta_1,\beta_{2})\in \Z^{2}\setminus \{0\}$ such that
\[ \langle (\beta_1,\beta_2), A(t,w)\rangle \in \Z.\] 
Moreover, it is not difficult to see that $A(t,w)=(t/t',w-(w't/t'))$. This proves the first part of the lemma.

Notice that the set of pairs $((t,v),(t',v'))$ such that $E$ is dense corresponds to the intersection, varying $a\in \Z^{2}\setminus \{0\}$, of the sets
\[\{((t,p(w)),(t',p(w'))): \langle a, A(t,w)\rangle \notin \Z\},\]
and each one of these sets is open and dense. 

Finally, $\{m(t,v)-n(t',v'):\,m,n\in \Z^+\}$ will be dense if and only if
\[\{mA(t,w)-n(1,0)+r(0,1):\,m,n\in\Z^+,\,r\in \Z\}\]
is dense in $\R^2$. If $E$ is dense, then the projection of $\{mA(t,w):\,m\in \Z^+\}$ to $\T^2$ is dense. If we also have $t/t'>0$, then from the expression
\[mA(t,w)-n(1,0)+r(0,1)=(m(t/t')-n,m[w-w'(t/t')]+r),\]
it is not difficult to see that $\{m(t,v)-n(t',v'):\,m,n\in \Z^+\}$ is dense in this case.
\end{proof}

\subsection{Scale Recurrence Lemma}\label{se:scrl}

In this section we will see how to adapt the scale recurrence lemma from \cite{MZ} to our context. First we note that limit geometries in \cite{MZ} were defined slightly different, they were defined as the limit of the function
\[\tilde{k}^{\tb}_n=\tilde{\Phi}_{\tb_n}\circ f_{\tb_n},\]
where $\tilde{\Phi}_{\tb_n}$ is the unique affine function satisfying $diam(\tilde{\Phi}_{\tb_n}(G(\tb_n)))=1$, $D\tilde{\Phi}_{\tb_n}\cdot Df_{\tb_n}(c_{\theta_0})\in \R^+$ and $\tilde{\Phi}_{\tb_n}(f_{\tb_n}(c_{\theta_0}))=0$. Denote those limit geometries by $\tilde{k}^{\tb}$. It is clear that there is a complex number $R(\tb_n)$ such that $\tilde{k}^{\tb}_n=R(\tb_n)\cdot k^{\tb}_n$. It is not difficult to prove that one can go to the limit and find a complex number $R(\tb)$ such that $\tilde{k}^{\tb}=R(\tb)\cdot k^{\tb}$. Moreover, $R(\tb)$ is uniformly bounded from above and below, i.e there is $c>0$ such that $c^{-1}\leq|R(\tb)|\leq c$. One can also show that $R(\tb)$ depends Lipschitz in $\tb$ in the sense that there is a constant $C$ such that
\[ \left| \frac{R(\tb^1)}{R(\tb^2)}-1\right|\leq C d(\tb^1, \tb^2).\]
We will denote by $\tilde{F}^{\tb}_{\und{a}}$ the affine function defined by
\[\tilde{k}^{\tb}\circ f_{\und{a}}=\tilde{F}^{\tb}_{\und{a}}\circ \tilde{k}^{\tb\und{a}}.\]
This affine function can be written in terms of the numbers $\tilde{r}^{\tb}_{\und{a}}\in \R^+$, $\tilde{v}^{\tb}_{\und{a}}\in \T$ and $\tilde{c}^{\tb}_{\und{a}}\in \C$ by the formula
\[\tilde{F}^{\tb}_{\und{a}}(z)=\tilde{r}^{\tb}_{\und{a}}\exp(i \tilde{v}^{\tb}_{\und{a}})z+\tilde{c}^{\tb}_{\und{a}}.\]
The maps $\tilde{F}^{\tb}_{\und{a}}$ and $F^{\tb}_{\und{a}}$ are related by the equations
\[DF^{\tb}_{\und{a}}=\frac{R(\tb\und{a})}{R(\tb)}D\tilde{F}^{\tb}_{\und{a}},\,\, F^{\tb}_{\und{a}}(0)=\frac{1}{R(\tb)}\tilde{F}^{\tb}_{\und{a}}(0).\]
Now we assume we have two Cantor sets, $K$ and $K'$, and discuss how to go from the renormalization operators in \cite{MZ} to ours. Define $\phi_{\tb,\tb'}:J\to J$ and $L:\R\times \T^2\to J$ by $$\phi_{\tb,\tb'}(s)=\frac{R(\tb)}{R'(\tb')}\cdot s,\,\,\,L(t,v,v')=(t,v-v').$$
Remember that we identify $J=\C^*$ with $\R\times \T$ through $(t,v)\to e^{t+iv}$, and define $\phi:\Sigma^-\times\Sigma^{\prime -}\times \R\times \T^2 \to \Sigma^-\times\Sigma^{\prime -}\times J$ given by
\[\phi(\tb,\tb',t,v,v')=(\tb,\tb', \phi_{\tb,\tb'}(L(t,v,v'))).\]
From the previous equations, one easily proves that the renormalization operators of \cite{MZ}, which are given by
\[T_{\und{a},\und{a}'}(\tb,\tb',t,v,v')=(\tb\und{a},\tb'\und{a}', t+\log\frac{\tilde{r}^{\tb}_{\und{a}}}{\tilde{r}^{\tb'}_{\und{a}'}},v+\tilde{v}^{\tb}_{\und{a}},v'+\tilde{v}^{\tb'}_{\und{a}'})\]
and act on $\Sigma^-\times \Sigma^{\prime -}\times \R\times \T^2$, are related to our renormalization operators $T_{\und{a}}T'_{\und{a}'}$ by the ``semiconjugacy'' $\phi$, i.e. $\phi \circ T_{\und{a},\und{a}'}=T_{\und{a}}T'_{\und{a}'}\circ \phi$. Using $\phi$, it is not difficult to transport the scale recurrence lemma from \cite{MZ} to one for $T_{\und{a}}T'_{\und{a}'}$:
\begin{lemma}\label{lem:scl}
Suppose that $K$, $K'$ are non-essentially affine and non-essentially real. If $R,c_0$ are conveniently large, there exist $c_1,c_2,c_3,c_4,\rho_0>0$ with the following properties: given $0<\rho<\rho_0$, and a family $E(\und{a},\und{a}')$ of subsets of $J_{R}$, $(\und{a},\und{a}')\in \Sigma(\rho)\times \Sigma'(\rho)$, such that
\[m(J_R\setminus E(\und{a},\und{a}'))\leq c_1, \forall (\und{a},\und{a}'),\] 
there is another family $E^*(\und{a},\und{a}')$ of subsets of $J_R$ satisfying:
\begin{itemize}
\item[(i)] For any $(\und{a},\und{a}')$, $E^*(\und{a},\und{a}')$ is contained in the $c_2 \rho$-neighborhood of $E(\und{a},\und{a}')$.
\item[(ii)] Let $(\und{a},\und{a}')\in \Sigma(\rho)\times \Sigma'(\rho)$, $s\in E^*(\und{a},\und{a}')$; there exist at least $c_3\rho ^{-(d+d')}$ pairs $(\und{b},\und{b}')\in \Sigma(\rho)\times \Sigma'(\rho)$ (with $\und{b}$, $\und{b}'$ starting with the last letter of $\und{a}$, $\und{a}'$) such that, if $\tb \in \Sigma^{-}$, $\tb' \in \Sigma^{\prime -}$ end respectively with $\und{a}$, $\und{a}'$ and
\[T_{\und{b}}T'_{\und{b}'}(\tb,\tb',s)=(\tb\und{b},\tb'\und{b}',\tilde{s}),\]
the $c_4\rho$-neighborhood of $\tilde{s}\in J$ is contained in $E^*(\und{b},\und{b}')$.
\item[(iii)] $m(E^*(\und{a},\und{a}'))\geq m(J_R)/c_2$ for at least half of the $(\und{a},\und{a}')\in \Sigma(\rho)\times \Sigma'(\rho)$.
\end{itemize}
\end{lemma}
\begin{remark}
Here $m$ is the unique measure in $\R\times \T$ giving measure $2\pi$ to $J_{1/2}$ and invariant by translations. We notice that the lemma remains true if we change $m$ by Lebesgue measure in $\C^*\subset \C$. Since all sets are in $J_R$, one just would need to redefine the constants $c_1$ and $c_2$. The same happens for the metric on $J$, we prove the lemma with the metric
\[d((t,v),(t',v'))=\max\{|t-t'|,\|v-v'\|\}\]
and $\|[x]\|=\min_{n\in \Z}|x-2\pi n|$ ($[x]\in \T=\R/(2\pi\Z)$ is the class generated by $x\in\R$). In $J_R$ we can change this metric for the usual metric in $\C$.

We remark that it can be assumed the sets $E^*(\und{a},\und{a}')$ are closed. To do this we just have to redefine $E^*(\und{a},\und{a}')$ by taking their closure and increase the parameter $c_2$.
\end{remark}
\begin{proof}
We are in the hypothesis of the scale recurrence lemma in \cite{MZ}. Let $\tilde{r}$, $\tilde{c_0}$, $\tilde{c_1}$, $\tilde{c_2}$, $\tilde{c_3}$ and $\tilde{\rho}_0$ be the constants given by the lemma. We choose $R$ big enough and $c_1$ small enough such that 
$$\nu(\tilde{J}_{\tilde{r}}\setminus L^{-1}(\phi_{\tb,\tb'}^{-1}(E)))<\tilde{c}_1$$
for any set $E\subset J_{R}$ with $m(J_R\setminus E)<c_1$, where $\tilde{J}_{\tilde{r}}=\{(t,v,v'):|t|\leq \ti{r}\}$ and $\nu$ is the Haar measure in $\R\times \T^2$ such that $\nu(\ti{J}_{1/2})=1$. This can be done since $\phi_{\tb,\tb'}$ is multiplication by a complex number, whose norm is uniformly bounded and away from zero. Indeed, if one chooses $R$ big such that $J_{\tilde{r}}\subset \phi_{\tb,\tb'}^{-1}(J_R)$ and $c_1$ small enough such that $m(\phi_{\tb,\tb'}^{-1}(J_R\setminus E))<\tilde{c}_1$, then using that $L_*\nu=m$ one gets
\[\nu(\tilde{J}_{\tilde{r}}\setminus L^{-1}(\phi_{\tb,\tb'}^{-1}(E)))=\nu(L^{-1}(J_{\tilde{r}}\setminus \phi_{\tb,\tb'}^{-1}(E)))=m(J_{\tilde{r}}\setminus \phi_{\tb,\tb'}^{-1}(E))\leq m(\phi_{\tb,\tb'}^{-1}(J_R\setminus E))<\tilde{c}_1.\]
We choose $c_0=\tilde{c}_0$, $c_3=\ti{c}_3$ and $\rho_0=\ti{\rho}_0$, the other constants will be chosen along the proof. Suppose we are given a family of sets $E(\und{a},\und{a}')$ as in the setting of the lemma, define a new family $\tilde{E}(\und{a},\und{a}')$ by
\[\tilde{E}(\und{a},\und{a}')= \bigcup_{\tb,\tb'} L^{-1}(\phi_{\tb,\tb'}^{-1}(E(\und{a},\und{a}')))\cap \tilde{J}_{\tilde{r}},\]
where the union is over all $\tb,\, \tb'$ finishing in $\und{a},\,\und{a}'$, respectively. Notice that thanks to the previous discussion one gets that $\nu(\tilde{J}_{\tilde{r}}\setminus \tilde{E}(\und{a},\und{a}'))<\tilde{c}_1$. Then we can apply the scale recurrence lemma from \cite{MZ} to get a new family $\tilde{E}^*(\und{a},\und{a}')$, which satisfies the properties given in \cite{MZ}. Now we go back to the space $J_R$, define a new family $E^*(\und{a},\und{a}')$ by
\[E^*(\und{a},\und{a}')= \bigcup_{\ttb,\ttb'} \phi_{\ttb,\ttb'}(L(\ti{E}^*(\und{a},\und{a}'))),\]
where the union is over all pairs ending in $\und{a}$, $\und{a}'$ respectively. We will prove that the family $E^*(\und{a},\und{a}')$ satisfies the desired properties:
\begin{itemize}
\item[(i)] Note that $L$ is Lipschitz with constant $2$. Enlarging the constants $c$, $C$ we can suppose that $\phi_{\ttb,\ttb'}$, $\phi_{\ttb,\ttb'}^{-1}$ are Lipschitz with constant $c$ and $|\phi_{\ttb,\ttb'}\circ \phi_{\tb,\tb'}^{-1} - Id|\leq \frac{C}{2}[d(\tb,\ttb)+d(\tb',\ttb')]$. Using this we get
\begin{align*}
E^*(\und{a},\und{a}')&= \bigcup_{\ttb,\ttb'} \phi_{\ttb,\ttb'}(L(\ti{E}^*(\und{a},\und{a}')))\\
&\subset \bigcup_{\ttb,\ttb'} \phi_{\ttb,\ttb'}(L(V_{\tilde{c_2}\rho}(\ti{E}(\und{a},\und{a}'))))\\
&\subset \bigcup_{\ttb,\ttb'}V_{2c\tilde{c_2}\rho}\left(\phi_{\ttb,\ttb'}\circ L \left(\bigcup_{\tb,\tb'}L^{-1}\circ \phi_{\tb,\tb'}^{-1}(E(\und{a},\und{a}'))\right)\right)\\
&= \bigcup_{\ttb,\ttb'}V_{2c\tilde{c_2}\rho}\left( \bigcup_{\tb,\tb'} \phi_{\ttb,\ttb'}\circ \phi_{\tb,\tb'}^{-1} \left(E(\und{a},\und{a}')\right)\right)\\
&\subset \bigcup_{\ttb,\ttb'}V_{2c\tilde{c_2}\rho}\left( V_{Rc_0C\rho} \left(E(\und{a},\und{a}')\right)\right)=V_{(2c\tilde{c}_2+Rc_0C)\rho}(E(\und{a},\und{a}')).
\end{align*}
Taking $c_2>2c\tilde{c}_2+Rc_0C$ gives the desired property.

\item[(ii)] Let $s\in E^*(\und{a},\und{a}')$, then $s\in \phi_{\ttb,\ttb'}(L(\ti{E}^*(\und{a},\und{a}')))$ for some $(\ttb,\ttb')$ ending in $(\und{a},\und{a}')$. Let $\tilde{s}\in \ti{E}^*(\und{a},\und{a}')$ such that $s=\phi_{\ttb,\ttb'}(L(\ti{s}))$. Let $(\und{b},\und{b}')$ be one of the $\tilde{c}_3\rho^{-(d+d')}$ pairs, associated to $\ti{s}$, given by the scale recurrence lemma in \cite{MZ}. If we write
\[T_{\und{b}}T'_{\und{b}'}(\ttb,\ttb',s)=T_{\und{b}}T'_{\und{b}'}\circ \phi(\ttb,\ttb',\ti{s})=\phi\circ T_{\und{b},\und{b}'}(\ttb,\ttb',\ti{s})=(\ttb\und{b},\ttb'\und{b}',\phi_{\ttb\und{b},\ttb'\und{b}'}(L(s^*))), \]
we know that the ball $B(s^*,\rho)$ is contained in $\ti{E}^*(\und{b},\und{b}')$. This implies
\begin{align*}
B(\phi_{\ttb\und{b},\ttb'\und{b}'}(L(s^*)),c^{-1}\rho)&\subset \phi_{\ttb\und{b},\ttb'\und{b}'}(B(L(s^*),\rho))\\
&\subset \phi_{\ttb\und{b},\ttb'\und{b}'}\circ L (B(s^*,\rho))\subset \phi_{\ttb\und{b},\ttb'\und{b}'}\circ L (\ti{E}^*(\und{b},\und{b}'))\subset E^*(\und{b},\und{b}').
\end{align*}
Thus it is enough to take $c_4<c^{-1}$.
\item[(iii)] Let $(\und{a},\und{a}')$ such that $\nu(\ti{E}^*(\und{a},\und{a}'))\geq \nu(\tilde{J}_{\tilde{r}})/2$ and $(\ttb,\ttb')$ ending in $(\und{a},\und{a}')$, we have
\begin{align*}
m(E^*(\und{a},\und{a}'))\geq m(\phi_{\ttb,\ttb'}(L(\ti{E}^*(\und{a},\und{a}'))))&\gtrsim m(L(\ti{E}^*(\und{a},\und{a}')))\\
&= \nu(L^{-1}(L(\ti{E}^*(\und{a},\und{a}'))))\\
&\geq \nu(\ti{E}^*(\und{a},\und{a}'))\geq \nu(\tilde{J}_{\tilde{r}})/2.
\end{align*}
The fact that $\nu(\ti{E}^*(\und{a},\und{a}'))\geq \nu(\tilde{J}_{\tilde{r}})/2$ for at least half of the $(\und{a},\und{a}')$ implies immediatly the desired property with $c_2$ big enough.
\end{itemize}
\end{proof}

\section{Random Perturbations of conformal Cantor sets}\label{sec:randpert}

%As a consequence of this discussion we have the following lemma:

%\begin{lemma} Arbitrarily close to a conformal Cantor set $(K,g)$ we can construct another conformal Cantor set $(\ti{K},\ti{g})$ with the following properties:
%\begin{itemize}
 %   \item Every point $x \in K$ is accumulated by points in $K$ through at least two different directions, that is, $K^{dir}_x$ has two linear independent vectors.  
 %   \item $\ti{g}$ is holomorphic on a neighbourhood of $\ti{K}$.
%\end{itemize}
%\end{lemma}

%We will call a Cantor set with these properties \emph{tame}.

 \subsection{Random perturbations}

From now on we will focus on the proof of theorem \ref{thm:main2}. Let $(K,K')$ be a pair of non-essentially real $C^{\infty}$ conformal regular Cantor sets such that $HD(K)+HD(K')>2$. We first perturb, in the $C^{\infty}$ topology, the pair of Cantor sets $(K,K')$ so they satisfy the hypothesis of the Scale Recurrence Lemma and the map $g$ defining the Cantor set $K$ is holomorphic on a neighborhood $V$ of $K$. All this can be done thanks to lemma \ref{lemma:c8approx}. Applying now the scale recurrence lemma gives constants $R,\, c_0,\, c_1,\, c_2,\, c_3,\, c_4$ verifying the conclusions of the lemma. With the aim of reducing the number of constants, we will also assume, without loss of generality, that the diameters of the sets $G^{\tb}(\theta_0)=k^{\tb}(G(\theta_0))$ are all less than one. This can be achieved by changing the metric. To prove theorem \ref{thm:main2}, we will now only perturb the Cantor set $K$, leaving $K'$ unaltered. 

Notice that a neighbourhood in $\Omega^{\infty}_{\Sigma}$ contains a neighbourhood in $\Omega^{k}_{\Sigma}$ for some integer $k \ge 2$. So from now on we fix this integer $k$. The desired $C^k$ perturbation for $g$ will be picked by a probabilistic argument out of a family of random perturbations that we will now construct.

The following constructions and arguments are made having a parameter $\rho > 0 $ in mind. All constants from now on are independent of this parameter, and everything fits together in the end  by choosing $\rho$ sufficiently small.
 
We first pick a subset $\Sigma_0$ of $\Sigma(\rho^{1/k})$ such that
$$
K = \bigcup_{\und{a}\in\Sigma_0} K(\und{a})
$$
is a \emph{partition} of $K$ into disjoint cylinders.

We then define $\Sigma_1$ as the subset of $\Sigma_0$ formed of the words
$\und{a} \in \Sigma_0$ such that no word in $\Sigma(\rho^{1/3k})$ appears twice
in $\und{a}$. To see that $\Sigma_1 \neq \emptyset$ check the claim in the proof of lemma \ref{lemma:c8approx}. 

Let $\ti{c}_4 > 0$ be a constant\footnote{This constant corresponds to $c_4$ in \cite{MY}, since we already used this symbol in the scale recurrence lemma then we changed it to $\ti{c}_4$.} sufficiently close to $0$ to have the following: let \[ 
\hat{G}(\und{a}) \coloneqq V_{\ti{c}_4 \cdot \text{diam}(G(\und{a}))}(G(\und{a})),
\]
for $\und{a} \in \Sigma_0$; then the $\hat{G}(\und{a})$, $\und{a} \in \Sigma_0$, are pairwise disjoint.

For each $\und{a} \in \Sigma_0$ we choose a smooth function $\chi_{\und{a}}\colon \C\to\R$
satisfying:
\begin{align*}
&\chi_{\und{a}}(z) = 1 \quad\text{for}\quad z \in V_{\frac{\ti{c}_4}{2} \cdot \text{diam}(G(\und{a}))}(G(\und{a})) \,,\\
&\chi_{\und{a}}(z) = 0 \quad\text{for}\quad z \not\in \hat{ G}(\und{a})\,.
\end{align*}

Notice that, since $\und{a} \in \Sigma(\rho^{1/k})$, we can choose these functions in a way that $\|D^j\chi_{\und{a}}\|\leq \ti{C}\,\rho^{-j/k}$, for all $0\le j\le k$ and $\ti{C}$ some constant independent of $\rho$ (but not from $\ti{c}_4$).

The probability space underlying the family of random perturbations is $\Om =
\D^{\Sigma_1}$, where $\D$ is the unitary disk in $\C$, equipped with the normalized Lebesgue measure.

For $\uom = (\om(\und{a}))_{\und{a}\in\Sigma_1} \in \Om$, we define
$\Phi_{\uom}$ to be the time-one map of the vector field 

\[X_{\uom}(z) = -c_5\,\rho^{(1+1/2k)}\,\sum\limits_{\und{a}} \chi_{\und{a}}(z)\,\om(\und{a}), \]

where $c_5 > 0$ is a conveniently large constant, to be chosen later. Finally, we define $g^{\uom}$ to be $g\circ \Phi_{\uom}$. 

By our previous estimative on $\|D^j\chi_{\und{a}}\|$ we have that $\|\Phi_{\uom}-Id\|_{C^k}$ is $O(\rho^{1/2k})$.

%Notice that $\|D^j\chi_{\und{a}}\|\leq \ti{C}\,\rho^{\frac{-j}{k}}$ since $\und{a} \in \Sigma(\rho^{\frac{1}{k}})$ 

%Let us discuss some basic properties of the maps $g^{\uom}$, $\uom \in \Om$. 

%For $\und{a} \in \Sigma_1$\,, we define the vector field $X_{\und{a}}$\,, with support $\subset\, \widehat G(\und{a})$ by
%$$
%X_{\und{a}}(z) = c_5\,\rho^{(1+1/2k)}\, \frac{\po}{\po z}
%$$
%where $c_5 > 0$ is a conveniently large constant, to be chosen later, and $B$is an affine map sending $I(\und{a})$ onto $[-1,+1]$; there are two such maps, but as $\chi$ is even, they give the same $X_{\und{a}}$\,.

Since $\Phi_{\uom}$\,, for any $\uom \in \Om$, is close to the identity in
the $C^k$-topology, then $g^{\uom}$ is close to $g$. Taking $\rho$ small enough we can suppose that $g^{\uom}$ generates a Cantor set  (with the same family of sets $G(a)$, $a\in \mathbb{A}$), which we denote by $K^{\uom}$. Moreover, taking $\rho$ small it can be proven that this Cantor set is in fact a conformal Cantor set. Indeed, let $V$ be the open set containing $K$ where the function $g$ is holomorphic. If $\rho$ is sufficiently small then\footnote{This is consequence of lemma \ref{lem:plg} part (ii), note that the proof of this part of the lemma does not use the conformality of $g$ at the points in the Cantor set. We can also get this from the fact that $K$ is an hyperbolic set for $g$ and use continuation of the hyperbolic set (see Theorem 7.8 in \cite{S}).}, for any $x \in K^{\uom}(\und{a})$ and $\und{a} \in \Sigma_0$, $\Phi_{\uom}(x) \in V $ and $x\in V_{\frac{\ti{c}_4}{2} \cdot \text{diam}(G(\und{a}))}(G(\und{a}))$.
It follows that $$D(g\circ \Phi_{\uom})(x)=Dg(\Phi_{\uom}(x))\cdot D\Phi_{\uom}(x)=Dg(\Phi_{\uom}(x))$$ which is a conformal linear transformation.

Our task will be to find $\uom \in \Om$ such that the pair of Cantor sets determined by $(g^{\uom},g')$ have a non empty recurrent compact set of relative configurations.

\begin{remark}\label{homeomorphism}
 All the objects introduced in section \ref{sec:def} are well defined for the Cantor sets $K^{\uom}$ and we will denote them by adding a superscript indicating the corresponding value of $\uom \in \Omega$,  such as $G^{\uom}(\und{a})$, $k^{\ute, \uom}$, $c^{\ute, \uom}_{\und{a}}$ and $F^{\ute, \uom}_{\und{a}}$ for example. Notice however that these Cantor sets have the same type as $K$, and therefore are close to $K$ in the $C^k$ topology. Besides, we consider for each $\uom \in \Omega$ the natural conjugation between the dynamical systems $(K^{\uom},g^{\uom}|_K)$ and $(\Sigma, \sigma)$
\begin{equation*}
    H^{\uom}: K^{\uom} \to  \Sigma,
\end{equation*}
which carries each point $x \in K^{\uom}$ to the sequence $\{a_n\}_{n \geq 0}$ that satisfies $(g^{\uom})^n(x) \in G(a_n)$.
For each $a \in \mathbb{A}$, we have a pre-periodic sequence $\und{x}_a \in \Sigma$ that begins with $a$, defined by $\und{x}_a \coloneqq H(c_a)$. The set of \emph{base points} $c^{\uom}_a \in G(a)$ for $a \in \mathbb{A}$ satisfies 
\[c^{\uom}_a = (H^{\uom})^{-1} (\und{x}_a)\]
for every $\uom \in \Omega$. This is important for the study of limit geometries.
\end{remark}

\subsection{Some properties of the family $g^{\uom}$}

Let $\und{a}' \in \Sigma_0$ and $a_{-1} \in \A$ be such that $(a_{-1},a_0) \in B$ and $\und{a}'$ begins with $(a_{-1},a_0)$; let $ (a_{-1},a_0)\und{a}=\und{a}'$. Any perturbed inverse branch $f^{\uom}_{a_{-1},a_0}$ is well defined in the neighborhood $V_{\rho}(G(\und{a}))$ and for any $x \in V_{\rho}(G(\und{a}))$
 
\begin{equation}\label{f}
f^{\uom}_{a_{-1},a_0}(x) =
\begin{cases}
f_{a_{-1},a_0}(x) \quad\text{if}\quad \und{a}' \in \Sigma_0 \setminus \Sigma_1,\\
f_{a_{-1},a_0}(x) + c_5\rho^{1+1/2k}\,\om(\und{a}') \text{ if } \und{a}' \in\Sigma_1.
\end{cases}
\end{equation}
Notice that $\|\Phi_{\om}-Id\|_{C^0}=O(\rho^{1+1/2k})$, therefore 
\[V_{\rho}(G(\und{a}'))\subset \Phi_{\uom}(V_{\frac{\ti{c}_4}{4}diam(G(\und{a}'))}(G(\und{a}')))\subset V_{\frac{\ti{c}_4}{2}diam(G(\und{a}'))}(G(\und{a}')),\]
for $\rho$ small enough. This implies that 
\[V_{\rho}(G(\und{a})) \subset  g(V_{\rho}(G(\und{a}'))) \subset g^{\uom}(V_{\frac{\ti{c}_4}{4}diam(G(\und{a}'))}(G(\und{a}')))\]
and 
$$
g^{\uom}(z) =
\begin{cases}
g(z) \quad\text{if}\quad \und{a}' \in \Sigma_0-\Sigma_1,\\
g(z - c_5\rho^{1+1/2k}\,\om(\und{a}')) \text{ if } \und{a}' \in\Sigma_1,
\end{cases}
$$
for all $z\in V_{\frac{\ti{c}_4}{4}diam(G(\und{a}'))}(G(\und{a}'))$. This in turn immediately implies the formula for $f^{\uom}_{a_{-1},a_0}$.

\begin{lemma}\label{lem:plg}
Let $\uom \in \Om$ and $H^{\uom}: K^{\uom}  \to \Sigma $ be the homeomorphism defined in remark \ref{homeomorphism}. If $\rho$ is sufficiently small,
\begin{itemize}
\item[(i)] for any $\und{a} \in \Sigma^{fin}$, we have $\lV f^{\uom}_{\und{a}} - f_{\und{a}}\rV_{C^0} \le c_{18} c_5\ro^{1+\frac{1}{2k}}$;
\item[(ii)] for any $\und{a} \in \Sigma$, we have $\lv (H^{\uom})^{-1}(\und{a}) - H^{-1}(\und{a}) \rv \le c_{18}c_5\,\ro^{1+\frac{1}{2k}}$;
\item[(iii)] for $\ute \in \Sigma^-$, we have
$$
\lV k^{\ute,\uom} - k^{\ute}\rV_{C^0}
\le c_{18}c_5\,\ro^{1-\frac{1}{2k}};
$$
\item[(iv)] for $\ute \in \Sigma^-$ and a word $\und{a} = (a_0, \,a_1,\,\dots,\,a_m)$ with $a_0=\te_0$ such that $diam(G(\und{a})) >
c_0^{-1}\,\ro$, we have
\begin{align*}
\lv  \frac{ DF^{\ute}_{\und{a}}}{DF^{\ute, \, \uom}_{\und{a}}} - 1 \rv & \le c_{18}c_5\,\ro^{1-\frac{1}{2k}}; \\
\left| \log
r^{\ute,\uom}_{\und{a}}-\log r^{\ute}_{\und{a}}\right|
& \le c_{18}c_5\,\ro^{1-\frac{1}{2k}}.
\end{align*}
\end{itemize}
The constant $c_{18}$ is independent of $\ute$, $\uom$, $\und{a}$,
$\ro$, and the size $c_5$ of the perturbation.
\end{lemma}

\begin{proof} 
(i): Let $\displaystyle{x_n=\max_{|\und{a}|=n}{\lV f^{\uom}_{\und{a}}-f_{\und{a}}\rV_{C^0}}}$ be the maximum distance between corresponding inverse branches of $g^n$ and $(g^{\uom})^n$. We will prove that $x_n \le c_{18}c_5\,\ro^{1+\frac{1}{2k}} $ by induction on $n$. For $n=1$ we have $x_1\le c_5\rho^{1+1/2k}$, this is a direct consequence of $\|\Phi_{\uom}-Id\|\leq c_5\rho^{1+1/2k}$. 

Observe that $g(G(a))$ covers all the pieces $G(b)$ it intersects, therefore there exists $\delta> 0$ sufficiently small such that if $(a,b) \in B$, then $V_{\delta}(G(b)) \subset g(G(a)) $. Consequently, if $x \in G(b)$, any point $x'$ such that $\lv x-x'\rv < \delta$ \emph{and the line segment joining} $x$ and $x'$ are contained in the extended domain $V_{\delta}(G(b))$ of $f_{(a,b)} = (g|_{G(a)})^{-1}|_{V_{\delta}(G(b))}$. In this domain, $\lV Df_{(a,b)} \rV \le \mu^{-1}$. Suppose $x_n \le c_{18} c_5\ro^{1+\frac{1}{2k}}$. If $\rho $ is sufficiently small, then $x_n < \delta$. Given a word $\und{b}=(b_0, b_1, \dots, b_n, b_{n+1})$, we write $\und{b}'=(b_1,b_2,\dots,b_{n+1})$. Given a point $x \in G(b_{n+1})$, 
\begin{align*}
   f^{\uom}_{\und{b}}(x)-f_{\und{b}}(x)=&f^{\uom}_{(b_0,b_1)}(f^{\uom}_{\und{b}'}(x))-f_{(b_0,b_1)}(f_{\und{b}'}(x))=\\
    &f_{(b_0,b_1)}(f^{\uom}_{\und{b}'}(x))-f_{(b_0,b_1)}(f_{\und{b}'}(x))+f^{\uom}_{(b_0,b_1)}(f^{\uom}_{\und{b}'}(x))-f_{(b_0,b_1)}(f^{\uom}_{\und{b}'}(x)).
\end{align*}
Of course, when writing this, we are assuming that $f^{\uom}_{\und{b}'}(x)$ belongs to the domain $V_{\delta}(G(b_1))$ of $f_{(b_0,b_1)}$. But this is the case when $\lv f^{\uom}_{\und{b}'}(x)-f_{\und{b}'}(x) \rv \le x_n < \delta$, which is true by hypothesis. More than that, because the segment joining the two points is inside this domain,
\[|f^{\uom}_{\und{b}}(x)-f_{\und{b}}(x)|\le \mu^{-1}|f^{\uom}_{\und{b}'}(x)-f_{\und{b}'}(x)|+c_5\ro^{1+\frac{1}{2k}} \le \mu^{-1} x_n + c_5\ro^{1+\frac{1}{2k}}  .\]

In this manner, choosing $c_{18} \ge \frac{1}{1 - \mu^{-1}}$, we obtain $x_{n+1}\le c_{18}c_5\ro^{1+\frac{1}{2k}}$, finishing this part.

(ii): Let $\und{a} = (a_0, \, a_1,\,\dots) \in \Sigma$. It follows that $H^{-1}(\und{a}) = \lim_{n \rightarrow \infty }{f_{\und{a}_n}(G(a_n))}$ and $(H^{\uom})^{-1}(\und{a}) = \lim_{n \rightarrow \infty }{f^{\uom}_{\und{a}_n}(G(a_n))}$. As the diameters of these sets converge exponentially to zero, the result follows from (i).

(iii): We now study the perturbed limit geometries. Notice that the base point used to define $k^{\ute,\uom}$ is not the same as the one for $k^{\ute}$, but the estimate of (ii) gives us control over this displacement. 

Fix $\ute\in \Sigma^{-}$ and let $z \in G(\te_{0})=G^{\uom}(\te_0)$. Let the base point $c_{\te_0}\in K(\te_0)$ be given by $c_{\te_0}=H^{-1}(\und{x}_a)$ and the base point $c^{\uom}_{\te_0}\in K^{\uom}(\te_0)$ be given by $c^{\uom}_{\te_0}=(H^{\uom})^{-1}(\und{x}_a)$ for some fixed sequence $\und{x}_a \in \Sigma$. From (ii), \[\lv c^{\uom}_{\te_0} - c_{\te_0}  \rv \le c_{18}c_5\ro^{1+\frac{1}{2k}}.\] Write $c_n=f_{\ute_n}(c_{\te_0})$ and $z_n=f_{\ute_n}(z)$; and $c^{\uom}_n = f^{\uom}_{\ute_n}(c^{\uom}_{\te_0})$ and $z^{\uom}_n = f^{\uom}_{\ute_n}(z)$ for $n \ge 1$. Notice that \[\lv z^{\uom}_n - z_n \rv = \lv (f^{\uom}_{\ute_n} - f_{\ute_n}) (z) \rv \le c_{18}c_5\ro^{1+\frac{1}{2k}} \] by (i). Likewise, seeing that $c^{\uom}_{\te_0}  \in V_{\delta}(G(\te_0)) $,
\begin{equation}\label{bases}
    \lv  c^{\uom}_n - c_n \rv \le \lv (f^{\uom}_{\ute_n} - f_{\ute_n})(c^{\uom}_{\te_0}) \rv + \lv f_{\ute_n}(c^{\uom}_{\te_0}) - f_{\ute_n}( c_{\te_0}) \rv \lesssim c_5\ro^{1+\frac{1}{2k}},
\end{equation}
by (i) again and the estimate for $\lv c^{\uom}_{\te_0} - c_{\te_0} \rv$ above ($f_{\ute_n}$ is a contraction).

Remember that $k^{\ute}_n=\Phi_{\ute_n} \circ f_{{\ute_n}}$, where $\Phi_{\ute_n}$ is an affine transformation, and $k^{\ute}=\lim_{n \rightarrow \infty}{k^{\ute}_n}$. Hence  
$$k^{\ute}_n(z)=\Phi_{\ute_n}(z_n)-\Phi_{\ute_n}(c_n)= \left(Df_{\ute_n}(c_{\te_0})\right)^{-1}(z_n-c_n)$$
and similarly 
$$k^{\ute, \uom}_n(z)=\Phi^{\uom}_{\ute_n}(z^{\uom}_n)-\Phi^{\uom}_{\ute_n}(c^{\uom}_n)=\left(Df^{\uom}_{\ute_n}(c^{\uom}_{\te_0})\right)^{-1}(z^{\uom}_n-c^{\uom}_n).$$
The difference $  k^{\ute, \uom}_n(z)  -  k^{\ute}_n(z) $ is thus equal to 
\begin{equation}\label{expression}
    \left(Df_{\ute_n}(c_{\te_0})\right)^{-1}(z^{\uom}_n - z_n + c_n - c^{\uom}_n)+\left[\left(Df^{\uom}_{\ute_n}(c^{\uom}_{\te_0})\right)^{-1}-\left(Df_{\ute_n}(c_{\te_0})\right)^{-1}\right](z^{\uom}_n-c^{\uom}_n).
\end{equation}
Let us analyze this expression for $n$ not very large. Define  
\begin{align*}
    A_n & \coloneqq \left(Df_{\ute_n}(c_{\te_0})\right)^{-1}=Dg^n(c_n),\\
    B_n & \coloneqq \left(Df^{\uom}_{\ute_n}(c^{\uom}_{\te_0})\right)^{-1}=D(g^{\uom})^n(c^{\uom}_n),
\end{align*} 
for $n \ge 1$. If $n$ is such that $\lv A_n \rv \le c_0c^{\prime}_0\ro^{-1/k}$ (remember that $c_0'$ was defined in eq. \eqref{eq:c0prima}), then, by the previous estimates, the first term of \eqref{expression} is $\lesssim \ro^{-1/k}c_5\ro^{1+1/2k}=c_5\ro^{1-1/2k}$. On the other hand, for every $m \ge 0$,
$$          
A_{m+1}=A_m \cdot Dg(c_{m+1}) \qquad \text{and} \qquad B_{m+1}=B_m \cdot Dg^{\uom}(c^{\uom}_{m+1}),            
$$
therefore 
\[          A_{m+1}-B_{m+1}=(A_m-B_m)\cdot Dg(c_{m+1})+B_m \cdot (Dg(c_{m+1})-Dg^{\uom}(c^{\uom}_{m+1})),\]
from which one can deduce, by induction on $n$, that for any $m,\,n \ge 0$, 
 \begin{align}\label{c_n}
 \begin{split}
 (A_{m+n}-B_{m+n}) - &(A_m-B_m)\cdot Dg^{n}(c_{m+n})  \\
   &  = \sum_{j=0}^{n-1}B_{m+j} \cdot \left(Dg(c_{m+j+1}) - Dg^{\uom}(c^{\uom}_{m+j+1})\right) \cdot Dg^{n-1-j}(c_{m+n})\\
  &  =  \sum_{j=0}^{n-1}B_{m+j} \cdot \left(Dg(c_{m+j+1}) - Dg^{\uom}(c^{\uom}_{m+j+1})  \right)\cdot{A_{m+j+1}}^{-1} \cdot A_{m+ n }.
\end{split}
 \end{align}
By \eqref{bases}, the fact that $D\Phi_{\uom}(c^{\uom}_{m+j+1}) = Id$, and the fact that the maps are $C^\infty$, \[\lv Dg(c_{m+j+1}) - Dg^{\uom}(c^{\uom}_{m+j+1}) \rv = \lv Dg(c_{m+j+1}) - Dg(\Phi_{\uom}(c^{\uom}_{m+j+1}))\cdot D\Phi_{\uom}(c^{\uom}_{m+j+1}) \rv \lesssim c_5\rho^{1+\frac{1}{2k}}.\] Now write $C_n \coloneqq -  (A_n-B_n)\cdot (A_n)^{-1}$. It follows that $B_n = (C_n + Id) \cdot A_n$. Then, making $m=0$ in \eqref{c_n} and dividing it by $\lv A_n \rv$, we get that
\begin{equation}\label{quotacn}
    \lv C_n \rv \lesssim  \sum_{j=0}^{n-1} c_5\rho^{1+\frac{1}{2k}} \cdot |Id + C_{j}|\cdot|A_{j}|\cdot|{A_{j+1}}|^{-1} \lesssim   c_5\rho^{1+\frac{1}{2k}}\cdot \sum_{j=0}^{n-1}{1+|C_j|}.
\end{equation}

Let $m_0$ be the largest value such that $\lv A_{m_0} \rv \le c_0 c^\prime_0 \rho^{-1/k}$. Thus $m_0 \lesssim -\frac{1}{k}\log{\ro}$ and, again by induction on $n$, $|C_n| \lesssim  1$ for $n \le m_0$. Indeed, if it is true for all $j \le n-1$, then
\[
|C_n| \lesssim c_5\rho^{1+\frac{1}{2k}}\cdot \sum_{j=0}^{n-1}{1+|C_j|} \lesssim c_5\rho^{1+\frac{1}{2k}}\cdot n \le m_0 c_5\rho^{1+\frac{1}{2k}} \lesssim -\frac{1}{k}\log{\ro}\cdot c_5\rho^{1+\frac{1}{2k}} \lesssim 1,
\]

if $\rho$ is sufficiently small. Plugging this estimate for $|C_j|$ with $0 \le j \le n-1$ in \eqref{quotacn} again yields $|C_n| \lesssim  c_5 \ro^{1-\frac{1}{2k}}$ for $n \le m_0$ if $\rho$ is sufficiently small. We also know that $|z_n-c_n|\lesssim |A_n|^{-1}$ and hence $|z^{\uom}_n-c^{\uom}_n| \lesssim |A_n|^{-1}$ for $n \le m_0$. Hence the second term in \eqref{expression} is $ \lesssim |C_n| \lesssim c_5 \ro^{1-\frac{1}{2k}}$. 

We are left with controlling the difference $k^{\ute,\uom}_n(z)-k^{\ute}_n(z)$ for $n > m_0$. Notice that if $\lv A_n \rv > c_0 c^{\prime}_0\ro^{-\frac{1}{k}}$, then the four points $z_n$, $c_n$, $z^{\uom}_n$ and $c^{\uom}_n$ belong to the same piece $G(\und{a})$ where $\und{a} \in \Sigma_0$. Thus, 
\[z^{\uom}_{n+1}-c^{\uom}_{n+1}=f^{\uom}_{(\te_{-n-1},\te_{-n})}(z^{\uom}_n)-f^{\uom}_{(\te_{-n-1},\te_{-n})}(c^{\uom}_n)=f_{(\te_{-n-1},\te_{-n})}(z^{\uom}_n)-f_{(\te_{-n-1},\te_{-n})}(c^{\uom}_n).\]
This way, if $\ro$ is small enough so that the segments joining $z^{\uom}_n$ to $c^{\uom}_n$ and $z_n$ to $c_n$ belong to the domain of $f_{(\te_{-n-1},\te_{-n})}$,
    \begin{align*}
        k^{\ute,\uom}_{n+1}(z)&=B_{n+1}\cdot \int_{0}^1{Df_{(\te_{-n-1},\te_{-n})}(c^{\uom}_n+(z^{\uom}_n-c^{\uom}_{n})t)dt}\cdot(z^{\uom}_n-c^{\uom}_{n}) \\
        &=B_{n+1}\cdot \int_{0}^1{Df_{(\te_{-n-1},\te_{-n})}(c^{\uom}_n+(z^{\uom}_n-c^{\uom}_{n})t)dt}\cdot B_{n}^{-1}\cdot k^{\ute,\uom}_n(z),\\
        k^{\ute}_{n+1}(z)&=A_{n+1}\cdot\int_{0}^1{Df_{(\te_{-n-1},\te_{-n})}(c_n+(z_n-c_{n})t)dt}\cdot A_{n}^{-1}\cdot k^{\ute}_n(z).
  \end{align*}
Write \begin{align*}
 I_n & \coloneqq   \int_{0}^1{Df_{(\te_{-n-1},\te_{-n})}(c_n+(z_n-c_{n})t)dt}, \\
 I^{\uom}_n & \coloneqq \int_{0}^1{Df_{(\te_{-n-1},\te_{-n})}(c^{\uom}_n+(z^{\uom}_n-c^{\uom}_{n})t)dt}.
\end{align*}
Notice that $I_n$ and $I^{\uom}_n$ are both conformal matrices. This happens because $c_n+(z_n-c_{n})t$ and $c^{\uom}_n+(z^{\uom}_n-c^{\uom}_{n})t$ belong, for every $t \in [0,1]$, to the domain in which $f_{(\te_{-n-1},\te_{-n})}$ is holomorphic, provided $\rho$ is sufficiently small. Besides, the difference between these two integrals is $\lesssim c_5\rho^{1+\frac{1}{2k}}$, because $f$ is $C^{\infty}$ and 
\[
\lv (1-t) (c^{\uom}_n- c_n) + t (z^{\uom}_n-z^{\uom}_{n}) \rv \lesssim c_5\rho^{1+\frac{1}{2k}}.
\] 
Furthermore, $Df_{(\te_{-n-1},\te_{-n})}(c^{\uom}_n)  = Df^{\uom}_{(\te_{-n-1},\te_{-n})}(c^{\uom}_n) $ by \eqref{f} and so
\[ 
\lv I^{\uom}_n - Df_{(\te_{-n-1},\te_{-n})}(c^{\uom}_n) \rv = \lv I^{\uom}_n - Df^{\uom}_{(\te_{-n-1},\te_{-n})}(c^{\uom}_n) \rv \lesssim |z^{\uom}_n-c^{\uom}_n|\lesssim\lv B_n \rv^{-1}  \text{ and}
\] 
\[
\lv I_n - Df_{(\te_{-n-1},\te_{-n})}(c_n)\rv \lesssim |z_n-c_n| \lesssim \lv A_n \rv^{-1} 
\]
respectively. This implies that there exists some constant $c>0$ independent of $c_5$ such that 
\begin{align*}
    \lv A_{n+1} \cdot I_n \cdot A_n^{-1} - Id \rv & \le  \frac{c}{2}\lv A_n \rv^{-1}, \\
    \lv B_{n+1} \cdot I^{\uom}_n \cdot B_n^{-1} - Id \rv & \le \frac{c}{2}\lv B_n \rv^{-1}\qquad \text{ and } \\
    \lv A_{n+1} \cdot I_n \cdot A_n^{-1} -  B_{n+1}  \cdot I^{\uom}_n \cdot B_n^{-1} \rv & \le \lv I_n (A_n^{-1}  \cdot A_{n+1}  -  B_n^{-1}  \cdot B_{n+1})\rv + \lv (I_n - I^{\uom}_n) B_n^{-1}  \cdot B_{n+1} \rv \\  & \le \lv I_n \rv \cdot \lv Dg(c_{n+1}) - Dg^{\uom}(c^{\uom}_{n+1}) \rv +  \lv I_n - I^{\uom}_n \rv \cdot \lv Dg^{\uom}(c^{\uom}_{n+1}) \rv \\ & \le  c c_5\ro^{1+\frac{1}{2k}},
\end{align*}  
since the matrices $A_n^{-1}$, $B_n^{-1}$, $A_{n+1}$, $B_{n+1}$ and $I_n$ commute (they are all conformal). Therefore, defining $a_n \coloneqq \lv  k^{\ute,\uom}_{n}(z) - k^{\ute}_{n}(z) \rv $, for $n \ge m_0$
\begin{align*}
    a_{n+1} = & \lv  k^{\ute,\uom}_{n+1}(z) - k^{\ute}_{n+1}(z) \rv \\ 
        &  \le \lv A_{n+1} \cdot I_n \cdot A_n^{-1} \cdot  \left(k^{\ute,\uom}_{n}(z) - k^{\ute}_{n}(z)\right) \rv  +  \lv (A_{n+1} \cdot I_n \cdot A_n^{-1} - B_{n+1} \cdot I^{\uom}_n \cdot B_n^{-1}) \cdot  k^{\ute,\uom}_{n}(z)  \rv \\
        &  \le (1+c|A_n|^{-1}) \lv k^{\ute,\uom}_{n}(z) - k^{\ute}_{n}(z) \rv +c\min{\{c_5\ro^{1+\frac{1}{2k}}, \max{\{\lv A_n \rv^{-1} ,\lv B_n \rv^{-1} \}}\}} \\
        & =  (1+c|A_n|^{-1}) a_n +c\min{\{c_5\ro^{1+\frac{1}{2k}}, \max{\{\lv A_n \rv^{-1} ,\lv B_n \rv^{-1} \}}\}}.
\end{align*}

For $n \gtrsim -(1+\frac{1}{2k})\log \ro$ the minimum above is equal to $\max{\{\lv A_n \rv^{-1} ,\lv B_n \rv^{-1} \}}$, which decays exponentially. Up to such a value, the formula above implies that $a_n \lesssim a_{m_0}+ c_5\rho^{1+\frac{1}{2k}}( -\log \ro)$. Using the fact shown before that $a_{m_0}\lesssim c_5\ro^{1-\frac{1}{2k}}$ and choosing $\ro$ sufficiently small, it follows that the sequence $a_n$ is $\lesssim c_5\rho^{1-\frac{1}{2k}}$. Making $n \rightarrow \infty$ we conclude (iii).

(iv) By lemma \ref{simpleformula},
\[
 r^{\ute}_{\und{a}} = \diam (k^{\ute} \circ f_{\und{a}}(G(a_m))) = \diam (F^{\ute}_{\und{a}} \circ k^{\ute}( G(a_m))) = \lv DF^{\ute}_{\und{a}} \rv \diam (k^{\ute}( G(a_m))) 
\]
and the analogous relation is valid for the perturbed version. To show that $ \left| \log r^{\ute,\uom}_{\und{a}}-\log r^{\ute}_{\und{a}}\right| \le c_{18}c_5\,\ro^{1-\frac{1}{2k}}$ it is thus sufficient to prove that 
\[
\lv \frac{\lv DF^{\ute}_{\und{a}} \rv \diam (k^{\ute}( G(a_m)))}{\lv DF^{\ute, \, \uom}_{\und{a}} \rv \diam (k^{\ute,\uom}( G(a_m)))} - 1\rv \lesssim c_5\,\ro^{1-1/2k}.
\]
Notice that by (iii), 
\[
\diam (k^{\ute}( G(a_m))) - \diam (k^{\ute,\uom}( G(a_m))) \lesssim c_5\,\ro^{1-1/2k},
\]
and so, as these diameters are uniformly bounded away from zero, 
\[
\lv \frac{\diam (k^{\ute,\uom}( G(a_m)))  } { \diam (k^{\ute}( G(a_m))) } -1 \rv \lesssim c_5\,\ro^{1-1/2k}.
\]
This way, we are left with analysing the derivatives of the affine maps $F^{\ute}_{\und{a}}$ and $F^{\ute\,\und{a}, \uom}$. Also from lemma \ref{simpleformula},
\begin{align*} 
  DF^{\ute}_{\und{a}}  & = \lim_{n \rightarrow \infty} \left(Df_{\ute_n}(c_{\te_0})\right)^{-1} \cdot Df_{(\ute\und{a})_{n+m}}(c_{a_m}) \\
  DF^{\ute, \, \uom}_{\und{a}}  & = \lim_{n \rightarrow \infty} \left(Df^{\uom}_{\ute_n}(c^{\uom}_{\te_0})\right)^{-1} \cdot Df^{\uom}_{(\ute\und{a})_{n+m}}(c^{\uom}_{a_m}).
\end{align*}

To meet our objectives we need only to show that for all $n \ge 0$
\begin{equation}\label{r}
   \lv \frac{ \left(Df_{\ute_n}(c_{\te_0})\right)^{-1} \cdot Df_{(\ute\und{a})_{n+m}}(c_{a_m})}{\left(Df^{\uom}_{\ute_n}(c^{\uom}_{\te_0})\right)^{-1} \cdot Df^{\uom}_{(\ute\und{a})_{n+m}}(c^{\uom}_{a_m})} - 1 \rv \lesssim c_5\,\ro^{1-1/2k}.
\end{equation}

To each $\tb' \in \Sigma^{-}$ let $ n_0(\ute')$ be the largest integer $n$ such that $\lv Df_{\ute'_n}(c_{{\te'_0}}) \rv \ge c_0^{-1}(c_{0}^{\prime})^{-1}\rho^{1/k}$. The analysis of $C_n$ in (iii) implies that, uniformly on $\ute' \in \Sigma^{-}$, 
\[
\frac{\lv \left(Df^{\uom}_{\ute'_n}(c^{\uom}_{\te'_0})\right)^{-1} - \left(Df_{\ute'_n}(c_{\te'_0})\right)^{-1} \rv}{\lv Df_{\ute'_n}(c_{\te'_0} )\rv^{-1} } \lesssim c_5\,\ro^{1-1/2k}.
\]
for all $n \le n_0(\ute')$. This also implies that for all $n \le n_0(\ute')$
\[
\frac{\lv Df^{\uom}_{\ute'_n}(c^{\uom}_{\te'_0}) - Df_{\ute'_n}(c_{\te'_0})\rv}{\lv Df_{\ute'_n}(c_{\te'_0} )\rv } \lesssim c_5\,\ro^{1-1/2k}.
\]
Let us now show that these estimates remain valid for a much larger value of $n$, that is, $n_1=\ceil{\rho^{-1/k}}$ when $\rho$ is sufficiently small. Define for each $\ute' \in \Sigma^{-}$ and $n \ge 0$
\[
x_n(\ute') \coloneqq \frac{ Df^{\uom}_{\ute'_{n}}(c^{\uom}_{\te'_0})}{Df_{\ute'_n}(c_{\te'_0} )}, \quad  c^{\uom}_n(\ute') \coloneqq f^{\uom}_{\ute'_{n}}(c^{\uom}_{\te'_0}), \quad  c_n(\ute') \coloneqq f_{\ute'_{n}}(c_{\te'_0}).
\]
Notice that for $n \ge n_0(\ute')$, the points $c^{\uom}_n(\ute')$ and $c_n(\ute')$ are always on the same piece $G(\und{b})$, with $\und{b} \in \Sigma(\rho^{1/k})$, because of the definition of this number. Thus, in a neighborhood of $c^{\uom}_n(\ute')$, $f^{\uom}_{(\te_{-n-1}, \te_{-n})}$ is just $f_{(\te_{-n-1}, \te_{-n})}$ composed with a translation (see \eqref{f}), and therefore
\begin{align*}
    \lv \frac{x_{n+1}(\ute')}{x_n(\ute')} - 1 \rv & = \lv \frac{Df^{\uom}_{(\te_{-n-1}, \te_{-n})}(c^{\uom}_n(\ute'))}{Df_{(\te_{-n-1}, \te_{-n})}(c_n(\ute'))} - 1 \rv \\ & = \lv \frac{Df_{(\te_{-n-1}, \te_{-n})}(c^{\uom}_n(\ute')) - Df_{(\te_{-n-1}, \te_{-n})}(c_n(\ute')) }{Df_{(\te_{-n-1}, \te_{-n})}(c_n(\ute'))} \rv \lesssim c_5 \,\ro^{1+1/2k},
\end{align*}
because of (i) and the fact that the $f_{(\te_{-n-1}, \te_{-n})}$ are $C^\infty$ with uniformly bounded derivatives. It follows that for every $n$ such that $n_0(\ute') \le n \le n_1$
\[
\lv \frac{x_{n}(\ute')}{x_{n_0(\ute')}(\ute')} -1 \rv \lesssim c_5 \,\ro^{1+1/2k} \rho^{-1/k} \lesssim c_5 \,\ro^{1-1/2k}.
\]
Hence $\lv x_{n}(\ute') -1 \rv \lesssim c_5 \,\ro^{1-1/2k} $ for all $n \le n_1$, because $\lv x_{n}(\ute') -1 \rv \lesssim c_5 \,\ro^{1-1/2k}$ for all $n \le n_0(\ute')$ by the discussion above. Moreover, since $m \lesssim \log{\rho}$, the same estimations imply that $\lv x_{n}(\ute') -1 \rv \lesssim c_5 \,\ro^{1-1/2k} $ for all $n \le n_1 +m$ (this is the only part we use $\diam(G(\und{a})) \ge c_0^{-1}\rho$).

Observe that $n_1 \gg \max{\{n_0(\ute), n_0(\ute\und{a})\}} \approx \log{\rho}$ if $\rho $ is sufficiently small. The estimates above imply that 
\[
\lv \frac{ \left(Df_{\ute_n}(c_{\te_0})\right)^{-1} \cdot Df_{(\ute\und{a})_{n+m}}(c_{a_m})}{\left(Df^{\uom}_{\ute_n}(c^{\uom}_{\te_0})\right)^{-1} \cdot Df^{\uom}_{(\ute\und{a})_{n+m}}(c^{\uom}_{a_m})} - 1 \rv = \lv \frac{x_n(\ute)}{x_{n+m}(\ute\und{a})}- 1\rv \lesssim c_5 \,\ro^{1-1/2k},
\]
and so \eqref{r} is true for $n \le n_1$.

For $n \ge n_1$, let 
\[
y_n = \left(Df_{\ute_n}(c_{\te_0})\right)^{-1} \cdot Df_{(\ute\und{a})_{n+m}}(c_{a_m}) \quad \text{ and } \quad y^{\uom}_n = \left(Df^{\uom}_{\ute_n}(c^{\uom}_{\te_0})\right)^{-1} \cdot Df^{\uom}_{(\ute\und{a})_{n+m}}(c^{\uom}_{a_m}).
\]
Following this notation,
\[
\lv \frac{y_{n+1}}{y_n} -1 \rv= \lv \frac{Df_{(\te_{-n-1},\te_{-n})}(c_{n+m}(\ute\und{a}))}{Df_{(\te_{-n-1},\te_{-n})}(c_n(\ute))} -1 \rv \lesssim \lv c_{n+m}(\ute\und{a}) -  c_n(\ute) \rv \le \diam(G(\ute_n)) 
\]
and the analogous relation is valid for the perturbed versions. However, remember that there is $C>0$ such that $\diam(G(\ute_n)), \diam(G^{\uom}(\ute_n)) \le C \mu^{-n}$. This geometric control implies that there is a positive constant $C'$ such that
\[
\lv \frac{y_n}{y^{\uom}_n} - 1 \rv \le C'(\mu^{-n_1} + c_5 \rho^{1-1/2k}), 
\]
for every $n \ge n_1$. Since $ n_1 \ge \rho^{-1/k}$, if $\rho$ is sufficiently small, then $ \mu^{-n_1} \ll c_5\, \rho^{1-1/2k}$, and so the estimate \eqref{r} is valid for all $n \ge 0$, concluding the first part of the proof.

\end{proof}

\begin{remark}\label{rem:per}
Some remarks relating the perturbation:
\begin{itemize}
\item Note that since $\|f^{\uom}_{\und{a}}-f_{\und{a}}\|_{C^0}\lesssim c_5\rho^{1+1/2k}$ then, supposing $\rho$ is small enough, we have $G^{\uom}(\und{a})\subset V_{\rho}(G(\und{a}))$.

\item Remember that we assume the base points $c_a$, $a\in \mathbb{A}$, are pre-periodic. From this, it is easy to prove that the base points do not depend on $\uom$, i.e. $c^{\uom}_a=c_a$. Indeed, let $\und{\alpha}\und{\beta}\und{\beta}...\und{\beta}...$ be the symbolic sequence associated to the points $c^{\uom}_a$. Then we can write
\[c^{\uom}_a=\lim_{n\to \infty} f^{\uom}_{\und{\alpha}\und{\beta}^n}(x),\]
where $x$ is any element in $G(\beta)$, and $\beta$ is the last letter of $\und{\beta}$. Notice that if $n_0$ is big enough then $\und{\beta}^{n_0}$ contains a word of $\Sigma(\rho^{1/3k})$ repeated twice, thus any $\und{\gamma}\in \Sigma(\rho^{1/k})$ containing $\und{\beta}^{n_0}$ can not be in $\Sigma_1$. This implies that
\[f^{\uom}_{\und{\alpha}\und{\beta}^n}(x)=f_{\und{\alpha}\und{\beta}^{n-n_0}}(f^{\uom}_{\und{\beta}^{n_0}}(x)),\]
for all $n>n_0$. Making $n$ go to infinity we conclude that $c^{\uom}_a$ does not depend on $\uom$.
\item If $diam(G(\und{a}))\ge c_0^{-1}\rho$ then $|DF^{\tb}_{\und{a}}|\approx |DF^{\tb, \uom}_{\und{a}}|$ and using
\[DF^{\tb, \uom}_{\und{a}}=Dk^{\tb, \uom}(c_{\und{a}})\cdot Df^{\uom}_{\und{a}}(c_{a})\]
one sees that $ |Df_{\und{a}}(c_{a})|\approx |Df^{\uom}_{\und{a}}(c_{a})|$ and then $diam(G(\und{a}))\approx dim(G^{\uom}(\und{a}))$. On can arrive to a similar estimate if $diam(G(\und{a}))\geq c_0^{-1}\rho^{3}$, in this case we can decompose $\und{a}$ as a concatenation of at most $4$ words in $\Sigma(\rho)$ and use the fact that $diam(G(\und{a}_1\und{a}_2))\approx diam(G(\und{a}_1)) diam(G(\und{a}_2))$. Notice that with this approach the constants get worse if we increase the power of $\rho$ in which we are interested, $\rho^3$ will be enough for us.
\item Let $\uom_1, \uom_2 \in \Omega$ and $\tb\in \Sigma^-$, suppose that
\[f^{\uom_1}_{\tb_n}(z)=f^{\uom_2}_{\tb_n}(z)\]
for all $n\leq N$ and all $z$ in a neighborhood of $z_0$, and in a neighborhood of $c_{\theta_0}$. Remember that limit geometries are defined by $k^{\tb, \uom}=\lim_{n\to \infty} k^{\tb,\uom}_n$, where
\[k^{\tb, \uom}_n(z)=Df^{\uom}_{\tb_n}(c_{\theta_0})^{-1} (f^{\uom}_{\tb_n}(z)-f^{\uom}_{\tb_n}(c_{\theta_0})).\]
By our assumption we have  $k^{\tb, \uom_1}_n(z)=k^{\tb, \uom_2}_n(z)$, for all $n\leq N$ and $z$ in a neighborhood of $z_0$. From the proof of the existence of limit geometries (see \cite{AM}) one has that there is a constant $C$ such that
\[\|k^{\tb,\uom}-k^{\tb,\uom}_n\|\leq C diam(G^{\uom}(\tb_n))\]
and
\[ \|D(k^{\tb,\uom}\circ (k^{\tb,\uom}_n)^{-1})-Id\| \leq C diam(G^{\uom}(\tb_n)),\]
the same constant $C$ works for all Cantor sets $K^{\uom}$, since they depend continuously on $\uom$.
It follows easily that there is a constant $C'$ such that
\[|k^{\tb,\uom_1}(z)-k^{\tb,\uom_2}(z)|\leq C' diam(G^{\uom_1}(\tb_N))\]
and
\[|Dk^{\tb,\uom_1}(z)-Dk^{\tb,\uom_2}(z)|\leq C' diam(G^{\uom_1}(\tb_N)),\]
for all $z$ in a neighborhood of $z_0$. Notice that since $diam(G^{\uom_j}(\tb_n))\approx |Df^{\uom_j}_{\tb_n}(c_{\theta_0})|$, $j=1,2$, and $Df^{\uom_1}_{\tb_N}(c_{\theta_0})=Df^{\uom_2}_{\tb_N}(c_{\theta_0})$ then $diam(G^{\uom_1}(\tb_N))\approx diam(G^{\uom_2}(\tb_N))$.
\end{itemize}
\end{remark}

\section{Proof of theorem \ref{thm:main2}}\label{sec:proof}

In this section we will define the set of relative configurations $\LL = \LL_{\uom}$, which will be a recurrent compact set for at least one of the Cantor sets in the family of random perturbations. We first give a primary description of $\LL$ and prove that assuming a probabilistic estimate, proposition \ref{prop:probes}, then we can prove theorem \ref{thm:main2}. The proof of the probabilistic estimate will be given in later sections.  

\subsection{The recurrent compact set}

The set $\LL = \LL_{\uom}$ will depend on $\uom$, but only the translation
coordinate $t$. The image of $\LL_{\uom}$ under the projection
map: $\CC \to \mathcal{S}$ will be a subset $\tilde{\LL}$ of $\mathcal{S}$ independent of
$\uom$.

We will choose a subset of $\Sigma^-$ with good combinatorial properties, this will be crucial to prove the estimate of proposition \ref{prop:probes}. First, let $\Sigma_{nr}(\rho^3)$ be the subset of $\Sigma(\rho^3)$ formed by words
$\und{a}$ such that:

\begin{enumerate}

\item no word $\und{b}\in\Sigma(\rho^{1/3k})$ appears twice in $\und{a}$;
\item if $\und{c} \in \Sigma(\rho^{1/6k})$ appears at the end of
$\und{a}$, then it does not appear elsewhere in $\und{a}$.
\end{enumerate}

We next define $\Sigma_{nr}^-$ as the subset of $\Sigma^-$ formed by
$\tb$ which end with a word in $\Sigma_{nr}(\rho^3)$. This is an open and
closed subset in $\Sigma^-$.

A family of subsets $E(\und{a}, \und{a}')$ of $J_R$\,, for $(\und{a}, \und{a}')
\in \Sigma(\rho^{1/2}) \times \Sigma'(\rho^{1/2})$ will be constructed
in subsection \ref{sec:gscl}, in relation to Marstrand's theorem, and it will satisfy the
hypothesis
$$
Leb(J_R \setminus E(\und{a}, \und{a}')) \le c_1\,, \quad \forall(\und{a}, \und{a}').
$$
Then, the Scale recurrence Lemma gives us another family $E^*(\und{a}, \und{a}')$, $(\und{a}, \und{a}') \in
\Sigma(\rho^{1/2}) \times \Sigma'(\rho^{1/2})$, with the properties indicated in the statement of the lemma.

The set $\widetilde{\LL}$ is defined to be the subset of $\mathcal{S}_R$ formed by the
$(\tb, \tb', s)$ such that $\tb \in \Sigma_{nr}^-$\,, and
there exists $\und{a} \in \Sigma(\rho^{1/2})$, $\und{a}' \in \Sigma'(\rho^{1/2})$
with $s \in E^*(\und{a}, \und{a}')$ and $\tb$, $\tb'$ ending with
$\und{a}$, $\und{a}'$ respectively.

For every $(\tb, \tb', s)$ in
$\widetilde{\LL}$, we will define in subsection \ref{sec:L0}, considering the properties given by Marstrand's theorem, a non empty subset
$L_{\uom}^0(\tb, \tb',s)$, depending on $\uom \in \Om$, of the fiber of $\CC$ over $(\tb, \tb',s)$.

Let
$$
\LL_{\uom}^0 = \{(\tb, \tb', s,t):\, (\tb, \tb',s) \in \widetilde{\LL}, t \in L_{\uom}^0(\tb, \tb',s)\};
$$
consider next the $\rho$-neighbourhood $\LL_{\uom}^1$ of $\LL_{\uom}^0$ in $\widetilde{\LL} \times \C$:
\begin{align*}
&\quad\LL_{\uom}^1 = \{(\tb, \tb',s,t):\, (\tb, \tb',s) \in \widetilde{\LL} \text{ and } \exists\, (\tb_0, \tb_0',s_0,t_0) \in \LL_{\uom}^0\\
&\text{with } d(\tb, \tb_0) < 2\rho^{5/2}, d(\tb', \tb_0') < 2\rho^{5/2}, |s-s_0| < \rho, |t-t_0| < \rho\}. 
\end{align*}

Fix $u = (\tb, \tb',s,t) \in \widetilde{\LL}\times\C$.
We define two subsets $\Om^0(u)$, $\Om^1(u)$ of $\Om$. First,
$$
\Om^1(u) = \{\uom \in \Om:\, (\tb, \tb',s,t) \in
\LL_{\uom}^1\}.
$$
Second, $\Om^0(u)$ is the set of $\uom \in \Om$ such that there exists $\und{b} \in \Sigma(\rho)$, $\und{b}' \in \Sigma'(\rho)$, with $b_0 = \theta_0$, $b_0' = \theta_0'$ and the image $T^{\uom}_{\und{b}}T'_{\und{b}'}(u) = (\tilde{\tb}, \tilde{\tb}',\tilde{s}, \tilde{t}$) satisfies:
\begin{itemize}
\item[(i)] for any $\tilde{s}_1$ with $|\tilde{s}-\tilde{s}_1| < \dfrac 12 c_4 \rho^{1/2}$, we have $(\tilde{\tb}, \tilde{\tb}', \tilde {s}_1) \in \widetilde{\LL}$;
\item[(ii)] $\tilde{t} \in L_{\uom}^0(\tilde{\tb},
\tilde{\tb}',\tilde{s})$.
\end{itemize}
The following crucial estimate will be proven in section \ref{sec:proofprobes}.

\vglue .2in

\begin{proposition}\label{prop:probes}
Assume that $c_5$ is chosen conveniently large. Then there exists $c_7 > 0$, such that, for any $u \in {\widetilde\LL}\times\C$, one has
$$
\bp(\Om^1(u) \setminus \Om^0(u)) \le \exp\big(-c_7\rho^{-\frac{1}{2k}(d+d'-2)}\big).
$$
\end{proposition}

Using the previous estimate, we will prove that for some $\uom$ the pair $K^{\uom},\, K$ has a recurrent compact set, thus obtaining theorem \ref{thm:main2}. We proceed as in \cite{MY}. We discretize the set $\tilde\LL\times \{t:|t|<4(1+e^R)\}$, i.e. we choose subsets $\Delta_i$, $i\in I$, such that
\[\tilde\LL\times \{t:|t|<4(1+e^R)\}= \bigcup_{i\in I} \Delta_i,\]
and for all $(\tb_1,\tb'_1,s_1,t_1),\,(\tb_2,\tb'_2,s_2,t_2)\in \Delta_i$ one has
\[d(\tb_1,\tb_2)< \rho^{5/2},\,d(\tb'_1,\tb'_2)< \rho^{5/2},\, |s_1-s_2|< \rho^3,\, |t_1-t_2|< \rho^3.\]
It is not difficult to see that this can be done in such a way that $\# I$ is polynomial in $\rho^{-1}$. For each $i\in I$, choose $u_i\in \Delta_i$. If $\rho$ is small enough, we have
\[\mathbb{P}(\cup_{i\in I} \Om^1(u_i) \setminus \Om^0(u_i))\leq \# I\cdot \exp\left(-c_7\rho^{-\frac{1}{2k}(d+d'-2)}\right)<1.\]
Therefore, the set $\cup_{i\in I} \Om^1(u_i) \setminus \Om^0(u_i)$ is not the whole $\Omega$ and we can choose $\uom_0$ outside of it. Observe that for all $i\in I$ we have that $\uom_0\in \Om^1(u_i)$ implies $\uom_0\in \Om^0(u_i)$. Now define
\begin{align*}
\LL_{\uom_0}=\{&(\tb,\tb',s,t)\in \tilde\LL\times \C:\,\exists (\tilde\tb,\tilde\tb',\tilde s,\tilde t)\in \LL^0_{\uom_0}\\
&\text{with } d(\tb,\tilde\tb)\leq \rho^{5/2},\,d(\tb',\tilde\tb')\leq \rho^{5/2},\, |s-\tilde s|\leq \rho/2,\, |t-\tilde t|\leq \rho/2\}.
\end{align*}
We will prove that $\LL_{\uom_0}$ is a recurrent compact set for $K^{\uom_0},\, K'$. First, notice that $\LL^0_{\uom_0}\subset \LL_{\uom_0}\subset \LL^1_{\uom_0}$. Then $\LL_{\uom_0}$ is not empty. Now, let $u=(\tb,\tb',s,t)\in \LL_{\uom_0}$, we have that there is $(\tilde\tb,\tilde\tb',\tilde s,\tilde t)\in \LL^0_{\uom_0}$ with the properties in the definition of $\LL_{\uom_0}$. Since $\tilde t\in L^0_{\uom_0}(\tilde\tb,\tilde\tb',\tilde s)$ then $|\tilde t|\leq 3 (1+e^R)$ (this is clear from the definition of $L^0$ in subsection \ref{sec:L0} and equation \eqref{eq:formularenorma}). Thus $|t|< 4(1+e^R)$ and $u\in \Delta_i$ for some $i\in I$. For $u_i=(\tb_i,\tb'_i,s_i,t_i)$ we have that
\[d(\tb,\tb_i)< \rho^{5/2},\,d(\tb',\tb'_i)< \rho^{5/2},\, |s-s_i|< \rho^3,\, |t-t_i|< \rho^3.\]
Therefore $u_i\in \LL^1_{\uom_0}$ and $\uom_0\in \Om^1(u_i)$, this implies that $\uom_0\in \Om^0(u_i)$ and there is some pair $(\und{b},\und{b}')\in \Sigma(\rho)\times \Sigma'(\rho)$ such that $T^{\uom_0}_{\und{b}}T'_{\und{b}'}(u_i) = (\tilde{\tb}_i, \tilde{\tb}'_i,\tilde{s}_i, \tilde{t}_i)$ satisfies the properties (i) and (ii) described above.

We will prove that $T^{\uom_0}_{\und{b}}T'_{\und{b}'}(u)$ is in the interior of $\LL_{\uom_0}$. Write $T^{\uom_0}_{\und{b}}T'_{\und{b}'}(u)=(\tb\und{b},\tb'\und{b}',\hat{s},\hat{t})$, using equation \eqref{eq:formularenorma} we have
\[\tilde{t}_i= \left(DF^{\tb'_i}_{\und{b}^{\prime}}\right)^{-1}\cdot (t_i+s_i c^{\tb_i,\uom_0}_{\und{b}}-c^{\tb'_i}_{\und{b}^{\prime}}),\text{ and } \hat{t}= \left(DF^{\tb'}_{\und{b}^{\prime}}\right)^{-1}\cdot (t+s c^{\tb,\uom_0}_{\und{b}}-c^{\tb'}_{\und{b}^{\prime}}).\]
Therefore $|\tilde{t}_i-\hat t|\lesssim \rho^{3/2}$. Analogously one has
\[\hat s= \frac{DF^{\tb,\uom_0}_{\und{b}}}{DF^{\tb'}_{\und{b}'}} \cdot s,\text{ and } \tilde{s}_i= \frac{DF^{\tb_i,\uom_0}_{\und{b}}}{DF^{\tb'_i}_{\und{b}'}} \cdot s_i.\]
In this case we get $|\tilde{s}_i-\hat s|\lesssim \rho^{5/2}$. One also has $d(\tb\und{b},\tilde\tb_i)\lesssim \rho^{7/2}$ and $d(\tb'\und{b}',\tilde{\tb}'_i)\lesssim \rho^{7/2}$. 

Thanks to property (ii), we know that $(\tilde{\tb}_i, \tilde{\tb}'_i,\tilde{s}_i, \tilde{t}_i)\in \LL^0_{\uom_0}$. Moreover, for any $(\und\eta,\und\eta',r,x)$ such that
\[d(\und\eta,\tb\und{b})<\rho^{7/2},\, d(\und\eta',\tb'\und{b}')<\rho^{7/2},\,|r-\hat s|<\rho/4,\,|x-\hat t|<\rho/4,\]
we have
\[d(\und\eta,\tilde\tb_i)\leq \rho^{5/2},\,d(\und\eta',\tilde\tb'_i)\leq \rho^{5/2},\, |r-\tilde s_i|\leq \rho/2,\, |x-\tilde t_i|\leq \rho/2.\]
To conclude that $(\und\eta,\und\eta',r,x)\in \LL_{\uom_0}$ we only need to show that $(\und\eta,\und\eta',r)\in \tilde\LL$. Property (i) above implies that $(\tilde\tb_i,\tilde\tb'_i,r)\in \tilde\LL$, by the definition of the set $\tilde\LL$ this means that $\tilde\tb_i\in \Sigma^-_{nr}$ and $r\in E^*(\und{a},\und{a}')$ for a pair $(\und{a},\und{a}')$ in $\Sigma(\rho^{1/2})\times \Sigma'(\rho^{1/2})$ such that $(\tilde\tb_i,\tilde\tb'_i)$ ends in it. However, since $d(\und\eta,\tilde\tb_i)\lesssim \rho^{7/2},\,d(\und\eta',\tilde\tb'_i)\lesssim \rho^{7/2}$ then $\und\eta\in \Sigma^-_{nr}$ and $(\und\eta,\und\eta')$ also ends in $(\und{a},\und{a}')$. Therefore $(\und\eta,\und\eta',r)\in \tilde\LL$ and $(\und\eta,\und\eta',r,x)\in \LL_{\uom_0}$, which shows that $T^{\uom_0}_{\und{b}}T'_{\und{b}'}(u)$ is in the interior of $\LL_{\uom_0}$. From the fact that the sets $E^*(\und{a},\und{a}')$ are closed and the definition of the sets $L^0_{\uom}$, it is not difficult to prove that $\LL_{\uom_0}$ is a compact set. Therefore, $\LL_{\uom_0}$ is a recurrent compact set for the pair $K^{\uom_0},\,K'$.

\subsection{Set of good scales}\label{sec:gscl}

In this subsection we will define the sets $E(\und{a},\und{a}')$ which we use to construct $\ti{\LL}$. Let $(\tb, \tb',s)$ in the space of relative scales, and points $x \in K(\theta_0)$, $x' \in K'(\theta_0')$. Consider
\begin{align*}
\lambda &= \pi_{\tb,\tb',s}(x,x'):= k^{\tb'}(x') - sk^{\tb}(x).
\end{align*}
Then $(\tb, \tb',s,\lambda)$ is the unique relative configuration above $(\tb, \tb',s)$ such
that
$$
A(k^{\tb}(x)) = A'(k^{\tb'}(x')),
$$
(where $(\tb,A), (\tb',A')$ represents this relative configuration).

Remember that, for some previously fixed $R>0$ (given by the scale recurrence lemma),
\[J_R=\{s\in J: e^{-R}\le |s|\le e^R\}\]
and $\si_R=\Sigma^{-} \times {\Sigma'}^- \times J_R$.

Let $d$, $d'$ be the Hausdorff dimension of $K$, $K'$, respectively. We equip each set $K(\theta_0)$ (resp. $K'(\theta_0'))$ with
the $d$-dimensional (resp. $d'$-dimensional) Hausdorff measure $\mu_d$ (resp. $\mu_{d'}$).

Then, for $(\ute,\ute',s) \in \si$, we denote by
$\mu(\ute,\ute',s)$ the image under $\pi_{\ute,\ute',s}$ of
$\mu_d \times \mu_{d'}$ on $K(\te_0) \times K'(\te_0')$.

As in the theory of Cantor sets in the real line, there are constants $c_{11} > c_{10} > 0$ such that,
for $\te_0 \in \A$, $\te_0' \in \A'$:
$$
c_{10} < \mu_d \times \mu_{d'}(K(\te_0) \times K'(\te_0')) < c_{11}\,.
$$

This can be proven using the results appearing in Zamudio's thesis (\cite{Z}). Indeed, for a given conformal Cantor set $K$ of dimension $d$, using equation \eqref{eq:sigrho}, one can find a sequence of coverings with size converging to zero and $d$-volume bounded by $Cc_0^d$, namely the covering by the pieces of $\Sigma(\rho)$, showing that $\mu_d(K)<\infty$. 

On the other hand, lemma 1.2.3 of \cite{Z} gives that there exist constants $C_1>0$, $L>0$, $r_0>0$, independent of $\rho$, such that for all $\und{a}\in \Sigma(\rho)$
\[C_1^{-1}\left( \frac{\rho}{r}\right)^{-d}\leq \#\{\und{b}\in \Sigma(\rho): \frac{r}{l}\leq d(G(\und{a}),G(\und{b})) < r\} \leq C _1\left( \frac{\rho}{r}\right)^{-d},\]
for all $l\geq L$, $r<r_0$. Using this lemma, we conclude that given a finite cover of $K$ by balls $U_i$ of radii $r_i>0$, $i=1,\dots,n$, each $U_i$ intersects at most $C_1\rho^{-d}r_i^d$ pieces $G(\und{a})$ of $\Sigma(\rho)$ if $\rho$ and $r_i$ are sufficiently small. Since $U_i$ is a cover and $\#\Sigma(\rho)>C^{-1}\rho^{-d}$, summing for all $i$ yields:
$$C_1\rho^{-d}\left(\sum r_i^d\right)\geq C^{-1}\rho^{-d}$$
and so $\sum r_i^d$ (and $\mu_d(K)$) is always bounded from zero. To obtain the statement just restrict the arguments to $K(\te_0)$ and ${K'}(\te'_0)$ and take their product. 

Notice that the same lemma 1.2.3 implies that there is a constant $c>0$ such that for $\mu:=\mu_d\times\mu_{d'}$, the product measure in $\C^2$,
$$ \mu(B(x,r))<cr^{d+d'} $$
for any ball of radius $r>0$. If $d+d'>2$ this condition implies that:
    $$I_2(\mu):=\int\int|u-v|^{-2}d\mu(u)d\mu(v)< \infty$$
This way, the proof of Theorem 9.7 in Mattila's book (a Marstrand-type theorem) \cite{mattila} can be adapted\footnote{all one needs to verify is that for any points $u,v \in \C^2$, $Leb(\{s: s \in J_R,\, |\pi_{\tb,\tb',s}(u)-\pi_{\tb,\tb',s}(v)|< \delta\})<c\delta^2|u-v|^{-2}$, where 
%$m$ is Lebesgue measure on $J$ and 
$c>0$ is some constant depending only on $R$. Notice that $\pi_{\tb,\tb',s}=\pi_s\circ F_{\tb,\tb'} $, where $\pi_s(u)=u_2-s\cdot u_1$ for $u=(u_1,u_2)$ and $F_{\tb,\tb'}=(k^\tb,k^{\tb'})$ are diffeomorphisms that distort area in a uniformly bounded way. A simple manipulation shows that the measure is bounded above by $c\delta^{2}|u_1-v_1|^{-2}$. If $|u_1-v_1|\le \delta$ then $\delta^2|u-v|^{-2}>\ti{c}$ for some constant $\ti{c}$ and the desired inequality follows choosing $c$ big enough. If $|u_1-v_1|> \delta$, using that $s$ is in $J_R$ a bounded set, one sees that $|u_1-v_1|^{-1}$ is bounded by $c_R|u-v|^{-1}$, for some constant $c_R$ depending on $R$.  } to our context to show that for fixed $(\ute, \ute')$ the measure $\mu(\ute,\ute',s)$ is absolutely continuous with respect to the Lebesgue measure for Lebesgue almost every $s$, with density $\chi_{\ute,\ute',s}$ in $L^2$ satisfying
$$
\int_{J_R} \| \chi_{\ute,\ute',s}\|_{L^2}^2\,ds \le c_{12}(R),
$$
where $c_{12}(R)$ is independent of $\ute$, $\ute'$.

When one controls $\|\chi_{\ute,\ute',s}\|_{L^2}$\,, this gives, by Cauchy-Schwarz inequality, a lower bound for the Lebesgue measure of $\pi_{\ute,\ute',s}(X)$,\,\, $X$ being a subset of $K \times K'$ with positive $(d+d')$-dimensional Hausdorff measure; indeed we have:
\begin{align*}
\mu_d\times\mu_{d'}(X) &\leq \int_{\pi_{\tb,\tb',s}(X)}
\chi_{\tb,\tb',s}(t)\,dt\\
&\le Leb(\pi_{\tb,\tb',s}(X))^{1/2} 
\| \chi_{\ute,\ute',s}\|_{L^2}
\end{align*}
and therefore
\begin{equation}\label{eq:cauchy}
Leb(\pi_{\ute,\ute',s}(X)) \ge (\mu_d\times\mu_{d'}(X))^2 
\| \chi_{\ute,\ute',s}\|_{L^2}^{-2}\,.
\end{equation}

Fix $(\ute,\ute')$ in $\Sigma^- \times
\Sigma^{\prime -}$. Let $\und{a} \in \Sigma(\rho^{1/2k})$, $\und{a}' \in
\Sigma'(\rho^{1/2k})$, with $a_0 = \te_0$\,, $a_0' = \te_0'$\,. One has
$$
T_{\und{a}} T'_{\und{a}'}(\ute,\ute',s) = \left(\ute\und{a}, \ute' \und{a}',
s\cdot DF^{\tb}_{\und{a}}/DF^{\tb'}_{\und{a}'} \right)
$$
and
$$
c_{13}^{-1} \le \frac{|DF^{\tb}_{\und{a}}|}{|DF^{\tb'}_{\und{a}'}|} \le
c_{13}\,.
$$

We therefore have
$$
\int_{J_R} \lVert \chi_{{T_{\und{a}}
T'_{\und{a}'}(\ute,\ute',s)}}\rVert_{L^2}^2 \,ds \le c_{12}'(R),
$$
with $c_{12}'(R)$ independent of $\ute$, $\ute'$, $\und{a}$,
$\und{a}'$. On the other hand, one has

\begin{align*}
   & \#\,\Sigma(\rho^{1/2k}) \le c_{14}\,\rho^{-d/2k}\,,\\
   & \#\,\Sigma'(\rho^{1/2k}) \le c_{14}\,\rho^{-d'/2k}\,. 
\end{align*}

We conclude that
$$
\int_{J_R} \sum_{\und{a},\und{a}'} \| \chi_{{T_{\und{a}}
T'_{\und{a}'}(\ute,\ute',s)}}\|_{L^2}^2 \,ds \le
c_{14}^2\,\rho^{-\frac{d+d'}{2k}}\, c_{12}'(R).
$$
We now define, with $c_{15} > 0$ conveniently large to be determined later:

\begin{align*}
    E(\ute,\ute') &= \bigg\{s \in J_R:\, \|
\chi_{\ute,\ute',s}\|_{L^2}^2  \le c_{15}\\
&\quad\text{and }\quad \sum_{\und{a},\und{a}'} \| \chi_{{T_{\und{a}}
T'_{\und{a}'}(\ute,\ute',s)}}\|_{L^2}^2 \le c_{15}\,\rho^{-
\frac{d+d'}{2k}}\bigg\}.
\end{align*}

For $\und{c} \in \Sigma(\rho^{1/2})$, $\und{c}' \in \Sigma'(\rho^{1/2})$, we
define $E(\und{c}, \und{c}')$ as the set of $s \in J_R$ such that there exists
$\ute$, $\ute'$ ending respectively with $\und{c}$, $\und{c}'$ such
that $s \in E(\ute, \ute')$.

One has, for any $\ute \in \Sigma^-$, $\ute' \in \Sigma^{\prime -}$:
$$
Leb(J_R \setminus E(\ute,\ute')) \le c_{15}^{-1}(c_{12}(R) +
c_{14}^2\,c_{12}'(R));
$$
therefore, provided that
$$
c_{15} > c_1^{-1}(c_{12}(R) + c_{14}^2\,c_{12}'(R)),
$$
we will have
$$
Leb(J_R \setminus E(\und{c},\und{c}')) \le c_1
$$
for all $\und{c} \in \Sigma(\rho^{1/2})$, $\und{c}' \in \Sigma'(\rho^{1/2})$. 
This means that we can apply the Scale recurrence Lemma with the family
$E(\und{c},\und{c}')$ of subsets of $J_R$. The sets $E^*(\und{c},\und{c}')$ are then defined using this lemma (see section \ref{se:scrl}), we can assume they are closed, this is justified in the remark after the lemma.

\subsection{Construction of $L^0_{\uom}$}\label{sec:L0}

We now consider the family of random perturbations $g^{\uom}$ again and proceed to construct the sets $L_{\uom}^0(\ute,\ute',s)$,
for $(\ute,\ute',s) \in \tilde\LL$. For $\und{a} \in
\Sigma(\ro^{1/2k})$, let $\Sigma^-(\und{a})$ be the open and closed subset of $\Sigma^-$ formed by the $\ute$ ending with $\und{a}$. Choose a subset $\Sigma_2^-$ of $\Sigma(\ro^{1/2k})$ such that
$$
\Sigma^- = \bigcup_{\Sigma_2^-} \Sigma^-(\und{a})
$$
is a partition of $\Sigma^-$.

For $\und{a} \in \Sigma_2^-$, define a subset $\Sigma_1(\und{a})$ of the subset $\Sigma_1$ (recall $\Sigma_1 \subset \Sigma(\rho^{1/k})$), as the set of words in $\Sigma_1$ starting with $\und{a}$. For $\ute \in \Sigma^-(\und{a})$, we also define $\Sigma_1(\ute) = \Sigma_1(\und{a})$.

Let $\ute \in \Sigma^-$. We write
\begin{align*}
\Om &= [-1,+1]^{\Sigma_1(\ute)} \times [-1,+1]^{\Sigma_1 \setminus
\Sigma_1(\ute)},\\
\uom &= (\uom',\uom'')\\
\intertext{and for such an $\uom$, we set}
\uom^* &= (\und{0},\uom'').
\end{align*}
This depends on $\ute$, but nearby $\und{\widehat\te}$ (with
$d(\ute, \und{\widehat\te}) < c_0^{-1}\,\ro^{1/2k}$) will belong to the same $\Sigma^-(\und{a})$ and give the same projection $\uom^*$ of $\uom$.

For $(\ute,\ute',s) \in \tilde\LL$, the set
$L_{\uom}^0(\ute,\ute',s)$ will actually only depend (as far as $\uom$ is concerned) on the projection $\uom^*$ of $\uom$ associated to $\ute$.

We will say that two words $\und{b}^0,\und{b}^1 \in \Sigma(\ro)$ are {\it independent\/} if there is no word $\und{b} \in \Sigma(\ro^{1/2k})$ such that both $\und{b}^0$ and $\und{b}^1$ start with $\und{b}$.

With $c_{16} > 0$ conveniently small, to be chosen in the following, let
$$
N = \bigg[c_{16}^2\,\ro^{-\frac{1}{2k}(d+d'-2)}\bigg].
$$

Let $(\ute,\ute',s) \in \tilde\LL$ and $\uom \in \Om$.

We define $L_{\uom}^0(\ute,\ute',s)$ to be the set of points $(\ute,\ute',s,t)$ in the fiber for which there exist pairs $(\und{b}^1,\und{b}^{\prime 1}),\dots,(\und{b}^N,\und{b}^{\prime N})$ in $\Sigma(\ro) \times \Sigma'(\ro)$, with $b_0^i = \te_0$\,,\,\, $b_0^{\prime i} = \te_0'$ such that, if we set
$$
T_{\und{b}^i}^{\uom^*}\,T'_{\und{b}^{\prime i}}(\ute,\ute',s,t)
= (\ute^i, \ute^{\prime i}, s_i,t_i),
$$
the following hold:
\begin{itemize}
\item[(i)] the words $\und{b}^1,\dots,\und{b}^N$ are pairwise independent;
\item[(ii)] for $1 \le i \le N$, \,\,\, $\ute^i \in \Sigma_{nr}^-$\,;
\item[(iii)] for $1 \le i \le N$, and $|\tilde s - s_i| \le  \dfrac 23 c_4
\ro^{1/2}$,\,\, $(\ute^i, \ute^{\prime i}, \tilde s) \in
\tilde\LL$;
\item[(iv)] for $1 \le i \le N$,\,\,\, $|t_i| \le 2(1+e^R)$.
\end{itemize}
We will use also a slightly smaller set $L_{\uom}^{-
1}(\ute,\ute',s)$; it is defined in the same way than
$L_{\uom}^0(\ute,\ute',s)$, but with (iii), (iv) replaced by:
\begin{itemize}
\item[(iii)'] for $1 \le i \le N$, and $|\tilde s-s_i| \le \dfrac 34 c_4
\ro^{1/2}$, $(\ute^i, \ute^{\prime i}, \tilde s) \in \tilde\LL$
\item[(iv)'] for $1 \le i \le N$, \,\, $|t_i| \le 1+e^R$.
\end{itemize}
In the next section, we will prove the following estimate.

\vglue .2in

\begin{proposition}\label{prop:bigL}
If $c_{16}$ has been chosen sufficiently small, there exists $c_{17} > 0$ such that, for any $(\ute, \ute',s) \in \tilde\LL$ and any $\uom \in \Om$, the Lebesgue measure of $L_{\uom}^{-1}(\ute,\ute',s)$  is $> c_{17}$\,.
\end{proposition}

\section{Proof of proposition \ref{prop:bigL}}\label{sec:bigL}

In this section we will prove proposition \ref{prop:bigL}. First we prove some lemmas that are necessary for the proposition. We follow the same argument as \cite{MY} sections 4.8-4.12 with some modifications.

Fix $(\ute,\ute',s) \in \tilde\LL$, we will work with this triple throughout this section and at the end we will prove that $Leb(L^{-1}_{\und{\om}}(\ute,\ute',s))>c_{17}$.

Choose a subfamily $\Sigma_2$ of $\Sigma(\ro^{1/2k})$ of words starting with $\te_0$ such that
$$
K(\te_0) = \bigcup_{\Sigma_2} K(\und{a})
$$
is a partition of $K(\te_0)$. Similarly, choose a subfamily $\Sigma_2'$ of $\Sigma'(\ro^{1/2k})$ of words starting with $\te_0'$ such that
$$
K'(\te_0') = \bigcup_{\Sigma_2'} K'(\und{a}').
$$
There is a constant $c_{19}>0$ such that, for each $(\und{a},\und{a}') \in \Sigma_2 \times \Sigma_2'$, we have
$$
c_{19}^{-1}\, \ro^{\frac{1}{2k}(d+d')} \le \mu_d \times \mu_{d'}(K(\und{a})
\times K'(\und{a}')) \le c_{19}\, \ro^{\frac{1}{2k}(d+d')}\,.
$$

Let $J(\und{a},\und{a}') := \pi_{\und{\te},\und{\te}',s}(G(\und{a}) \times
G(\und{a}'))$ and $c(\und a, \und a') \coloneqq \pi_{\und{\te},\und{\te}',s}(c_{\und{a}},
c_{\und{a}'}) \in \C$ for $\und{a}  \in \Sigma_2$ and $\und{a}' \in \Sigma'_2$.
%$\tilde{J}(\und{a},\und{a}') := \pi_{\tilde{\und{\te}},\tilde{\und{\te}}',\tilde{s}}(I(\und{a}) \times I(\und{a}'))$, for $\und{a} \in \Sigma_2$, $\und{a}' \in \Sigma_2$\,; 
Since $s$ is bounded above and below, we have
\begin{equation}\label{eq:bt}
B(c(\und a, \und a'),c_{19}^{-1}\,\ro^{\frac{1}{2k}}) \subset J(\und a,\und a') \subset B(c(\und a, \und a'), c_{19}\,\ro^{\frac{1}{2k}})\,,
\end{equation}
if $c_{19}$ is sufficiently large. We assume that the previous relations involving $c_{19}$ hold for any other triples in $\Sigma^-\times \Sigma^{\prime -}\times J_R$ and for any value of $\und{\om}$, choosing $c_{19}$ large enough this can be easily guaranteed.

Say $(\und{a},\und{a}')$ is {\it good\/} if there are no more than $c_{16}^{-
1}\, \ro^{-1/2k(d+d'-2)}$ pairs $(\und{\tilde a},\und{\tilde a}')$ such that the distance between the points $c(\und{a}, \und{a}')$ and $c(\und{\ti{a}}, \und{\ti{a}}')$ is less than $(2+1/10)c_{19}\,\ro^{1/2k}$. Otherwise, say it is \emph{bad}.

\begin{lemma}
The number of bad pairs $(\und{a},\und{a}')$ is less than $2^2 \cdot 21^2 \cdot\pi\, c_{19}^{4}\, c_{15}\, c_{16}\, \ro^{-1/2k(d+d')}$.
\end{lemma}

Before begining the proof, we remember the Vitali covering lemma. Let $B_1,\,B_2,\,\dots,\,B_n \subset \R^d$ be a finite collection of balls. For $i=1,\,\dots,n$, denote by $3B_i$ the ball with same center as $B_i$ but having radius three times larger. The lemma states that there exists a subcollection of balls $B_{j_1},\,B_{j_2},\,\dots,\,B_{j_k}$ with the Vitali property, this is
\begin{itemize}
    \item The balls $B_{j_1},\,B_{j_2},\,\dots,\,B_{j_k}$ are pairwise disjoint and
    \item The union $B_1\cup B_2 \cup \dots \cup B_n$ is contained in $3B_{j_1}\cup 3B_{j_2} \cup \dots \cup 3B_{j_k}$.
\end{itemize}

In the case that the balls $B_i$ are subsets of the complex plane and have the same radius $R$, one can see that every point $z \in \C$ is covered by no more than $16$ of the balls $3B_{j_1},\,3B_{j_2},\,\dots,\,3B_{j_k}$. Indeed, consider the ball $B$ centered at $z$ with radius $4R$. It contains all the balls $B_{i_1},\,\dots,\,B_{i_l}$ such that $3B_{i_1},\,\dots,\,3B_{i_l}$ cover $z$. However, the balls $B_{i_1},\,\dots,\,B_{i_l}$ are pairwise disjoint, and so there are no more than $(4R)^2/R^2=16$ of them inside $B$, otherwise they would overlap.

\begin{proof} By construction of $\widetilde{\LL}$, there exists $\widetilde{\ute}$, $\widetilde{\ute}'$, $\tilde s$ with $d(\ute,\widetilde{\ute}) \le c_0\ro^{1/2}$, $d(\ute',
\widetilde{\ute}') \le c_0\ro^{1/2}$, $|s-\tilde s| \le
c_2\,\ro^{1/2}$ such that
$$
\left\| \chi_{\widetilde{\ute},\widetilde{\ute}',\tilde s} \right\|_{L^2}^2 \le
c_{15}\,.
$$

%Therefore, if we replace $\und{\te}$, $\und{\te}'$, $s$ by
%$\und{\widetilde\te}$, $\und{\widetilde\te}'$, $\tilde s$ the points of $J(\und{a},\und{a}')$ only move by a distance of order at most $\ro^{1/2}$. 
Because of remark \ref{rmk:dgl}, the distance between the points $k^{\ute}(c_{\und{a}})$ and $k^{\tilde{\ute}}(c_{\und{a}})$ is of order $\rho^{1/2}$ for every $(\und{a}, \und{a}') \in \Sigma_2 \times \Sigma'_2$ and the same is true for their $K'$ versions. Thus, if $c_{19}$ is sufficiently large, for each bad pair $(\und{a}, \und{a}')$ there are more than $c_{16}^{- 1}\, \ro^{-1/2k(d+d'-2)}$ pairs $(\tilde{\und{a}}, \und{\tilde{a}}')$ satisfying 
\[|     \pi_{\tilde{\ute},{\ti{\ute}}',\tilde{s}}(c_{\und{a}}, c_{\und{a}'})  - \pi_{\tilde{\ute},\tilde{\ute}',\tilde{s}} (c_{\und{\tilde{a}}}, c_{{\und{\ti{a}}'}}
)	|	\le \frac{5}{2}c_{19}\rho^{\frac{1}{2k}} . \]
From now on, we denote $ \pi_{\tilde{\ute},{\ti{\ute}}',\tilde{s}}(c_{\und{a}}, c_{\und{a}'})$ as $ \ti{c}(\und a, \und{a}')$ for any $(\und a, \und{a}') \in \Sigma_2 \times \Sigma'_2$.

For each bad pair $(\und{a}, \und{a}')$, consider the disk $J^{*}(\und{a}, \und{a}')$ of radius $\frac{7}{2}\,c_{19}\rho^{\frac{1}{2k}} $ and center at $\ti{c}(\und a, \und{a}')$. Then the corresponding $c_{16}^{- 1}\, \ro^{-1/2k(d+d'-2)}$  sets $ \pi_{\tilde{\ute},{\ti{\ute}}',\tilde{s}} (G(\ti{\und{a}}) \times G(\ti{\und{a}}')) $ are subsets of $J^{*}(\und{a}, \und{a}')$. This way,

\begin{align*}
    \int_{J^{*}(\und{a}, \und{a}')}    {\chi_{\tilde{\ute},\tilde{\ute}',\tilde{s} } ds}  = (\mu_d \times \mu_{d'})\left(\pi^{-1}_{\tilde{\ute},\tilde{\ute}',\tilde{s}}   (J^{*}(\und{a}, \und{a}')) \right) \ge \sum_{(\ti{\und{a}}, \ti{\und{a}}' )}{(\mu_d \times \mu_{d'}) (G(\ti{\und{a}}) \times G(\ti{\und{a}}')) } \\ \ge c_{16}^{- 1}\, \ro^{-1/2k(d+d'-2)} \cdot c_{19}^{-1} \rho^{\frac{d+d'}{2k}}= (c_{16}c_{19})^{-1} \ro^{1/k}. 
\end{align*} 

Let $J^*$ be the union of all the disks $J^{*}(\und{a}, \und{a}')$ corresponding to bad pairs and $B$ be the number of these pairs. Choose a subcover of $J^*$ as in the Vitali lemma, indexed by the pairs $(\und{a},\und{a}')$ belonging to a subset $V$ of the set of bad pairs. It follows that  
\[B\cdot c_{19}^{-1}\ro^{\frac{(d+d')}{2k}} \le 
(\mu_d \times \mu_{d'})(\pi^{-1}_{(\tilde{\ute},\tilde{\ute}',\tilde{s})}   (J^{*}))   = 
\int_{J^{*}}   {\chi_{\tilde{\ute},\tilde{\ute}',\tilde{s} } ds}   
\le \sum_{(\und{a},\und{a}') \in V} \int_{3\,J^{*}(\und{a},\und{a}')}   {\chi_{\tilde{\ute},\tilde{\ute}',\tilde{s} } ds}.\]

On the other hand, by Cauchy-Schwartz theorem, 
\begin{align*}
    \left( \int_{3\,J^{*}(\und{a}, \und{a}')}   {\chi_{\tilde{\ute},\tilde{\ute}',\tilde{s} } ds}\right) \cdot (c_{16}c_{19})^{-1} \ro^{1/k}  & \le      
    \left( \int_{3\,J^{*}(\und{a}, \und{a}')}   {\chi_{\tilde{\ute},\tilde{\ute}',\tilde{s} } ds}\right)^2 \\ \le  \text{Leb}(3\,J^{*}(\und{a}, \und{a}')) \cdot & \int_{3\,J^{*}(\und{a}, \und{a}')}   {\chi^2_{\tilde{\ute},\tilde{\ute}',\tilde{s} }ds} =
\frac{21^2}{2^2}\pi c^2_{19} \ro^{1/k} \cdot \int_{3\,J^{*}(\und{a}, \und{a}')}   {\chi^2_{\tilde{\ute},\tilde{\ute}',\tilde{s} }ds},
\end{align*}    

for every bad pair $(\und{a},\und{a}')$. But the Vitali covering covers each point $z \in \C$ at most $16$ times, so 
    \[\sum_{(\und{a},\und{a}') \in V} \int_{3\,J^{*}(\und{a},\und{a}')}   {\chi_{\tilde{\ute},\tilde{\ute}',\tilde{s} } ds} \le 
    \frac{21^2}{2^2}\pi\, c^3_{19}\,c_{16}\sum_{(\und{a},\und{a}') \in V}\int_{3\,J^{*}(\und{a}, \und{a}')}   {\chi^2_{\tilde{\ute},\tilde{\ute}',\tilde{s} }ds} \le 
    2^2 \cdot 21^2 \cdot\pi\, c^3_{19}\,c_{16} \int_{\C}   {\chi^2_{\tilde{\ute},\tilde{\ute}',\tilde{s} }ds}.\]

It follows that $B \le 2^2 \cdot 21^2 \cdot \pi\, c_{19}^4\,c_{16}\,c_{15}\, \ro^{-(d+d')/2k}$, concluding the proof.
\end{proof}

Now, we construct the pairs $(\und{b},\und{b}')$ amongst which the pairs $(\und{b}^i, \und{b}^{\prime i})$ of \ref{prop:bigL} must be looked for. We make the following observation:

\begin{lemma}
Let $\und{\te} \in \Sigma_{nr}^-$. The number of words $\und{c} \in \Sigma(\ro^{1/2})$ with $c_0=\te_0$, such that $\und{\te}\und{c} \notin \Sigma_{nr}^-$ is $o(\ro^{-d/2})$ as $\ro\to0$, uniformly in $\und{\te}$.
\end{lemma}

%%%%%%Comment%%%%%%
% How do we prove this? Maybe not so easy
%%%%%%%%%%%%%%%%%%%

Remember that $(\und{\te},\und{\te}',s) \in \widetilde{\LL}$. It follows from conclusion (ii) of the Scale recurrence Lemma (lemma \ref{lem:scl}) and the last observation that we can find at least $\frac{1}{2} c_3\,\ro^{-1/2(d+d')}$ pairs $(\und{c}^i, \und{c}^{\prime i}) \in
\Sigma(\ro^{1/2}) \times \Sigma'(\ro^{1/2})$ such that, writing $T_{\und{c}^i} T'_{\und{c}^{\prime i}}(\und{\te},\und{\te}',s) = (\und{\te}^i,\und{\te}^{\prime i},s_i)$, we have:

\begin{itemize}
\item $\und{\te}^i \in \Sigma_{nr}^-$\,;
\item $B(s_i,c_4\ro^{1/2}) \subset E^*(\und{c}^i, \und{c}^{\prime i})$.
\end{itemize}

\noindent As $(\und{\te}^i, \und{\te}^{\prime i},s_i)$ again belongs to $\widetilde{\LL}$, we can for each $i$ find at least $\frac 12 c_3\,\ro^{-1/2(d+d')}$ pairs $(\und{d}^{ij}, \und{d}^{\prime ij}) \in \Sigma(\ro^{1/2}) \times \Sigma'(\ro^{1/2})$ (with the first letter of $\und{d}^{ij}$, $\und{d}^{\prime ij}$ being the last one of $\und{c}^i$, $\und{c}^{\prime i}$ respectively), such that writing $T_{\und{d}^{ij}} T'_{\und{d}^{\prime ij}}(\und{\te}^i,\und{\te}^{\prime i},s_i) = (\und{\te}^{ij}, \und{\te}^{\prime ij}, s_{ij})$, we have

\begin{itemize}
\item $\und{\te}^{ij} \in \Sigma_{nr}^-$\,;
\item $B(s_{ij},c_4\ro^{1/2}) \subset E^*(\und{d}^{ij},
\und{d}^{\prime ij})$.
\end{itemize}

\noindent Concatenation of the $\und{c}^i$, $\und{c}^{\prime i}$ and 
$\und{d}^{ij}, \und{d}^{\prime ij}$ gives a family of words
$(\und{b}^{ij},\und{b}^{\prime ij})$ in $\Sigma(\ro) \times \Sigma'(\ro)$ with at least $\frac 14 c_3^2\,\ro^{-(d+d')}$ elements.

We now consider the perturbed operators. In this case $T^{\und{\om}}_{\und{c}^i} T'_{\und{c}^{\prime i}}(\und{\te},\und{\te}',s) = (\und{\te}^i,\und{\te}^{\prime i},s_i(\und{\om}))$ and by lemma \ref{lem:plg} the distance between $s_i(\und{\om})$ and $s_i$ is of order $c_5 \rho^{1-1/2k}$. Similarly one has 
\[
T^{\und{\om}}_{\und{d}^{ij}} T'_{\und{d}^{\prime ij}}(\und{\te}^i,\und{\te}^{\prime i},s_i(\und{\om})) = (\und{\te}^{ij}, \und{\te}^{\prime ij}, s_{ij}(\und{\om}))
\]
and again the distance between $s_{ij}(\und{\om})$ and $s_{ij}$ is of order $c_5 \rho^{1-1/2k}$.

Now we fix $(\und{\widetilde{\te}},\und{\widetilde{\te}}',\tilde{s})$ such that
$d(\und{\te},\und{\widetilde{\te}}) \le c_0 \ro^{1/2}$, $d(\und{\te}',\und{\widetilde\te}') \le c_0 \ro^{1/2}$, $|s-\tilde s|\le c_2\,\ro^{1/2}$ and  $\tilde{s} \in E(\widetilde{\und \te}, \widetilde{\und \te}')$.

\begin{lemma}\label{lem:exp}
If $c_{16}$ has been chosen sufficiently small, there are at least $\frac 16 c_3\,c_{19}^{-2}\, \ro^{-\frac{d+d'}{2k}}$ pairs $(\und{a},\und{a}') \in \Sigma_2 \times \Sigma_2'$ which are good and satisfy
$$
\left\| \chi_{T_{\und{a}}T'_{\und{a}'}(\widetilde{\und \te}, \widetilde{\und \te}',\tilde s)}\right\|_{L^2}^2 \le c_{16}^{-1}
$$
and such that at least $\frac 16 c_3c_{19}^{-1}\,\ro^{-(d+d')(\frac 12-\frac{1}{2k})}$ pairs $(\und{c}^i, \und{c}^{\prime i})$ start with $(\und{a}, \und{a}')$.
\end{lemma}

\begin{proof}
The proof is the same as in \cite{MY}, bearing in mind the different but similar bound in the number of bad pairs.
\end{proof}

We call the pairs verifying the properties of the previous lemma \emph{excellent pairs}. The following general lemma will be used later to estimate the measure of the union of the perturbed version of the sets $J(\und{b}^{ij},\und{b}^{\prime ij})$.

\begin{lemma}\label{lem:pb}
Let $J_{\alpha}, J_{\alpha}', K_{\alpha}$, $\alpha \in A$, be families of sets in $\C$ such that for some $\lambda, \ve, \nu, \sigma \in \R^+$, and $c_{\alpha}, c_{\alpha}' \in \C$:
\begin{itemize}
\item $B(c_{\alpha},\ve)\subset J_{\alpha} \subset B(c_{\alpha},\lambda \ve)$, $K_{\alpha} \subset J_{\alpha}$, $B(c_{\alpha}',\ve) \subset J_{\alpha}'$.
\item $d(c_{\alpha},c_{\alpha}')\leq \nu \ve$, $Leb(K_{\alpha})\geq \sigma^{-1} Leb(J_{\alpha})$.
\end{itemize}
Then
\[Leb\left(\bigcup_{\alpha \in A} J_{\alpha}'\right)\geq \frac{1}{9(\lambda+\nu)^2} Leb\left(\bigcup_{\alpha \in A} J_{\alpha}\right),\]
and
\[Leb\left(\bigcup_{\alpha \in A} K_{\alpha}\right) \geq \frac{\sigma^{-1}}{9\lambda ^2} Leb\left(\bigcup_{\alpha \in A} J_{\alpha}\right).\]
\end{lemma}

%The lemma is saying that if we have family of ball-type sets $J_{\alpha}$ and if we perturb them into a family $J_{\alpha}'$ with similar form and close to $J_{\alpha}$, then the union $\cup J_{\alpha}'$ has measure at least of the order of $\cup J_{\alpha}$. Moreover, the lemma also states that if we have subsets $K_{\alpha} \subset J_{\alpha}$ with comparable measure to $J_{\alpha}$ then the union $\cup K_{\alpha}$ is comparable to $\cup J_{\alpha}$.

\begin{proof}
Notice that $J_{\alpha} \subset B(c_{\alpha}', (\lambda +\nu )\ve )$. Let $\tilde{A}$ be a subset of $A$ such that the balls $\{B(c_{\alpha}', (\lambda +\nu )\ve )\}_{\alpha \in \tilde{A}}$ have the Vitali property. Thus
\begin{align*}
Leb\left( \bigcup_{\alpha \in A} J_{\alpha}\right) &\leq Leb\left( \bigcup_{\alpha \in A} B(c_{\alpha}', (\lambda +\nu )\ve )\right)\\ 
&\leq  Leb\left( \bigcup_{\alpha \in \tilde{A}} B(c_{\alpha}', 3(\lambda +\nu )\ve )\right)\\
&\leq [3(\lambda +\nu)]^2  Leb\left( \bigcup_{\alpha \in \tilde{A}} B(c_{\alpha}', \ve )\right)\\
&\leq 9(\lambda+\nu)^2  Leb\left( \bigcup_{\alpha \in A} B(c_{\alpha}', \ve )\right)\leq 9(\lambda+\nu)^2  Leb\left( \bigcup_{\alpha \in A} J_{\alpha}'\right),
\end{align*}
where in the passage from the second to the third line we use the fact that the sets $B(c_{\alpha}', \ve )$, $\alpha \in \ti{A}$, are disjoint. This proves the first inequality. For the second, use again Vitali to find a subset $A'$ of $A$ such that the balls $\{B(c_{\alpha}, \lambda \ve)\}_{\alpha \in A'}$ have the Vitali property. Then
\begin{align*}
Leb\left( \bigcup_{\alpha \in A} J_{\alpha} \right)&\leq Leb\left( \bigcup_{\alpha \in A'} B(c_{\alpha}, 3\lambda \ve) \right)\\
&\leq 9\lambda^2 \sum_{\alpha \in A'} Leb\left( B(c_{\alpha}, \ve) \right)\\
&\leq 9\lambda^2 \sigma \sum_{\alpha \in A'} Leb\left( K_{\alpha} \right)=9\lambda^2 \sigma\cdot Leb\left(\bigcup_{\alpha \in A'} K_{\alpha} \right) \leq 9\lambda^2 \sigma \cdot Leb\left(\bigcup_{\alpha \in A} K_{\alpha} \right).
\end{align*}
\end{proof}

\begin{lemma}\label{lem:j2}
Let $(\und{a}, \und{a}')$ be an excellent pair. Consider all the pairs $(\und{c}^i, \und{c}^{\prime i})$ described above which begin with $(\und{a}, \und{a}')$, and for each pair $(\und{c}^i, \und{c}^{\prime i})$ consider the corresponding pairs $(\und{b}^{ij},\und{b}^{\prime ij})$. Define the sets
\begin{align*}
J^{\und{\om}}(\und{b}^{ij},\und{b}^{\prime ij})&= \pi^{\und{\om}}_{\tb, \tb', s}(G^{\und{\om}}(\und{b}^{ij})\times G(\und{b}^{\prime ij})),\\
J_2^{\und{\om}}(\und{a}, \und{a}')&= \bigcup_{(\und{c}^i,\und{c}^{\prime i})} \bigcup_{(\und{b}^{ij},\und{b}^{\prime ij})} J^{\und{\om}}(\und{b}^{ij},\und{b}^{\prime ij}).
\end{align*}
Then
\[Leb\left(J_2^{\und{\om}}(\und{a}, \und{a}')\right) \gtrsim c_{16}\rho^{1/k},\]
and the constant can be chosen to be independent of $\uom$. 
\end{lemma}

\begin{proof}
First we make the following observation: 
\begin{equation}\label{eq:rpi}
\pi_{\und{\eta},\und{\eta}', w}\circ (f_{\und{d}},f_{\und{d}'})=A\circ \pi_{T_{\und{d}}T'_{\und{d}'}(\und{\eta},\und{\eta}', w)},
\end{equation}
where $A$ is an affine function with $|DA|\approx diam(G(\und{d}))$, and this holds for any $(\und{d},\und{d}')\in \Sigma(\alpha)\times \Sigma'(\alpha)$, for some $\alpha\in \R^+$, and any $(\und{\eta},\und{\eta}', w)$. 

For an excellent pair $(\und{a},\und{a}')$ we consider the associated pairs $(\und{c}^i,\und{c}^{\prime i})$ and the sets
\begin{align*}
J^{\und{\om}}(\und{c}^i, \und{c}^{\prime i})&= \pi^{\und{\om}}_{\tb, \tb', s}(G^{\und{\om}}(\und{c}^i)\times G(\und{c}^{\prime i})),\\
J(\und{c}^i, \und{c}^{\prime i})&= \pi_{\tb, \tb', s}(G(\und{c}^i)\times G(\und{c}^{\prime i})),\\
\tilde{J}(\und{c}^i, \und{c}^{\prime i})&= \pi_{\tilde\tb, \tilde\tb', \tilde{s}}(G(\und{c}^i)\times G(\und{c}^{\prime i})),\\
J_1^{\und{\om}}(\und{a}, \und{a}')&= \bigcup_{(\und{c}^i,\und{c}^{\prime i})} J^{\und{\om}}(\und{c}^i, \und{c}^{\prime i}).
\end{align*}
We will prove that the measure of $J^{\und{\om}}_1(\und{a}, \und{a}')$ is at least of the order $c_{16}\rho^{1/k}$. To do this, we use lemma \ref{lem:pb} to see that the measures of the sets $\bigcup J^{\und{\om}}(\und{c}^i, \und{c}^{\prime i})$, $\bigcup J(\und{c}^i, \und{c}^{\prime i})$ and $\bigcup \tilde{J}(\und{c}^i, \und{c}^{\prime i})$ are of the same order. It is clear from equation \eqref{eq:bt} that the sets $J^{\und{\om}}(\und{c}^i, \und{c}^{\prime i})$, $J(\und{c}^i, \und{c}^{\prime i})$ and $\tilde{J}(\und{c}^i, \und{c}^{\prime i})$ are all contained in, and contain, balls with radius of order $\rho^{1/2}$, and centered at the points $\pi^{\und{\om}}_{\tb, \tb', s}(c^{\und{\om}}_{\und{c}^i},c_{\und{c}^{\prime i}})$, $\pi_{\tb, \tb', s}(c_{\und{c}^i},c_{\und{c}^{\prime i}})$ and $\pi_{\tilde\tb, \tilde\tb', \tilde s}(c_{\und{c}^i},c_{\und{c}^{\prime i}})$ respectively. We remark that
\begin{align*}
|\pi^{\und{\om}}_{\tb, \tb', s}(c^{\und{\om}}_{\und{c}^i},c_{\und{c}^{\prime i}})-\pi_{\tb, \tb', s}(c_{\und{c}^i},c_{\und{c}^{\prime i}})|&= |s|\cdot |k^{\tb}(c_{\und{c}^i})-k^{\tb,\und{\om}}(c^{\und{\om}}_{\und{c}^i})|\\
&\leq |s|\cdot \left[|k^{\tb}(c_{\und{c}^i})-k^{\tb,\und{\om}}(c_{\und{c}^i})|+|k^{\tb,\und{\om}}(c_{\und{c}^i})-k^{\tb,\und{\om}}(c^{\und{\om}}_{\und{c}^i})|\right]\\
&\lesssim c_5 \rho^{1-1/2k} \lesssim \rho^{1/2},
\end{align*}
we also have
\begin{align*}
|\pi_{\tilde\tb, \tilde\tb', \tilde s}(c_{\und{c}^i},c_{\und{c}^{\prime i}})-\pi_{\tb, \tb', s}(c_{\und{c}^i},c_{\und{c}^{\prime i}})|&\leq |k^{\tilde\tb'}(c_{\und{c}^{\prime i}})-k^{\tb'}(c_{\und{c}^{\prime i}})|+|s-\tilde{s}|\cdot |k^{\tb}(c_{\und{c}^{i}})|+ |\tilde s|\cdot |k^{\tb}(c_{\und{c}^{i}})-k^{\tilde \tb}(c_{\und{c}^{i}})|\\
&\lesssim \rho^{1/2},
\end{align*}
given that $d(\tb,\tilde \tb)\leq c_0 \rho^{1/2}$, $d(\tb',\tilde \tb')\leq c_0 \rho^{1/2}$ and $|s-\tilde s|\leq c_2\rho^{1/2}$. All this allows us to conclude that we can use lemma \ref{lem:pb}.

Now we can estimate the measure of $J^{\und{\om}}_1(\und{a}, \und{a}^{\prime})$. In the following lines of equations we will be using: lemma \ref{lem:pb} for the first three lines, observation in equation \eqref{eq:rpi} for the fifth and sixth line, equation \eqref{eq:cauchy} in the seventh line, in the last line we use that $(\und{a},\und{a}')$ is an excellent pair and the fact that $diam(G(\und{c}^{i}/\und{a}))$ is of order $\rho^{ \frac{1}{2}-\frac{1}{2k}}$.

\begin{align*}
Leb\left(J^{\und{\om}}_1(\und{a}, \und{a}^{\prime})\right)&= Leb\left( \bigcup J^{\und{\om}}(\und{c}^i, \und{c}^{\prime i}) \right)\\
&\gtrsim Leb\left( \bigcup J(\und{c}^i, \und{c}^{\prime i}) \right)\\
&\gtrsim Leb\left( \bigcup \tilde{J}(\und{c}^i, \und{c}^{\prime i}) \right)\\
&= Leb\left( \pi_{\tilde\tb, \tilde\tb', \tilde{s}}\left(\bigcup G(\und{c}^i)\times G(\und{c}^{\prime i})\right) \right)\\
&= Leb\left( A\circ \pi_{T_{\und{a}}T'_{\und{a}'}(\tilde\tb, \tilde\tb', \tilde{s})}\left(\bigcup G(\und{c}^i/\und{a})\times G(\und{c}^{\prime i}/\und{a}')\right) \right)\\
&\approx diam(G(\und{a}))^2 \cdot Leb\left( \pi_{T_{\und{a}}T'_{\und{a}'}(\tilde\tb, \tilde\tb', \tilde{s})}\left(\bigcup G(\und{c}^i/\und{a})\times G(\und{c}^{\prime i}/\und{a}')\right) \right)\\
&\gtrsim diam(G(\und{a}))^2 \cdot \mu_d \times \mu_{d'}\left( \bigcup G(\und{c}^i/\und{a})\times G(\und{c}^{\prime i}/\und{a}')\right)^2 \cdot \left\| \chi_{T_{\und{a}}T'_{\und{a}'}(\tilde\tb, \tilde\tb', \tilde{s})} \right\|_{L^2}^{-2}\\
&\gtrsim c_{16}\rho^{1/k} \cdot \left(\rho^{\left(\frac{1}{2}-\frac{1}{2k} \right)(d+d')}\cdot \ro^{-(d+d')(\frac 12-\frac{1}{2k})}\right)^2 \approx c_{16}\rho^{1/k}
\end{align*}

We now consider the sets 
\begin{align*}
J^{\und{\om}}(\und{b}^{ij},\und{b}^{\prime ij})&= \pi^{\und{\om}}_{\tb, \tb', s}(G^{\und{\om}}(\und{b}^{ij})\times G(\und{b}^{\prime ij})),\\
J_1^{\und{\om}}(\und{c}^i, \und{c}^{\prime i})&= \bigcup_{(\und{b}^{ij},\und{b}^{\prime ij})} J^{\und{\om}}(\und{b}^{ij}, \und{b}^{\prime ij}),\\
J_2^{\und{\om}}(\und{a}, \und{a}')&= \bigcup_{(\und{c}^i,\und{c}^{\prime i})} J^{\und{\om}}_1(\und{c}^i, \und{c}^{\prime i}).
\end{align*}

We will estimate the measure of $J_2^{\und{\om}}(\und{a}, \und{a}')$. Notice that $J^{\und{\om}}_1(\und{c}^i, \und{c}^{\prime i}) \subset J^{\und{\om}}(\und{c}^i, \und{c}^{\prime i})$ and if we are able to prove that $Leb\left(J^{\und{\om}}_1(\und{c}^i, \und{c}^{\prime i}) \right) \geq \sigma^{-1} \cdot Leb\left( J^{\und{\om}}(\und{c}^i, \und{c}^{\prime i})\right)$, for some constant $\sigma$, then we can use lemma \ref{lem:pb} with $K_{\alpha}$ being $J^{\und{\om}}_1(\und{c}^i, \und{c}^{\prime i})$ and $J_{\alpha}$ being $J^{\und{\om}}(\und{c}^i, \und{c}^{\prime i})$ to conclude that
\[Leb\left(J_2^{\und{\om}}(\und{a}, \und{a}')\right)=Leb\left(\bigcup J^{\und{\om}}_1(\und{c}^i, \und{c}^{\prime i}) \right)\gtrsim Leb\left(\bigcup J^{\und{\om}}(\und{c}^i, \und{c}^{\prime i})\right) \gtrsim c_{16}\rho^{1/k}.\]
To prove that there is such $\sigma$ we will proceed similarly to what we did when estimating $Leb\left(J^{\und{\om}}_1(\und{a}, \und{a}^{\prime})\right)$. Note that $T_{\und{c}^{i}}T'_{\und{c}^{\prime i}}(\tb,\tb',s)=(\tb^i,\tb^{\prime i}, s_i) \in \tilde \LL$ and then there exists $(\tilde{\tb}^i,\tilde{\tb}^{\prime i}, \tilde{s}_i)$ such that $|s_i-\tilde{s}_i|\leq c_2 \rho^{1/2}$, $d(\tilde{\tb}^i,\tb^i)\leq c_0 \rho^{1/2}$, $d(\tilde{\tb}^{\prime i},\tb^{\prime i})\leq c_0 \rho^{1/2}$ and
\begin{equation}\label{eq:eschi}
\left\|\chi_{\tilde{\tb}^i,\tilde{\tb}^{\prime i}, \tilde{s}_i} \right\|_{L^2}^{2}\leq c_{15}.
\end{equation}

The sets 
\[
\pi^{\und{\omega}}_{\tb^i,\tb^{\prime i}, s_i(\und{\om})}\left( G^{\und{\om}}(\und{d}^{ij})\times G(\und{d}^{\prime ij})\right),
\pi_{\tb^i,\tb^{\prime i}, s_i}\left( G(\und{d}^{ij})\times G(\und{d}^{\prime ij})\right) \text{ and } \pi_{\tilde{\tb}^i,\tilde{\tb}^{\prime i}, \tilde{s}_i}\left( G(\und{d}^{ij})\times G(\und{d}^{\prime ij})\right)
\]
are all contained in, and contain, balls with radius of order $\rho^{1/2}$, and centered at the points 
\[\pi^{\und{\omega}}_{\tb^i,\tb^{\prime i}, s_i(\und{\om})}( c_{\und{d}^{ij}}^{\und{\om}},c_{\und{d}^{\prime ij}}), \pi_{\tb^i,\tb^{\prime i}, s_i}( c_{\und{d}^{ij}},c_{\und{d}^{\prime ij}}) \text{ and } \pi_{\tilde{\tb}^i,\tilde{\tb}^{\prime i}, \tilde{s}_i}( c_{\und{d}^{ij}},c_{\und{d}^{\prime ij}})
\]
respectively. By lemma \ref{lem:plg} we know that 
$$\|k^{\tb^i,\und{\om}}-k^{\tb^i}\|_{C^0}\lesssim c_5\rho^{1-1/2k},\, \,|c_{\und{d}^{ij}}^{\und{\om}}-c_{\und{d}^{ij}}|\lesssim c_5\rho^{1+1/2k},$$
on the other hand we also have $|s_i-s_i(\und{\om})|\lesssim \rho^{1/2}$, $\|k^{\tb^i}-k^{\tilde{\tb}^i}\|_{C^0}\lesssim \rho^{1/2}$ and $\|k^{\tb^{\prime i}}-k^{\tilde{\tb}^{\prime i}}\|_{C^0}\lesssim \rho^{1/2}$, thus we can conclude that the distance between any two centers is of order less than $\rho^{1/2}$. Therefore, we can apply lemma \ref{lem:pb} taking $J_{\alpha}$ as one of the families 
$$\pi^{\und{\omega}}_{\tb^i,\tb^{\prime i}, s_i(\und{\om})}\left( G^{\und{\om}}(\und{d}^{ij})\times G(\und{d}^{\prime ij})\right),\,\, \pi_{\tb^i,\tb^{\prime i}, s_i}\left( G(\und{d}^{ij})\times G(\und{d}^{\prime ij})\right),\,\, \pi_{\tilde{\tb}^i,\tilde{\tb}^{\prime i}, \tilde{s}_i}\left( G(\und{d}^{ij})\times G(\und{d}^{\prime ij})\right)$$
and $J_{\alpha}'$ as another of these families.

Using the previous analysis together with equations \eqref{eq:rpi}, \eqref{eq:cauchy}, \eqref{eq:eschi} and the fact that the number of $(\und{d}^{ij}, \und{d}^{\prime ij})$ is a positive proportion of $\Sigma(\rho^{1/2})\times \Sigma'(\rho^{1/2})$ we obtain 
\begin{align*}
Leb\left(J^{\und{\om}}_1(\und{c}^i, \und{c}^{\prime i})\right)&= Leb\left( \bigcup J^{\und{\om}}(\und{b}^{ij}, \und{b}^{\prime ij})\right)\\
&=Leb\left( \pi^{\und{\om}}_{\tb, \tb', s}\left(\bigcup G^{\und{\om}}(\und{b}^{ij})\times G(\und{b}^{\prime ij})\right) \right)\\
&= Leb\left( A\circ \pi^{\und{\om}}_{T^{\und{\om}}_{\und{c}^{i}}T'_{\und{c}^{\prime i}}(\tb, \tb', s)}\left(\bigcup G^{\und{\om}}(\und{d}^{ij})\times G(\und{d}^{\prime ij})\right) \right)\\
&\approx Leb\left(J(\und{c}^i, \und{c}^{\prime i})\right) \cdot Leb\left( \bigcup \pi^{\und{\omega}}_{\tb^i,\tb^{\prime i}, s_i(\und{\om})}\left( G^{\und{\om}}(\und{d}^{ij})\times G(\und{d}^{\prime ij})\right) \right)\\
&\approx Leb\left(J(\und{c}^i, \und{c}^{\prime i})\right) \cdot Leb\left( \bigcup \pi_{\tb^i,\tb^{\prime i}, s_i}\left( G(\und{d}^{ij})\times G(\und{d}^{\prime ij})\right) \right)\\
&\approx Leb\left(J(\und{c}^i, \und{c}^{\prime i})\right) \cdot Leb\left( \bigcup \pi_{\tilde{\tb}^i,\tilde{\tb}^{\prime i}, \tilde{s}_i}\left( G(\und{d}^{ij})\times G(\und{d}^{\prime ij})\right) \right)\\
&\gtrsim Leb\left(J(\und{c}^i, \und{c}^{\prime i})\right) \cdot \mu_d \times \mu_{d'}\left( \bigcup G(\und{d}^{ij})\times G(\und{d}^{\prime ij})\right)^2 \cdot \left\| \chi_{\tilde{\tb}^i,\tilde{\tb}^{\prime i}, \tilde{s}_i} \right\|_{L^2}^{-2}\\
&\gtrsim Leb\left(J(\und{c}^i, \und{c}^{\prime i})\right) \cdot \left(\rho^{\frac{1}{2}(d+d')}\cdot \ro^{-\frac 12 (d+d')}\right)^2 \approx Leb\left(J(\und{c}^i, \und{c}^{\prime i})\right).
\end{align*}

This guarantees the existence of the desired constant $\sigma$ and finishes the proof of the lemma.
\end{proof}

Now we can prove proposition \ref{prop:bigL}. Consider the function
\[\varphi^{\uom}_2= \sum_{(\und{a},\und{a}')}1_{J^{\und{\om}}_2(\und{a},\und{a}')},\]
where $1_B$ means indicator function of the set $B$ and the sum is over all excellent pairs. We want to estimate the measure of the set
\[X^{\uom}=\{t\in \C: \varphi^{\uom}_2(t)\geq c''c_{16}^2 \rho^{-\frac{1}{2k}(d+d'-2)}\},\]
where $c''$ is a constant defined in the following way. Suppose that we have two excellent pairs with the same first coordinate $(\und{a},\und{a}')$, $(\und{a},\tilde{\und{a}}')$ and such that $J^{\und{\om}}(\und{a},\und{a}')\cap J^{\und{\om}}(\und{a},\tilde{\und{a}}')\neq \emptyset$. Then
\[k^{\tb'}(y)-sk^{\tb, \und{\om}}(x)=k^{\tb'}(\tilde{y})-sk^{\tb, \und{\om}}(\tilde{x}),\]
for some $(x,y)\in G^{\und{\om}}(\und{a})\times G(\und{a}')$, $(\tilde{x},\tilde{y})\in G^{\und{\om}}(\und{a})\times G(\tilde{\und{a}}')$. Thus
\[|y-\tilde{y}|\approx |x-\tilde{x}|\lesssim \rho^{1/2k},\]
which shows that
\[d(G(\und{a}'),G(\tilde{\und{a}}'))\lesssim \rho^{1/2k}.\]
This implies that if we fix $(\und{a},\und{a}')$, then the number of possible pairs $(\und{a},\tilde{\und{a}}')$ such that $J^{\und{\om}}(\und{a},\und{a}')\cap J^{\und{\om}}(\und{a},\tilde{\und{a}}')\neq \emptyset$ is bounded by a uniform constant, independent of $\rho$ and $(\und{a},\und{a}')$, we denote this constant by $c''$ (this last statement is a consequence of lemma 1.2.3 in \cite{Z}).

Notice that since $s\in J_R$ then $\varphi^{\uom}_2$ is supported in a ball of radius proportional to $1+e^R$ centered at $0$, thus there is a constant $c$ such that
\[\int \varphi^{\uom}_2 dt \leq (\sup \varphi^{\uom}_2) \cdot Leb(X^{\uom})+c c''c_{16}^2 \rho^{-\frac{1}{2k}(d+d'-2)}.\]
Now we estimate $\sup \varphi^{\uom}_2$ from above and $\int \varphi^{\uom}_2 dt$ from below. By lemmas \ref{lem:j2} and \ref{lem:exp} there is a constant $c'$ such that
\begin{equation}\label{eq:chfi}
\int \varphi^{\uom}_2 dt\geq c' c_{16}\rho^{1/k}\cdot \ro^{-\frac{d+d'}{2k}}=c'c_{16}\rho^{-\frac{1}{2k}(d+d'-2)}.
\end{equation}
Let $x\in \C$ and excellent pairs $(\und{a},\und{a}')$, $(\tilde{\und{a}},\tilde{\und{a}}')$ such that $x\in J^{\und{\om}}(\und{a},\und{a}')\cap J^{\und{\om}}(\tilde{\und{a}},\tilde{\und{a}}')$. Remember that
\[|\pi^{\und{\om}}_{\tb,\tb',s}(c^{\und{\om}}_{\und{a}},c_{\und{a}'})-\pi_{\tb,\tb',s}(c_{\und{a}},c_{\und{a}'})|\lesssim c_5\rho^{1-1/2k}=o(\rho^{1/2k}),\]
where the $o$ notation means that, once we have chosen $c_5$, we can choose any $\ve>0$ and $c_5\rho^{1-1/2k}\leq \ve \rho^{1/2k}$ will hold provided $\rho$ is small enough. With this in mind we obtain
\begin{align*}
|\pi_{\tb,\tb',s}(c_{\und{a}},c_{\und{a}'})-\pi_{\tb,\tb',s}(c_{\tilde{\und{a}}},c_{\tilde{\und{a}}'})| &\leq o(\rho^{1/2k})+|\pi^{\und{\om}}_{\tb,\tb',s}(c^{\und{\om}}_{\und{a}},c_{\und{a}'})-x|+|x-\pi^{\und{\om}}_{\tb,\tb',s}(c^{\und{\om}}_{\tilde{\und{a}}},c_{\tilde{\und{a}}'})|\\
&\leq o(\rho^{1/2k})+2c_{19} \rho^{1/2k}.
\end{align*}
Given that $(\und{a},\und{a}')$, $(\tilde{\und{a}},\tilde{\und{a}}')$ are excellent, we conclude that there can be no more than $c_{16}^{-1}\rho^{-\frac{1}{2k}(d+d'-2)}$ excellent pairs intersecting at $x$. Since $\varphi^{\uom}_2= \sum_{(\und{a},\und{a}')} 1_{J^{\und{\om}}(\und{a},\und{a}')}$ we get that
\[\varphi^{\uom}_2(x)\leq c_{16}^{-1}\rho^{-\frac{1}{2k}(d+d'-2)}.\]
We are ready to bound $Leb(X^{\uom})$, using equation \eqref{eq:chfi} and the previous estimates
\[c_{16}^{-1}\rho^{-\frac{1}{2k}(d+d'-2)}Leb(X^{\uom})+cc''c_{16}^2\rho^{-\frac{1}{2k}(d+d'-2)}\geq c'c_{16}\rho^{-\frac{1}{2k}(d+d'-2)},\]
and from this we get
\[Leb(X^{\uom})\geq c_{16}^2(c'-cc''c_{16}).\]
We fix $c_{16}$ small enough such that $c_{17}:=c_{16}^2(c'-cc''c_{16})>0$.

To finish the proof of the proposition we will show that $\{(\tb,\tb',s,t):t\in X^{\uom^*}\}\subset L_{\uom}^{-1}(\ute,\ute',s)$. Let $t\in X^{\uom^*}$, then there are at least $\left[c''c_{16}^2\rho^{-\frac{1}{2k}(d+d'-2)}\right]$ excellent pairs $(\und{a},\und{a}')$, each one with an associated pair $(\und{b}^{ij},\und{b}^{\prime ij})$ which starts with $(\und{a},\und{a}')$ and such that $t\in J^{\und{\om}^*}(\und{b}^{ij},\und{b}^{\prime ij})$. By the definition of $c''$, we can extract from this family of excellent pairs a subfamily $(\und{a}^l,\und{a}^{\prime l})$, $l=1,...,\left[c_{16}^2\rho^{-\frac{1}{2k}(d+d'-2)}\right]$, such that all firsts coordinates are different. For $(\und{a}^l,\und{a}^{\prime l})$ we denote the associated pair $(\und{b}^{ij},\und{b}^{\prime ij})$ by $(\und{b}^l,\und{b}^{\prime l})$. We will prove the pairs $(\und{b}^1,\und{b}^{\prime 1}),...,(\und{b}^N,\und{b}^{\prime N})$ have the properties necessary to guarantee that $(\tb,\tb',s,t)\in L_{\uom}^{-1}(\ute,\ute',s)$. Write
\[T^{\und{\om}^*}_{\und{b}^l}T'_{\und{b}^{\prime l}}(\tb, \tb', s,t)=(\tb^l, \tb^{\prime l}, s_l,t_l)\]
then:
\begin{itemize}
\item[(i)] Since all firsts coordinates of the excellent pairs are different we conclude that $\und{b}^1,...,\und{b}^N$ are pairwise independent.
\item[(ii)] By the way in which $\und{d}^{ij}$ was defined we get that all $\tb^l\in \Sigma^-_{nr}$.
\item[(iii)'] By the scale recurrence lemma we know that $B(s_{ij},c_4\rho^{1/2})\subset E^*(d^{ij},d^{\prime ij})$. We also know that $|s_{ij}-s_{ij}(\und{\om}^*)|\lesssim c_5\rho^{1-1/2k}=o(\rho^{1/2})$, then, if $\rho$ is small enough, we have
\[\{\tilde{s}: |\tilde{s}-s_{ij}(\und{\om}^*)|\leq \frac{3}{4}c_4\rho^{1/2}\}\subset E^*(d^{ij},d^{\prime ij}).\]
We conclude that $(\tb^l, \tb^{\prime l}, \tilde{s})\in \tilde \LL $ if $|\tilde{s}-s_{l}|\leq \frac{3}{4}c_4\rho^{1/2}$ (remember that for every $l$ there is $i,j$ such that $s_l=s_{ij}(\und{\om}^*)$).
\item[(iv)'] Given that $t\in J^{\und{\om}^*}(\und{b}^l,\und{b}^{\prime l})$, there is $(x,y)\in \C^2$ such that
\begin{align*}
t&=k^{\tb'}(f_{\und{b}^{\prime l}}^{\und{\om}^*}(y))-sk^{\tb,\und{\om}^*}(f_{\und{b}^l}^{\und{\om}^*}(x))\\
F^{\tb'}_{\und{b}^{\prime l}}(k^{\tb'\und{b}^{\prime l}}(y))&= t+s\cdot F^{\tb,\und{\om}^*}_{\und{b}^l}(k^{\tb\und{b}^l,\und{\om}^*}(x))\\
k^{\tb'\und{b}^{\prime l}}(y)&=\left( F^{\tb'}_{\und{b}^{\prime l}} \right)^{-1}\left(t+s\cdot F^{\tb,\und{\om}^*}_{\und{b}^l}(k^{\tb\und{b}^l,\und{\om}^*}(x))\right)\\
k^{\tb'\und{b}^{\prime l}}(y)&=t_l+s_l\cdot k^{\tb\und{b}^l,\und{\om}^*}(x).
\end{align*}
Since $s_l\in J_R$ we conclude that $|t_l|\leq 1+e^R$.
\end{itemize}

\section{Proof of Proposition \ref{prop:probes}}\label{sec:proofprobes}

In this section we will prove proposition \ref{prop:probes}. Given $u=(\tb,\tb',s, t) \in \tilde{\LL}\times \C$, remember the decomposition $\uom=(\uom',\uom'')$ where $\uom'\in \mathbb{D}^{\Sigma_1(\tb)}$ and $\uom'' \in \mathbb{D}^{\Sigma_1\setminus \Sigma_1(\tb)}$. Recall that the set $\Sigma_1(\tb)$ is given by the words in $\Sigma_1$ starting with the same word, in $\Sigma(\rho^{1/2k})$, in which $\tb$ finishes. In the same way as in \cite{MY}, one uses Fubini's theorem to reduce the proof of proposition \ref{prop:probes} to proving
\begin{equation}\label{eq:fubini1}
\mathbb{P}'(\Omega'\setminus\Omega^{\prime 0}(u))\leq \exp(-c_7\rho^{-\frac{1}{2k}(d+d'-2)}),
\end{equation}
where we have fixed $\uom''$ such that $u\in \LL^1_{(\und{0},\uom'')}$ and $\Omega'=\mathbb{D}^{\Sigma_1(\tb)}$, $\Omega^{\prime 0}(u)=\{\uom':(\uom',\uom'')\in \Omega^0(u)\}$, $\mathbb{P}'$ is normalized Lebesgue measure in $\Omega'$.

The fact that $u=(\tb,\tb',s, t)\in \LL^1_{(\und{0},\uom'')}$ means that $(\tb,\tb',s)\in \tilde\LL$ and there is $(\tilde\tb,\tilde{\tb}',\tilde s, \tilde t)$ for which
\[d(\tb,\tilde\tb)<2\rho^{5/2},\, d(\tb',\tilde{\tb}')<2\rho^{5/2},\, |s-\tilde s|<\rho,\, |t-\tilde t|<\rho,\]
$(\tilde\tb,\tilde{\tb}',\tilde s)\in\tilde\LL$ and $\tilde t\in L_{(\und{0},\uom'')}^0(\tilde\tb,\tilde{\tb}',\tilde s)$. Notice that $\Sigma_1(\tb)=\Sigma_1(\tilde \tb)$, moreover $u\in \LL^1_{(\und{0},\uom'')}$ if and only if $u\in \LL^1_{(\und{\om}',\uom'')}$ for any $\uom' \in \Sigma_1(\tb)$.

Next, $\tilde t\in L_{(\und{0},\uom'')}^0(\tilde\tb,\tilde{\tb}',\tilde s)$ means that there are pairs $(\und{b}^1,\und{b}^{\prime 1}),...,\, (\und{b}^N,\und{b}^{\prime N})$ in $\Sigma(\rho)\times \Sigma'(\rho)$ such that if we set
\[T^{(\und{0},\uom'')}_{\und{b}^i}T'_{\und{b}^{\prime i}}(\tilde\tb,\tilde{\tb}',\tilde s, \tilde t)=(\tilde{\tb}^i,\tilde{\tb}^{\prime i},\tilde{s}_i, \tilde{t}_i)\]
then:
\begin{itemize}
\item[(i)] the words $\und{b}^1,...,\, \und{b}^N$ are pairwise independent;
\item[(ii)] $\tilde{\tb}^i\in \Sigma^-_{nr}$;
\item[(iii)] $(\tilde{\tb}^i,\tilde{\tb}^{\prime i},\tilde{s}_{i}')\in \tilde{\LL}$ if $|\tilde{s}_{i}-\tilde{s}_{i}'|<\frac{2}{3}c_4\rho^{1/2}$;
\item[(iv)] $|\tilde{t}_i|\leq 2(1+e^R)$.
\end{itemize}

Let $\und{a}$ be the word in $\Sigma(\rho^{1/2k})$ in which $\tb$ ends, for each $\und{b}^i$ define
$\und{a}^i$ in $\Sigma_1(\tb)$ given by the concatenation of $\und{a}$ and a word at the beginning of $\und{b}^i$, in such a way that $\und{a}^i\in \Sigma_1(\tb)$ (this can be done since $\tb\und{b}^i \in \Sigma_{nr}^{-}$). Notice that the independence of the words $\und{b}^i$ imply that the words $\und{a}^i$ are all different.

Now we consider the decomposition of $\uom'\in \Omega'$ as $\uom'=(\om_1,...,\om_N, \tilde{\uom}')$, where $\tilde{\uom}'\in \mathbb{D}^{\Sigma_1(\tb)\setminus\{a_1,...,a_N\}}$ and $\om_i$ is the component of $\uom'$ corresponding to $\und{a}^i$. We use again Fubini's theorem to reduce the proof of equation \eqref{eq:fubini1} to a similar formula in a smaller space. For $\tilde{\uom}'$ fixed, we will prove that the set of $(\om_1,..,\om_N)$ such that $\uom'\notin \Omega^{\prime 0}(u)$
has measure $\leq \exp(-c_7\rho^{-\frac{1}{2k}(d+d'-2)})$.

To prove the desired inequality, we will prove that for each $\om_i$ there is a set with positive measure such that whenever $\om_i$ is in this set we have $\uom'\in \Omega^{\prime 0}(u)$ (no matter the value of $\om_j$, $j\neq i$). More precisely, we will prove that if $\om_i$ is in this set then $\und{b}^i$, $\und{b}^{\prime i}$ verify that if we set 
$$T_{\und{b}^i}^{\uom}T'_{\und{b}^{\prime i}}(u)=(\tb^i,\tb^{\prime i},s_i(\uom),t_i(\uom))$$
then:
\begin{itemize}
\item[(i)]  $(\tb^i,\tb^{\prime i},s_i')\in \tilde\LL$ if $|s_i'-s_i(\uom)|<\frac{1}{2}c_4\rho^{1/2}$;
\item[(ii)] $t_i(\uom)\in L^0_{\uom}(\tb^i,\tb^{\prime i},s_i(\uom))$.
\end{itemize}
The first property can be easily obtained. In fact, we already know (from (iii) above) that $(\tilde{\tb}^i,\tilde{\tb}^{\prime i},\tilde{s}_{i}')\in \tilde{\LL}$ if $|\tilde{s}_{i}-\tilde{s}_{i}'|<\frac{2}{3}c_4\rho^{1/2}$. Notice that since $d(\tb^i,\tilde{\tb}^i)\lesssim \rho^{7/2}$, $d(\tb^{\prime i},\tilde{\tb}^{\prime i})\lesssim \rho^{7/2}$ then $\tilde{\tb}^i\in \Sigma^-_{nr}$ and the fiber of $\tilde\LL$ over $(\tb^i,\tb^{\prime i})$ is the same as the one over $(\tilde{\tb}^{i},\tilde{\tb}^{\prime i})$, thus we only need to estimate $|\tilde{s}_{i}-s_i(\uom)|$. Using that $|s-\tilde s|<\rho$, $d(\tb,\tilde{\tb})<2\rho^{5/2}$, $d(\tb',\tilde{\tb}')<2\rho^{5/2}$ and lemma \ref{lem:plg} one gets that $|\tilde{s}_{i}-s_i(\uom)|=o(\rho^{1/2})$, then choosing $\rho$ sufficiently small gives the desired property (for any value of $\om_i$).

For the second property we choose $\overline{\tb}\in \Sigma^-$ such that $d(\overline{\tb},\tb)<\rho^{5/2}$ (then $\overline{\tb}$ ends with $\und{a}$) and such that it does not contain $\und{a}$ anywhere else (this is possible since $\tb\in\Sigma^-_{nr}$). Set
\[T^{\uom}_{\und{b}^i} T'_{\und{b}^{\prime i}}(\overline{\tb},\tb',s,t)=(\overline{\tb}^i,\tb^{\prime i},\overline{s}_i(\uom),\overline{t}_i(\uom)).\]
We will prove the following lemmas:

\begin{lemma}\label{lem:si}
Once $\uom''$ has been fixed, the number $\overline{t}_i(\uom)$ only depends on $\om_i$, not in $\tilde{\uom}'$ or $\om_j$ for $j\neq i$. Moreover, if $c_5$ is big enough then there is a constant $c_7'>0$ such that
\[Leb\left(\{\om_i:\, \overline{t}_i(\uom)\in L^{-1}_{\hat{\uom}}(\tb^i,\tb^{\prime i},s_i(\hat{\uom}))  \}\right)\geq c_7',\]
where $\hat{\uom}=(\und{0}, \tilde{\uom}',\uom'')$.
\end{lemma}

\begin{lemma}\label{lem:tib}
If $\overline{t}_i(\om_i)\in L^{-1}_{\hat{\uom}}(\tb^i,\tb^{\prime i},s_i(\hat{\uom}))$ then $t_i(\uom)\in L^{0}_{\uom}(\tb^i,\tb^{\prime i},s_i(\uom))$.
\end{lemma}
These two lemmas imply that
\[Leb(\{(\om_1,...,\om_N):\,\,(\om_1,...,\om_N,\tilde\uom',\uom'')\notin \Omega^{0}(u)\})\leq (1-c_7')^{N}=e^{N\log (1-c_7')},\]
this finishes the proof of proposition \ref{prop:probes}.
\begin{proof}[Proof of lemma \ref{lem:si}:]
From equation \eqref{eq:formularenorma} we know that
\[\overline{t}_i(\uom)= \left(DF^{\tb'}_{\und{b}^{\prime i}}\right)^{-1}\cdot (t+sc^{\overline{\tb},\uom}_{\und{b}^{i}}-c^{\tb'}_{\und{b}^{\prime i}}).\]
The dependency on $\uom$ is on the term $c^{\overline{\tb},\uom}_{\und{b}^{i}}$. Let $\hat{\tb}$ such that $\overline{\tb}=\hat{\tb}\und{a}$, note that $\hat{\tb}$ does not contain the word $\und{a}$, let $b$ be the last letter of $\und{b}^i$, we have
\begin{equation}\label{eq:c}
c^{\overline{\tb},\uom}_{\und{b}^{i}}= k^{\overline{\tb},\uom}(c^{\uom}_{\und{b}^{i}})=k^{\hat{\tb}\und{a},\uom}(f^{\uom}_{\und{b}^i}(c^{\uom}_{b}))= \left(F^{\hat{\tb},\uom}_{\und{a}}\right)^{-1}\circ k^{\hat{\tb},\uom}(f^{\uom}_{\und{a}\und{b}^i}(c^{\uom}_{b})).
\end{equation}
We will prove that, assuming $\uom''$ is fixed, this expression only depends on $\om_i$ and not in $\om_j$ for $j\neq i$ or $\tilde{\uom}'$, and it depends in a very specific way. Remember, see remark \ref{rem:per}, that the base points $c^{\uom}_a$ were chosen to be pre-periodic points and that in fact they do not depend on $\uom$, i.e. $c^{\uom}_a=c_a$. Notice that $\overline{\tb}^i=\overline{\tb}\und{b}^i$ and $d(\overline{\tb}^i,\tb^i)\lesssim \rho^{7/2}$, then $\overline{\tb}^i\in \Sigma^-_{nr}$ and $\und{a}$ appears only once in $\overline{\tb}^i$. We will study the dependency on $\om_i$ for the different terms in equation \eqref{eq:c}:
\begin{itemize}
\item  Let $\und{\alpha}$ be a finite word at the end of $\overline{\tb}^i$ strictly shorter than $\und{a}\und{b}^i$. It is easy to prove by induction that $f^{\uom}_{\und{\alpha}}(c^{\om}_b)$ does not depend on $\uom'$. Indeed, suppose that $\und{\alpha}$ and $\und{\beta}=(c,d)\und{\alpha}$ are two such consecutive words. Assume $f^{\uom}_{\und{\alpha}}(c^{\om}_b)$ does not depend on $\uom'$, we have 
\[f^{\uom}_{\und{\beta}}(c^{\om}_b)=f^{\uom}_{(c,d)}\left( f^{\uom}_{\und{\alpha}}(c^{\om}_b)\right).\]
If the word $\und{\beta}$ is shorter than some word in $\Sigma(\ro^{1/2k})$, then $f^{\uom}_{\und{\alpha}}(c^{\om}_b)\in G^{\uom}(\und{\gamma})\subset V_{\rho}(G(\und{\gamma}))$ such that $(c,d)\und{\gamma}\in \Sigma(\rho^{1/k})$ and $\und{\gamma}$ contains a word in $\Sigma(\rho^{1/3k})$ repeated (remember that $c^{\om}_b$ is pre-periodic, with a uniform bounded period). Thus $(c,d)\und{\gamma}$ cannot belong to $\Sigma_1$ and using equation \eqref{f} one gets that $f^{\uom}_{\und{\beta}}(c^{\om}_b)$ does not depend on $\uom'$ (in fact, in this case it does not depend on $\uom$).

If the word $\und{\beta}$ is longer than any word in $\Sigma(\rho^{1/2k})$ then it begins with a word in $\Sigma(\rho^{1/2k})$ which can not be $\und{a}$. This implies that $f^{\uom}_{\und{\alpha}}(c^{\om}_b)\in G^{\uom}(\und{\gamma})\subset V_{\rho}(G(\und{\gamma}))$ such that $(c,d)\und{\gamma}\in \Sigma(\rho^{1/k})$, and the word $(c,d)\und{\gamma}$ is not in $\Sigma_1(\tb)$. Thus $f^{\uom}_{\und{\beta}}(c^{\om}_b)$ does not depend on $\uom'$.

\item Let $(a_0,a_1)$ be the first two letters of $\und{a}$ and define $\und{\alpha}_0$ by $\und{a}\und{b}^i=(a_0,a_1)\und{\alpha}_0$. Define $x_0=f_{a_0,a_1}(f^{\uom}_{\und{\alpha}_0}(c^{\uom}_b))$, we already proved that $x_0$ does not depend on $\uom'$. We have
\[f^{\uom}_{\und{a}\und{b}^i}(c^{\uom}_b)=f^{\uom}_{a_0,a_1}(f^{\uom}_{\und{\alpha}_0}(c^{\uom}_b))=x_0+c_5\rho^{1+1/2k}\om_i.\]
\item We now study $k^{\hat{\tb},\uom}$. By definition
\[k^{\hat{\tb},\uom}(z)=\lim_{n\to \infty} Df^{\uom}_{\hat{\tb}_n}(c_{\hat{\theta}_0})^{-1} (f^{\uom}_{\hat{\tb}_n}(z)-f^{\uom}_{\hat{\tb}_n}(c_{\hat{\theta}_0})).\]
Using the same arguments as before we see that, since $\hat{\tb}$ does not contain $\und{a}$ and $c_{\hat{\theta}_0}$ is pre-periodic, $f^{\uom}_{\hat{\tb}_n}(c_{\hat{\theta}_0})$ does not depend on $\und{\om}'$. This also happens for $z$ in a neighborhood of $c_{\hat{\theta}_0}$, hence $Df^{\uom}_{\hat{\tb}_n}(c_{\hat{\theta}_0})$ is independent of $\uom'$. Again, the same arguments prove that if $z\in G^{\uom}(\und{a})$ then $f^{\uom}_{\hat{\tb}_n}(z)=f^{(\und{0},\uom'')}_{\hat{\tb}_n}(z)$. We conclude that $$k^{\hat{\tb},\uom}(z)=k^{\hat{\tb},(\und{0},\uom'')}(z),\text{ and } Dk^{\hat{\tb},\uom}(z)=Dk^{\hat{\tb},(\und{0},\uom'')}(z),$$
for all $z\in G^{\uom}(\und{a})$.
\item We now treat $F^{\hat{\tb},\uom}_{\und{a}}$. One has that
\[F^{\hat{\tb},\uom}_{\und{a}}(0)=c^{\hat{\tb},\uom}_{\und{a}}=k^{\hat{\tb},\uom}(c^{\uom}_{\und{a}}),\,\,DF^{\hat{\tb},\uom}_{\und{a}}=Dk^{\hat{\tb},\uom}(c^{\uom}_{\und{a}})\cdot Df^{\uom}_{\und{a}}(c^{\uom}_a),\]
where $a$ is the last letter in $\und{a}$. Given that $\und{a}\in \Sigma(\rho^{1/2k})$ and $c^{\uom}_a=c_a$ is pre-periodic we conclude that $c^{\uom}_{\und{a}}=f^{\uom}_{\und{a}}(c^{\uom}_a)$ and $Df^{\uom}_{\und{a}}(c^{\uom}_a)$ are independent of $\uom$. On the other hand, given that $c^{\uom}_{\und{a}}\in G^{\uom}(\und{a})$ and $\hat{\tb}$ does not contain $\und{a}$ we obtain that $k^{\hat{\tb},\uom}(c^{\uom}_{\und{a}})$ does not depend on $\uom'$. We conclude that $F^{\hat{\tb},\uom}_{\und{a}}$ is independent of $\uom'$. 
\end{itemize}
From the previous analysis we get that $\overline{t}_i(\uom)$ only depends on $\om_i$ and not in $\om_j$ for $j\neq i$ or $\uom'$. Moreover, we have
\[c^{\overline{\tb},\uom}_{\und{b}^{i}}=\left(F^{\hat{\tb},(\und{0},\uom'')}_{\und{a}}\right)^{-1}\circ k^{\hat{\tb},(\und{0},\uom'')}(x_0+c_5\rho^{1+1/2k}\om_i),\]
taking derivative respect to $\om_i$ we get
\[D_{\om_i}c^{\overline{\tb},\uom}_{\und{b}^{i}}=\left(DF^{\hat{\tb},(\und{0},\uom'')}_{\und{a}}\right)^{-1}\cdot Dk^{\hat{\tb},(\und{0},\uom'')}(x_0+c_5\rho^{1+1/2k}\om_i)\cdot c_5\rho^{1+1/2k}.\]
Observe that since $x_0+c_5\rho^{1+1/2k}\om_i=f^{\uom}_{\und{a}\und{b}^i}(c^{\uom}_b)$ belongs to the Cantor set $K^{\uom}$ then the matrix $Dk^{\hat{\tb},(\und{0},\uom'')}(x_0+c_5\rho^{1+1/2k}\om_i)$ is conformal. Moreover
\[D_{\om_i}\overline{t}_i(\uom)= \left(DF^{\tb'}_{\und{b}^{\prime i}}\right)^{-1} \cdot s D_{\om_i}c^{\overline{\tb},\uom}_{\und{b}^{i}},\]
therefore $\om_i\to \overline{t}_i(\uom)$ defines a holomorphic function, which we will denote by $\overline{t}_i(\om_i)$. From the previous formulas it is not difficult to see that $\|D_{\om_i}\overline{t}_i(\om_i)\|\approx c_5$.

Now we show that $\overline{t}_i(0)=\overline{t}_i(\und{0},\uom'')$ is uniformly bounded. We already know that $|\tilde{t}_i|\leq 2(1+e^R)$, using this inequality together with $d(\overline{\tb},\tilde{\tb})<3\rho^{5/2}$, $d(\tb',\tilde{\tb}')<2\rho^{5/2}$, $|t-\tilde{t}|<\rho$, $|s-\tilde{s}|<\rho$ and equation \eqref{eq:formularenorma} one gets that $|\overline{t}_i(0)|\lesssim 1$. Therefore, choosing $c_5$ big enough one guarantees that the image of $\overline{t}_i(\om_i)$ contains $L^{-1}_{\hat{\uom}}(\tb^i,\tb^{\prime i},s_i(\hat{\uom}))$ (see lemma \ref{lem:cover}). Having chosen $c_5$ this way, the fact that $\|D_{\om_i}\overline{t}_i(\om_i)\|\approx c_5$ and proposition \ref{prop:bigL} gives that there is a constant $c_7'$ such that
\[Leb\left(\{\om_i:\, \overline{t}_i(\uom)\in L^{-1}_{\hat{\uom}}(\tb^i,\tb^{\prime i},s_i(\hat{\uom}))  \}\right)\geq c_7'.\]
\end{proof}
In the previous proof we used the following lemma, it is proven using standard arguments in analysis, for completeness we present the proof.
\begin{lemma}\label{lem:cover}
Let $f:B(0,1)\subset \R^n\to \R^n$ be $C^1$ on $B(0,1)$ and such that $m(Df(x))>c_5$ for all $x\in B(0,1)$. Given $r$, if $c_5$ is big enough, depending on $r$, then $B(0,r)\subset f(B(0,1))$.
\end{lemma}
\begin{proof}
Redefining $f$ as $f-f(0)$ and taking $r$ as $r+f(0)$, one can assume $f(0)=0$. Since $f$ is $C^1$ and $m(Df(x))>c_5>0$ then $f$ is a $C^1$ local diffeomorphism, hence its image $f(B(0,1))$ is open. Let $h$ be a point in $\R^n\setminus f(B(0,1/2))$ which is closest to $0$, thus $\lambda h\in f(B(0,1/2))$ for all $0\leq\lambda<1$. We can cover $\overline{B}(0,1/2)$ by a finite number of open sets such that in each one of these sets $f$ is a diffeomorphism onto its image. We can use this cover to lift curves in $\overline{f(B(0,1/2))}$ (In fact, one can prove that $f|_{B(0,1/2)}:B(0,1/2)\to f(B(0,1/2))$ is a covering map). Consider $\alpha:[0,1]\to \overline{f(B(0,1/2))}$ given by $\alpha(t)=th$. Denote by $\beta$ a lifting of $\alpha$, i.e. $f\circ \beta=\alpha$, such that $\beta(0)=0$. The curve $\beta$ is $C^1$ and $\beta(1)\notin B(0,1/2)$, otherwise this would contradict the choice of $h$. Therefore
\[|h|=length(\alpha)= \int_0^1 \left|\frac{d}{dt}(f\circ \beta)\right| dt\geq \int_0^1 m(Df(\beta(t)))\left|\frac{d}{dt} \beta\right| dt\geq c_5 \int_0^1\left|\frac{d}{dt} \beta \right| dt\geq c_5/2.\]
If $c_5>2r$ then $|h|>r$ and the desired result follows.
\end{proof}

\begin{proof}[Proof of lemma \ref{lem:tib}:]
Let $\overline{t}_i(\om_i)\in L^{-1}_{\hat{\uom}}(\tb^i,\tb^{\prime i},s_i(\hat{\uom}))$, then there exists pairs $(\und{d}^j,\und{d}^{\prime j})$, $1\leq j\leq N$, such that if we write
\[T^{\hat{\uom}(i)}_{\und{d}^j}T'_{\und{d}^{\prime j}} (\tb^i,\tb^{\prime i},s_i(\hat{\uom}),\overline{t}_i(\om_i))=(\tb^i\und{d}^j,\tb^{\prime i}\und{d}^{\prime j},\hat{s}_{(j)},\hat{t}_{(j)}),\]
where $\hat{\uom}(i)$ is obtained from $\hat{\uom}$ by setting the value $0$ in the coordinates belonging to $\Sigma_1(\tb^i)$, then
\begin{itemize}
\item[(i)] $\und{d}^1$,...,$\und{d}^N$ are pairwise disjoint.
\item[(ii)] $\tb^i\und{d}^j \in \Sigma^-_{nr}$.
\item[(iii)'] $|s^*-\hat{s}_{(j)}|<\frac{3}{4}c_4\rho^{1/2}$ implies $(\tb^i\und{d}^j,\tb^{\prime i}\und{d}^{\prime j},s^*)\in \tilde\LL$.
\item[(iv)'] $|\hat{t}_{(j)}|<1+e^R$.
\end{itemize}
We will prove that $t_i(\uom)\in L^{0}_{\uom}(\tb^i,\tb^{\prime i},s_i(\uom))$. To do this, we will prove that if write
\[T^{\uom(i)}_{\und{d}^j}T'_{\und{d}^{\prime j}} (\tb^i,\tb^{\prime i},s_i(\uom),t_i(\uom))=(\tb^i\und{d}^j,\tb^{\prime i}\und{d}^{\prime j},s_{(j)},t_{(j)}),\]
where $\uom(i)$ is obtained from $\uom$ by setting the value $0$ in the coordinates belonging to $\Sigma_1(\tb^i)$, then
\begin{itemize}
\item[(iii)] $|s^*-s_{(j)}|<\frac{2}{3}c_4\rho^{1/2}$ implies $(\tb^i\und{d}^j,\tb^{\prime i}\und{d}^{\prime j},s^*)\in \tilde\LL$.
\item[(iv)] $|t_{(j)}|<2(1+e^R)$.
\end{itemize}
First, notice that by lemma \ref{lem:plg} we have that $|s_i(\uom)-s_i(\hat{\uom})|\lesssim \rho^{1-1/2k}$. To obtain $s_{(j)}$ and $\hat{s}_{(j)}$ we applied the same renormalizations, with the same limit geometries but with different values of the perturbation parameter, therefore $|s_{(j)}-\hat{s}_{(j)}|\lesssim \rho^{1-1/2k}=o(\rho^{1/2})$. We conclude that, taking $\rho$ small enough, item (iii)' implies item (iii).

Now, to prove that (iv)' implies (iv) we need stronger estimates. We have
\begin{align*}
\hat{t}_{(j)}&= \left(DF^{\tb^{\prime i}}_{\und{d}^{\prime j}}\right)^{-1} \cdot (\bar{t}_i(\om_i)+s_i(\hat{\uom})c^{\tb^i,\hat{\uom}(i)}_{\und{d}^{j}}-c^{\tb^{\prime i}}_{\und{d}^{\prime j}}),\\
t_{(j)}&= \left(DF^{\tb^{\prime i}}_{\und{d}^{\prime j}}\right)^{-1} \cdot (t_i(\uom)+s_i(\uom)c^{\tb^i,\uom(i)}_{\und{d}^{j}}-c^{\tb^{\prime i}}_{\und{d}^{\prime j}}).
\end{align*}
We will compare the corresponding terms:
\begin{itemize}
\item Using that $d(\tb,\overline{\tb})<\rho^{5/2}$, one has
\[|\overline{t}_i(\om_i)-t_i(\uom)|= \left|\left(DF^{\tb'}_{\und{b}^{\prime i}} \right)^{-1} \cdot s \left[c^{\overline{\tb},\uom}_{\und{b}^i}-c^{\tb,\uom}_{\und{b}^i}\right]\right|\lesssim \rho^{3/2}.\]
\item Notice that $\uom(i)$ and $\hat\uom(i)$ only differ at their values in the coordinates $(\uom_1,...,\uom_N)$. Using that $\tb^i\und{d}^j\in \Sigma_{nr}^-$ one sees that $\und{a}$ is not contained in $\und{d}^j$. Then, from the arguments used in lemma \ref{lem:si}, we see that $c^{\uom(i)}_{\und{d}^{j}}=c^{\hat\uom(i)}_{\und{d}^{j}}$. Moreover, since $\tb^i$ ends in $\und{a}\und{b}^i$ and $\tb^i\und{d}^j\in \Sigma^-_{nr}$ then
\[f^{\uom(i)}_{\tb^i_n}(c^{\uom(i)}_{\und{d}^{j}})=f^{\hat\uom(i)}_{\tb^i_n}(c^{\hat\uom(i)}_{\und{d}^{j}})\]
for all $n$ such that $\tb^i_n$ is strictly shorter that $\und{a}\und{b}^i$. For the same reasons we also have
\[f^{\uom(i)}_{\tb^i_n}(c_{\theta_0})=f^{\hat\uom(i)}_{\tb^i_n}(c_{\theta_0})\]
for all $n$ such that $\tb^i_n$ is strictly shorter than $\und{a}\und{b}^i$. Both equalities still hold for $z$ in a neighborhood of either of the points. Thus we can use remark \ref{rem:per} and lemma \ref{lem:plg} to obtain that
\begin{align*}
|c^{\tb^i,\uom(i)}_{\und{d}^{j}}-c^{\tb^i,\hat\uom(i)}_{\und{d}^{j}}|&=|k^{\tb^i,\uom(i)}(c^{\uom(i)}_{\und{d}^{j}})-k^{\tb^i,\hat\uom(i)}(c^{\hat\uom(i)}_{\und{d}^{j}})|\\
&\lesssim diam(G^{\uom(i)}(\und{a}\und{b}^i))\lesssim diam(G^{\uom(i)}(\und{a}))diam(G^{\uom(i)}(\und{b}^i))\\
&\lesssim diam(G(\und{a}))diam(G(\und{b}^i))\approx \rho^{1+1/2k}.
\end{align*}
\item Notice that $\uom$ and $\hat\uom$ only differ in the values associated to the coordinates $(\om_1,...,\om_N)$. We will use $\overline{\tb}$ again. Consider
\[T^{\uom}_{\und{b}^i}T'_{\und{b}^{\prime i}} (\overline\tb,\tb',s)=(\overline{\tb}\und{b}^i,\tb'\und{b}^{\prime i},\overline{s}_i(\uom))\]
and
\[T^{\hat\uom}_{\und{b}^i}T'_{\und{b}^{\prime i}} (\overline\tb,\tb',s)=(\overline{\tb}\und{b}^i,\tb'\und{b}^{\prime i},\overline{s}_i(\hat\uom)).\]
Since $d(\tb,\overline\tb)<\rho^{5/2}$ one gets $|s_i(\uom)-\overline{s}_i(\uom)|\lesssim \rho^{5/2}$ and $|s_i(\hat\uom)-\overline{s}_i(\hat\uom)|\lesssim \rho^{5/2}$. Thus we only need to estimate $|\overline{s}_i(\uom)-\overline{s}_i(\hat\uom)|$. Remember $\hat{\tb}$ which verified $\overline\tb=\hat\tb\und{a}$, write
\[k^{\overline{\tb},\uom}=(F^{\hat\tb,\uom}_{\und{a}})^{-1}\circ k^{\hat\tb,\uom}\circ f^{\uom}_{\und{a}},\]
then
\[Dk^{\overline{\tb},\uom}=(DF^{\hat\tb,\uom}_{\und{a}})^{-1}\cdot (Dk^{\hat\tb,\uom}\circ f^{\uom}_{\und{a}})\cdot Df^{\uom}_{\und{a}}.\]
Using the analysis and notation from lemma \ref{lem:si} we see that: $F^{\hat\tb, \uom}_{\und{a}}$ does not depend on $\uom'$, $Dk^{\hat\tb,\uom}(z)=Dk^{\hat\tb,\hat\uom}(z)$ for all $z\in G^{\uom}(\und{a})\cup G^{\hat\uom}(\und{a})$ (in particular $c^{\uom}_{\und{a}\und{b}^{i}}$ and $c^{\hat\uom}_{\und{a}\und{b}^{i}}$). We also have $$c^{\hat\uom}_{\und{a}\und{b}^i}=x_0,\,\, c^{\uom}_{\und{a}\und{b}^i}=x_0+c_5\rho^{1+1/2k}\om_i$$ and $Df^{\uom}_{\und{a}}(c^{\uom}_{\und{b}^i})=Df^{\hat\uom}_{\und{a}}(c^{\hat\uom}_{\und{b}^i})$.
Therefore
\begin{align*}
|Dk^{\overline{\tb},\uom}(c^{\uom}_{\und{b}^{i}})-Dk^{\overline{\tb},\hat\uom}(c^{\hat\uom}_{\und{b}^{i}})|&= \frac{|Df^{\uom}_{\und{a}}(c^{\uom}_{\und{b}^i})|}{|DF^{\hat\tb,\uom}_{\und{a}}|}\cdot |Dk^{\hat\tb,\uom}(x_0+c_5\rho^{1+1/2k}\om_i)-Dk^{\hat\tb,\uom}(x_0)|\\
&\lesssim \rho^{1+1/2k},
\end{align*}
Here we used that limit geometries are $C^2$ and the norm of $D^2 k^{\hat{\tb},\uom}$ can be uniformly bounded. Using that $DF^{\tb,\uom}_{\und{d}}= Dk^{\tb,\uom}(c^{\uom}_{\und{d}})\cdot Df^{\uom}_{\und{d}}(c_d)$ for any word $\und{d} \in \Sigma^{fin}$ that ends with the letter $d$, one can conclude that
\begin{align*}
|DF^{\overline\tb,\uom}_{\und{b}^i}-DF^{\overline\tb,\hat\uom}_{\und{b}^i}|&=|Df^{\uom}_{\und{b}^i}(c_{b^i})|\cdot|Dk^{\overline\tb,\uom}(c^{\uom}_{\und{b}^{i}})-Dk^{\overline\tb,\hat\uom}(c^{\hat\uom}_{\und{b}^{i}})|\lesssim \rho^{2+1/2k}.
\end{align*}
And from this
\[\left|\frac{DF^{\overline\tb,\uom}_{\und{b}^i}}{DF^{\tb'}_{\und{b}^{\prime i}}}-\frac{DF^{\overline\tb,\hat\uom}_{\und{b}^i}}{DF^{\tb'}_{\und{b}^{\prime i}}}\right|\lesssim \rho^{1+1/2k},\]
therefore $|\overline{s}_i(\uom)-\overline{s}_i(\hat{\uom})|\lesssim \rho^{1+1/2k}$ and $|s_i(\uom)-s_i(\hat{\uom})|\lesssim \rho^{1+1/2k}$.
\end{itemize}
From the previous estimates we conclude that $|\hat{t}_{(j)}-t_{(j)}|\lesssim \rho^{1/2k}$, then if $\rho$ is small enough (iv)' implies (iv).
\end{proof}
\bibliographystyle{unsrt}
\bibliography{bibliography}

\end{document}